\documentclass[a4paper]{amsart}
\usepackage{color}
\usepackage{amssymb,amsmath,amsthm,amstext,amsfonts}

\usepackage{psfrag}
\usepackage{url}
\usepackage{amsfonts}
\usepackage{amsmath}
\usepackage{mathrsfs}
\usepackage{graphicx}
\usepackage{amsmath,amsthm,amssymb,amscd}
\usepackage{enumerate}
\usepackage[colorlinks,linkcolor=black,anchorcolor=black,citecolor=black,hyperindex=true,CJKbookmarks=true]{hyperref}
\usepackage[T1]{fontenc}
\numberwithin{equation}{section}

\theoremstyle{plain}

\newtheorem{Thm}{Theorem}[section]
\newtheorem{Lem}[Thm]{Lemma}
\newtheorem{Prop}[Thm]{Proposition}
\newtheorem{Cor}[Thm]{Corollary}
\theoremstyle{remark}
\newtheorem{Def}[Thm] {Definition}
\newtheorem{Rem}[Thm] {Remark}

\theoremstyle{remark}
\theoremstyle{definition}

\newcommand{\eps}{\varepsilon}

\newcommand{\N}{\mathbb{N}}

\newcommand{\B}{\mathcal{B}}

\begin{document}

\title{ Topological entropy and Hausdorff dimension of shrinking target sets }
\author{Xiaobo Hou, Xueting Tian and Yiwei Zhang}

\address{Xiaobo Hou, School of Mathematical Sciences,  Fudan University\\Shanghai 200433, People's Republic of China}
\email{20110180003@fudan.edu.cn}

\address{Xueting Tian, School of Mathematical Sciences,  Fudan University\\Shanghai 200433, People's Republic of China}
\email{xuetingtian@fudan.edu.cn}

\address{Yiwei Zhang,  Department of Mathematics \& SUSTech International Center for Mathematics, Southern University of Science and Technology, Shenzhen, Guangdong 518055, China}
\email{zhangyw@sustech.edu.cn}

\begin{abstract}
In this paper, we study the topological entropy and the Hausdorff dimension of a shrinking target set. We give lower and upper bounds of topological entropy and Hausdorff dimension for dynamical systems with exponential specification property and Lipschitz continuity for maps and homeomorphisms. It generally  applies to uniformly hyperbolic systems, expanding systems, and some symbolic dynamics. We show that lower and upper bounds coincide for both topological entropy and Hausdorff dimension when the systems are hyperbolic automorphisms of torus induced from a matrix with only two different eigenvalues, expanding endomorphism of the torus induced from a matrix with only one eigenvalue or some symbolic systems including one or two-sided shifts of finite type and sofic shifts.
\end{abstract}

\maketitle

\section{Introduction}
Let $(X,f)$ be a non-degenerate dynamical system, where $(X,d)$ be a nondegenerate $($i.e.,
with at least two points$)$ compact metric space, and $f:X \rightarrow X$ be a continuous self map. For each dynamical system, let $\varphi$ be a positive real-valued function defined on $\mathbb{N}$, $\mathcal{Z}=\{z_i\}_{i=0}^{\infty}\subset X$  and  $S=\{s_i\}_{i=0}^{\infty}\subset \mathbb{N},$ we define the \emph{$S$-shrinking target set} by
\begin{equation}\label{equ:Sshrinkingtarget}
\mathfrak{S}(f,\varphi,\mathcal{Z},S):=\{x\in X: d(f^n(x),z_n)<\varphi(n)\text{ for  infinitely many }n\in S\}.
\end{equation}
When $S=\mathbb{N}$, the set $\mathfrak{S}(f,\varphi,\mathcal{Z},S)$ in \eqref{equ:Sshrinkingtarget} reduces to the classical shrinking target set. \eqref{equ:Sshrinkingtarget} means if at time $n$, one has a ball $B(z_{n},\varphi(n))$ centered at $z_{n}$ with radius $\varphi(n)\to0$ (as $n\to\infty$), then what kind of properties does the set of point $x$ have, whose image $f^{n}(x)$ are in ball $B(z_{n},\varphi(n))$ for infinitely many times $n\in S$.

The studies of shrinking target sets have been back to the seminal works from Hill and Velani [17], where they call the points in the set of well approximatable points in analogy with the classical theory of metric Diophantine approximation [11, 27], particularly the Jarn\'{i}k--Besicovitch Theorem [3, 19].

The complexity (e.g., entropies and dimensions) of the shrinking target sets have been rigorously investigated for several important classes of dynamical systems, including  expanding rational maps \cite{Hill-Velani,Hill-Velani-2}, parabolic rational maps \cite{Stratmann-Urba}, $\beta$-expansions \cite{Tan-Wang}, finite Klenian groups \cite{DMPV,Hill-Velani-3}, Markov expanding systems \cite{Urbanski,LWWX,Reeve}, matrix transformations on the torus \cite{Hill-Velani-4,Li-Liao-Velani-Zorin}, irrational rotations \cite{BD,BHKV,Fan-Wu}. Usually, various versions of mass transfer principles are developed (see e.g., \cite{Beresnevich-Velani,Koi-Rams,Wang-Wu2,Li-Liao-Velani-Zorin}) to get the lower bound of the Hausdorff dimension of the shrinking target sets. We recommend the readers to the survey paper by Wang and Wu \cite{Wang-Wu} for more information. Recently, the shrinking target sets associated with some two-dimensional hyperbolic linear Anosov systems are also investigated 
in \cite{HPWZ} through results on Riesz energies (e.g., \cite{Persson}).

On the other hand, the studies on the shrinking target sets are the basis and inspire many related works, such as the so-called quantitative
recurrence properties \cite{Boshernitzan}, dynamical Borel-Cantelli lemma \cite{Che-kle}, uniform Diophantine approximation \cite{BL},
recurrence time \cite{Bar-Sau}, waiting time \cite{Galatolo}, etc.

In this paper, we consider both the topological entropy and the Hausdorff dimension of the shrinking target sets in the context of \textbf{general} topological dynamical systems with the exponential non-uniform specification property. In contrast to the studies on the shrinking target sets associated with special classes of dynamical systems, we will use the exponential property derived from hyperbolicity to trace orbit segments exponentially. This method is available for both expanding and hyperbolic systems, and it allows us to find points that visit the given balls with an exponentially shrinking radius.  

Using Lipschitz properties, we give upper bounds on the topological entropy by finding dynamical balls covering the shrinking target set. The hyperbolicity allows us to put dynamical balls in balls. Thus, we can obtain upper bounds on the Hausdorff dimension. The lower bounds are obtained by constructing a Cantor subset with large enough topological entropy and Hausdorff dimension using the so-called `exponential non-uniform specification property'. Therefore, our method is independent of the mass transfer principle.

Our main results enable us to give lower and upper bounds of topological entropy and Hausdorff dimension for general uniformly hyperbolic systems, expanding systems, and some symbolic systems. We show that lower and upper bounds coincide for both topological entropy and Hausdorff dimension when the systems are hyperbolic automorphisms of torus induced from a matrix with only two different eigenvalues, expanding endomorphism of the torus induced from a matrix with only one eigenvalue or some symbolic systems including one or two-sided shifts of finite type and sofic shifts. In particular, we can deal with any dimensional hyperbolic matrix with only two different eigenvalues, which is an improvement of \cite[Theorem 1]{HPWZ}.


\subsection{Statement of main results}
Our main results contain the results on both lower and upper bounds for topological entropy and Hausdorff dimension of the shrinking target sets, which will be stated in Subsections \ref{subsection_lowerbounds} and \ref{subsection_upperbounds}, respectively.

\subsubsection{Lower bounds for topological entropy and Hausdorff dimension}\label{subsection_lowerbounds}
	Let us first concentrate on the lower bounds for topological entropy and Hausdorff dimension on $\mathfrak{S}(f,\varphi,\mathcal{Z},S)$. Before we state the main results, we introduce some assumptions:
	
	(A) $(X,f)$  has the exponential non-uniform specification property  with respect to the exponent $(\lambda_1,\lambda_2)$ at some scale $\varepsilon_0>0$ (see Subsection \ref{sec-specfication} for precise definition).
	
	(B1) \emph{$L$-Lipschitz:} $f:X\to X$ is a Lipschitz map with $\sup\{\frac{d(f(x),f(y))}{d(x,y)}:x\neq y\in X\}\leq L<\infty.$
	
	(B2) \emph{$(L_1,L_2)$-bi-Lipschitz:} $f:X\to X$ is a bi-Lipschitz homeomorphism with
	$$ \sup\{\frac{d(f^{-1}(x),f^{-1}(y))}{d(x,y)}:x\neq y\in X\}\leq L_1<\infty,\ \sup\{\frac{d(f(x),f(y))}{d(x,y)}:x\neq y\in X\}\leq L_2<\infty.$$
	
	We consider the intersection of countable many shrinking target sets. For each $j\geq 1,$ let
	\begin{enumerate}
		\item[$\cdot$] $\varphi_j$ be a map from $\mathbb{N}$ to $(0,1],$
		\item[$\cdot$] $\mathcal{Z}_j=\{z_j^i\}_{i=0}^{\infty}\subset X$ be a sequence of points,
		\item[$\cdot$] $S_j=\{s_j^i\}_{i=0}^\infty\subset \mathbb{N}$ be a sequence of positive integers with $\lim\limits_{i\to\infty}s_j^i=\infty.$
	\end{enumerate}
	Define $$\overline{\tau}_j:=\limsup\limits_{n\to\infty}-\frac{\ln\varphi_j(n)}{n},\ \overline{\tau}:=\sup\limits_{j\geq 1}\overline{\tau}_j,$$
	$$\underline{\tau}_j:=\liminf\limits_{n\to\infty}-\frac{\ln\varphi_j(n)}{n},\ \underline{\tau}:=\sup\limits_{j\geq 1}\underline{\tau}_j.$$
\begin{Thm}\label{theorem-shrinking-1}
	Suppose that $(X,f)$ is a dynamical system satisfying (A). If $\overline{\tau}<\lambda_1$, then
    the set
    $\cap_{j=1}^{\infty}\mathfrak{S}(f,\varphi_j,\mathcal{Z}_j,S_j)$ is non-empty and
    \begin{equation}\label{equa-thm-1}
    	h_{top}(f,\cap_{j=1}^{\infty}\mathfrak{S}(f,\varphi_j,\mathcal{Z}_j,S_j))\geq \frac{\lambda_1\lambda_2-\lambda_2\overline{\tau} }{\lambda_1\lambda_2+\lambda_1\overline{\tau}}h_{top}(f,X).
    \end{equation}
    Moreover,
    \begin{enumerate}
    	\item[(1)] if further $f:X\to X$ is  a $L$-Lipschitz map, then
    	\begin{equation}\label{equa-thm-2}
    		\operatorname{dim}_H \cap_{j=1}^{\infty}\mathfrak{S}(f,\varphi_j,\mathcal{Z}_j,S_j)\geq \frac{1}{\ln L}\frac{\lambda_1\lambda_2-\lambda_2\overline{\tau} }{\lambda_1\lambda_2+\lambda_1\overline{\tau}}h_{top}(f,X).
    	\end{equation}
    	\item[(2)] if further $f:X\to X$ is a $(L_1,L_2)$-bi-Lipschitz homeomorphism, then
    	\begin{equation}\label{equa-thm-3}
    		\operatorname{dim}_H \cap_{j=1}^{\infty}\mathfrak{S}(f,\varphi_j,\mathcal{Z}_j,S_j)\geq (\frac{1}{\ln L_1}+\frac{1}{\ln L_2}\frac{\lambda_1\lambda_2-\lambda_2\overline{\tau} }{\lambda_1\lambda_2+\lambda_1\overline{\tau}})h_{top}(f,X).
    	\end{equation}
    \end{enumerate}
\end{Thm}
\begin{Rem}\label{Rem-under}
	 $\overline{\tau}$ in the results of  Theorem \ref{theorem-shrinking-1} can be repaced by $\underline{\tau}$ when $S_j=\mathbb{N}$ for any $j\geq 1.$    In fact we can choose $\tilde{S}_j=\{\tilde{s}_j^i\}_{i=1}^{\infty}\subset \mathbb{N}$ such that $\lim\limits_{i\to \infty}-\frac{\ln\varphi(\tilde{s}_j^i)}{\tilde{s}_j^i}=\underline{\tau}_j,$ and define  $\tilde{\varphi}_j$ on $\mathbb{N}$  as
	$$
	\tilde{\varphi}_j(n):=\left\{\begin{array}{ll}
		\varphi_j(n), & \text { if } n\in \tilde{S}_j,\\
		1, & \text { if } n\in \mathbb{N}\setminus \tilde{S}_j.
	\end{array}\right.
	$$
	Then $\overline{\tau}_j(\tilde{\varphi}_j):=\limsup\limits_{n\to\infty}-\frac{\ln\tilde{\varphi}_j(n)}{n}=\lim\limits_{i\to \infty}-\frac{\ln\varphi(\tilde{s}_j^i)}{\tilde{s}_j^i}=\underline{\tau}_j,$ and  $\mathfrak{S}(f,\tilde{\varphi}_j,\mathcal{Z}_j,\tilde{S}_j)\subset\mathfrak{S}(f,\varphi_j,\mathcal{Z}_j,S_j).$
\end{Rem}

\begin{Rem}
	Under weaker assumptions, the lower bounds in Theorem \ref{theorem-shrinking-1}  will be smaller.
	In the definition of exponential non-uniform specification property (Definition \ref{def-exp-nonuni-speci}), it's required that $$
	\limsup _{n \rightarrow \infty} \frac{1}{n} p(n, \varepsilon)=0.
	$$
	When the exponential  non-uniform  specification property is weaker in the following sense\footnote{Readers can refer to \cite{LTY} for examples with weak non-uniform specification property.}:
	$$
	\chi:=\limsup _{n \rightarrow \infty} \frac{1}{n} p(n, \varepsilon)>0,$$
	the lower bounds of topological entropy and Hausdorff dimension  obtained  in Theorem \ref{theorem-shrinking-1}  will be smaller. Precisely, proceeding with the same argument as the proof of Theorem \ref{theorem-shrinking-1}, one can check that under the assumption of weak non-uniform  specification property, the results of Theorem \ref{theorem-shrinking-1} holds if
	$\frac{\lambda_1\lambda_2-\lambda_2\overline{\tau} }{\lambda_1\lambda_2+\lambda_1\overline{\tau}}$ is replaced by $$\frac{1}{1+\chi}\frac{\lambda_1\lambda_2-\lambda_2\overline{\tau} }{\lambda_1\lambda_2+\lambda_1\overline{\tau}+(\lambda_1+\lambda_2)\chi\overline{\tau}}.$$
\end{Rem}

The assumption $(A)$ requires that the dynamical system $(X,f)$ is topologically mixing. When  $(X,f)$ is topologically transitive, one can always find a mixing subset for some iterate of $f$ utilizing the periodic decompositions theorem. For such systems, we also obtain a result on the shrinking target set. 
Suppose that $(X,f)$ is a dynamical system with a non-empty compact set $Y\subseteq X$ and a positive integer $N$ such that $f^{N}(Y)=Y$ and  $X=\bigcup_{l=0}^{N-1}f^{l}(Y)$.
For a sequence of points $\mathcal{Z}=\{z_i\}_{i=0}^{\infty}\subset X$  and  a sequence of positive integers $S=\{s_i\}_{i=0}^{\infty}$ with $\lim\limits_{i\to\infty}s_i=\infty$, we define the index set of $(\mathcal{Z},S)$ by $$\mathrm{Ind}(\mathcal{Z},S,Y):=\{(I_1,I_2)\in\{0,1,\dots,N-1\}^2: \text{ there are infinitely many }i \text{ s.t. } z_{s_i}\in f^{I_1}(Y) \text{ and }N|(s_i-I_2)\}.$$
And we define $\mathrm{Ind}'(\mathcal{Z},S,Y):=\{I_1-I_2:(I_1,I_2)\in \mathrm{Ind}(\mathcal{Z},S,Y)\}.$
Using the pigeon-hole principle, one has $\mathrm{Ind}(\mathcal{Z},S,Y)\neq\emptyset$ and thus $\mathrm{Ind}'(\mathcal{Z},S,Y)\neq\emptyset.$

Now, we introduce the following assumption:

($\tilde{A}$) $(X,f)$ is a dynamical system with a non-empty compact set $Y\subseteq X$ and a positive integer $N$ such that $f^{N}(Y)=Y$, $X=\bigcup_{l=0}^{N-1}f^{l}(Y)$ and  $(Y,f^N)$  has the exponential non-uniform specification property  with respect to the exponent $(N\lambda_1,N\lambda_2)$ at some scale $\varepsilon_0>0$.

\begin{Cor}\label{theorem-shrinking-2}
	Suppose that $(X,f)$ is a dynamical system satisfying ($\tilde{A}$) and $f$ is Lipschitz. If $\overline{\tau}<\lambda_1$ and $\cap_{j=1}^{\infty}\mathrm{Ind}'(\mathcal{Z}_j,S_j,Y)\neq\emptyset$, then
	the set
	$\cap_{j=1}^{\infty}\mathfrak{S}(f,\varphi_j,\mathcal{Z}_j,S_j)$ is non-empty and
	\begin{equation}\label{equa-thm-4}
		h_{top}(f,\cap_{j=1}^{\infty}\mathfrak{S}(f,\varphi_j,\mathcal{Z}_j,S_j))\geq \frac{\lambda_1\lambda_2-\lambda_2\overline{\tau} }{\lambda_1\lambda_2+\lambda_1\overline{\tau}}h_{top}(f,X).
	\end{equation}
	Moreover,
	\begin{enumerate}
		\item[(1)] if further $f:X\to X$ is  a $L$-Lipschitz map, then
		\begin{equation}\label{equa-thm-5}
			\operatorname{dim}_H \cap_{j=1}^{\infty}\mathfrak{S}(f,\varphi_j,\mathcal{Z}_j,S_j)\geq \frac{1}{\ln L}\frac{\lambda_1\lambda_2-\lambda_2\overline{\tau} }{\lambda_1\lambda_2+\lambda_1\overline{\tau}}h_{top}(f,X).
		\end{equation}
		\item[(2)] if further $f:X\to X$ is a $(L_1,L_2)$-bi-Lipschitz homeomorphism, then
		\begin{equation}\label{equa-thm-6}
			\operatorname{dim}_H \cap_{j=1}^{\infty}\mathfrak{S}(f,\varphi_j,\mathcal{Z}_j,S_j)\geq (\frac{1}{\ln L_1}+\frac{1}{\ln L_2}\frac{\lambda_1\lambda_2-\lambda_2\overline{\tau} }{\lambda_1\lambda_2+\lambda_1\overline{\tau}})h_{top}(f,X).
		\end{equation}
	\end{enumerate}
\end{Cor}
\begin{Rem}
	The condition $\cap_{j=1}^{\infty}\mathrm{Ind}'(\mathcal{Z}_j,S_j,Y)\neq\emptyset$ in Corollary \ref{theorem-shrinking-2} can not be removed.  In fact, when $\cap_{j=1}^{\infty}\mathrm{Ind}'(\mathcal{Z}_j,S_j,Y)=\emptyset,$ the set  $\cap_{j=1}^{\infty}\mathfrak{S}(f,\varphi_j,\mathcal{Z}_j,S_j)$ can be empty. For example, if $f^{2}(Y)=Y$, $X=Y\cup f(Y)$ and $Y\cap f(Y)=\emptyset,$ $\mathcal{Z}_1 \subset Y,$ $\mathcal{Z}_2\subset f(Y)$, $S_1=S_2=2\mathbb{N}$ where  $l\mathbb{N}:=\{lm:m\in\mathbb{N}\},$ and $0<\sup_{n\in\mathbb{N}}\varphi_1(n),\sup_{n\in\mathbb{N}}\varphi_2(n)<\inf_{y_1\in Y,y_2\in f(Y)}d(y_1,y_2),$ then one has $\mathrm{Ind}(\mathcal{Z}_1,S_1,Y)=\{(0,0)\},$ $\mathrm{Ind}(\mathcal{Z}_2,S_2,Y)=\{(1,0)\}$ and thus $\cap_{j=1}^{2}\mathrm{Ind}'(\mathcal{Z}_j,S_j,Y)=\emptyset.$ Note that for any $x\in\mathfrak{S}(f,\varphi_1,\mathcal{Z}_1,S_1)$, one has $f^{2m_0}(x)\in Y$ for some $m_0\in\mathbb{N}.$ It follows that $f^{2m}(x)\in Y$ for any $m\in\mathbb{N}$ and thus $x\not\in\mathfrak{S}(f,\varphi_2,\mathcal{Z}_2,S_2).$ So $\cap_{j=1}^{2}\mathfrak{S}(f,\varphi_j,\mathcal{Z}_j,S_j)=\emptyset.$
\end{Rem}

Now, we apply Theorem \ref{theorem-shrinking-1} to the covering problem. Let $\varphi$ be a positive function defined on $\mathbb{N}.$ For any $x\in X,$ set
$$\mathfrak{C}(f,\varphi,x):=\{y\in X:  d(f^n(x),y)<\varphi(n)\text{ for infinitely many }n\in \mathbb{N}\}.$$
Using Theorem \ref{theorem-shrinking-1}, we show that there is $x_0$ such that $\mathfrak{C}(f,\varphi,x_0)$ is dense in $X$, and we give  lower bounds of topological entropy and Hausdorff dimension of the set of such points.
\begin{Cor}\label{Cor-covering}
	Suppose that $(X,f)$ is a dynamical system satisfying ($\tilde{A}$) and $f$ is Lipschitz.
	Let $\varphi$ be a positive function defined on $\mathbb{N}$ such that $\overline{\tau}<\lambda_1$ where $\overline{\tau}=\limsup\limits_{n\to\infty}-\frac{\ln\varphi(n)}{n}.$
	Then
	the set
	$\mathfrak{D}(f,\varphi):=\{x\in X: \mathfrak{C}(f,\varphi,x)\text{ is dense in }X\}$
	is non-empty and
	\begin{equation}
		h_{top}(f,\mathfrak{D}(f,\varphi))\geq \frac{\lambda_1\lambda_2-\lambda_2\overline{\tau} }{\lambda_1\lambda_2+\lambda_1\overline{\tau}}h_{top}(f,X).
	\end{equation}
	Moreover,
	\begin{enumerate}
		\item[(1)] if further $f:X\to X$ is  a $L$-Lipschitz map, then
		\begin{equation}
			\operatorname{dim}_H \mathfrak{D}(f,\varphi)\geq \frac{1}{\ln L}\frac{\lambda_1\lambda_2-\lambda_2\overline{\tau} }{\lambda_1\lambda_2+\lambda_1\overline{\tau}}h_{top}(f,X).
		\end{equation}
		\item[(2)] if further $f:X\to X$ is a $(L_1,L_2)$-bi-Lipschitz homeomorphism, then
		\begin{equation}
			\operatorname{dim}_H \mathfrak{D}(f,\varphi)\geq (\frac{1}{\ln L_1}+\frac{1}{\ln L_2}\frac{\lambda_1\lambda_2-\lambda_2\overline{\tau} }{\lambda_1\lambda_2+\lambda_1\overline{\tau}})h_{top}(f,X).
		\end{equation}
	\end{enumerate}
	In particular, when $N=1$, the results  hold if we replace $\overline{\tau}$ by $\underline{\tau},$
	where $\underline{\tau}:=\liminf\limits_{n\to\infty}-\frac{\ln\varphi(n)}{n}.$
\end{Cor}

\subsubsection{Upper bounds for topological entropy and Hausdorff dimension}\label{subsection_upperbounds}
Next, we turn to the upper bounds for topological entropy and Hausdorff dimension. For any $n \in \mathbb{N}$, the $d_n$-distance between $x, y \in X$ is defined as
$$d_{n,f}(x,y):=\max\{d(f^{i}(x),f^{i}(y)):0\leq i\leq n-1\}.$$
Let $x \in X, n \geq 1$ and $\varepsilon>0$, the dynamical ball $B_n(x, \varepsilon,f)$ is defined as the set
$$
B_{n}(x,\varepsilon,f):=\{y\in X:d_n(x,y)<\varepsilon\},\quad
\overline{B}_{n}(x,\varepsilon,f):=\{y\in X:d_n(x,y)\leq \varepsilon\}.
$$

Now, we introduce the following hyperbolic assumption:

(C1) \emph{$\lambda$-hyperbolic:} Let $\lambda>0.$ $f:X\to X$ is said to be $\lambda$-hyperbolic if for any $\varepsilon_h>0$ there is $0<\varepsilon\leq\varepsilon_h$  such that for any $x\in X$ and $n\in\mathbb{N},$ one has $B_{n}(x,\varepsilon,f)\subset B_n(x,\varepsilon_h,\lambda,f),$ where $$B_n(x,\varepsilon_h,\lambda,f):=\{y\in X:d(f^{i}(x),f^i(y))<\varepsilon_h e^{-(n-1-i)\lambda} \text{ for any } 0\leq i\leq n-1\}.$$

(C2) \emph{$(\lambda_1,\lambda_2)$-hyperbolic:} Let $\lambda_1,\lambda_2>0.$ $f:X\to X$ is said to be $(\lambda_1,\lambda_2)$-hyperbolic if for any $\varepsilon_h>0$ there is $0<\varepsilon\leq\varepsilon_h$ such that for any $x\in X$ and $n\in\mathbb{N},$  one has $B_{n}(x,\varepsilon,f)\subset B_n(x,\varepsilon_h,\lambda_1,\lambda_2,f),$ where $$B_{n}(x,\varepsilon_h,\lambda_1,\lambda_2,f):=\{y\in X:d(f^{i}(x),f^i(y))< \varepsilon_h e^{-\min\{i\lambda_1,(n-1-i)\lambda_2\} }\text{ for any }0\leq i\leq n-1\}.$$
\begin{Rem}
	(1) When $\lambda_1=+\infty,$ one has $B_{n}(x,\varepsilon,\lambda_1,\lambda_2,f)=B_n(x,\varepsilon,\lambda_2,f)$. Thus $f:X\to X$ is $\lambda$-hyperbolic if and only if  $f:X\to X$ is $(+\infty,\lambda)$-hyperbolic.
	
	(2) Under the assumption $(\lambda_1,\lambda_2)$-hyperbolic, the exponential (non-uniform) specification property is equivalent to the classic (non-uniform) specification property. (See Subsection \ref{sec-specfication} and Figure \ref{Fig-1})
\end{Rem}

\begin{Thm}\label{Thm-upper-bound}
	Suppose that $(X,f)$ is a dynamical system. Let $\varphi$ be a positive function defined on $\mathbb{N},$ $\mathcal{Z}=\{z_i\}_{i=0}^{\infty}\subset X$ be a sequence of points. Denote $\underline{\tau}:=\liminf\limits_{n\to\infty}-\frac{\ln\varphi(n)}{n}.$
	\begin{enumerate}
		\item[(1)] Assume that $(X,f)$ is $L$-Lipschitz.  Then
		 \begin{equation}\label{equa-thm-10}
			h_{top}(f,\mathfrak{S}(f,\varphi,\mathcal{Z},\mathbb{N}))\leq (1-\frac{\underline{\tau}}{\ln L+\underline{\tau}})h_{top}(f,X).
		\end{equation}
		If further $f:X\to X$ is  $\lambda$-hyperbolic, then
		\begin{equation}\label{equa-thm-11}
			\operatorname{dim}_H \mathfrak{S}(f,\varphi,\mathcal{Z},\mathbb{N})\leq \frac{1}{\lambda}(1-\frac{\underline{\tau}}{\ln L+\underline{\tau}})h_{top}(f,X).
		\end{equation}
		\item[(2)] Assume that $(X,f)$ is a $(L_1,L_2)$-bi-Lipschitz homeomorphism.
		\begin{enumerate}
			\item[(2.1)]  If $\ln L_1>\underline{\tau},$ then
			\begin{equation}\label{equa-thm-12}
				h_{top}(f,\mathfrak{S}(f,\varphi,\mathcal{Z},\mathbb{N}))\leq \frac{\ln L_1\ln L_2-\underline{\tau}\ln L_2}{\ln L_1\ln L_2+\underline{\tau}\ln L_1}h_{top}(f,X).
			\end{equation}
			If further $f:X\to X$ is  $(\lambda_1,\lambda_2)$-hyperbolic, then
			\begin{equation}\label{equa-thm-13}
				\operatorname{dim}_H \mathfrak{S}(f,\varphi,\mathcal{Z},\mathbb{N})\leq (\frac{1}{\lambda_1}+\frac{1}{\lambda_2}\frac{\ln L_1\ln L_2-\underline{\tau}\ln L_2}{\ln L_1\ln L_2+\underline{\tau}\ln L_1})h_{top}(f,X).
			\end{equation}
			\item[(2.2)] If $\ln L_1= \underline{\tau},$ then
			\begin{equation}\label{equa-thm-14}
				h_{top}(f,\mathfrak{S}(f,\varphi,\mathcal{Z},\mathbb{N}))=0.
			\end{equation}
				If further $f:X\to X$ is  $(\lambda_1,\lambda_2)$-hyperbolic, then
				\begin{equation}\label{equa-thm-15}
					\operatorname{dim}_H \mathfrak{S}(f,\varphi,\mathcal{Z},\mathbb{N})\leq \frac{1}{\lambda_1}h_{top}(f,X).
				\end{equation}
			\item[(2.3)] If $\ln L_1< \underline{\tau},$ then
			\begin{equation}\label{equa-thm-16}
				h_{top}(f,\mathfrak{S}(f,\varphi,\mathcal{Z},\mathbb{N}))=0.
			\end{equation}
			If further $f:X\to X$ is  $(\lambda_1,\lambda_2)$-hyperbolic, then
			\begin{equation}\label{equa-thm-17}
				\operatorname{dim}_H \mathfrak{S}(f,\varphi,\mathcal{Z},\mathbb{N})=0.
			\end{equation}
		\end{enumerate}
	\end{enumerate}
\end{Thm}

\subsection{Applications}
Derived from Theorem \ref{theorem-shrinking-1} and \ref{theorem-shrinking-2}, we will provide several useful applications (e.g., hyperbolic sets, uniformly expanding maps, and symbolic dynamics). In some applications, the upper and lower bounds coincide, and therefore the explicit value of the Hausdoff dimension and the topological entropy of the shrinking target sets are estimated.
\subsubsection{Hyperbolic set}
Let $M$ be a compact smooth Riemann manifold and $f:M\to M$ be a diffeomorphism. A compact $f$-invariant set $\Lambda\subset M$ is said to be \emph{uniformly $(\lambda_1,\lambda_2)$-hyperbolic} where $0< \lambda_1,\lambda_2<1$, if for any $x\in \Lambda$ there is a splitting of the tangent space $T_{x}M=E^{s}(x)\oplus E^{u}(x)$ which is preserved by the differential $Df$ of $f$:
\begin{equation*}
	Df(E^{s}(x))=E^{s}(f(x)),\ Df(E^{u}(x))=E^{u}(f(x)),
\end{equation*}
and 
\begin{equation*}
	|Df(v)|\leq \lambda_1|v|,\ \forall x\in \Lambda,\ v\in E^{s}(x),
\end{equation*}
\begin{equation*}
	|Df^{-1}(v)|\leq \lambda_2|v|,\ \forall x\in \Lambda,\ v\in E^{u}(x).
\end{equation*}

$\Lambda$ is said to be \emph{locally maximal} if there exists an open neighborhood $U$ of $\Lambda$ such that $\Lambda=\bigcap_{n \in \mathbb{Z}} f^n U$,
\begin{Prop}\label{Prop-hyper-spec}
	Let $\Lambda\subset M$ be a transitive locally maximal uniformly $(\lambda_1,\lambda_2)$-hyperbolic set for  a diffeomorphism $f.$ Then there is a non-empty compact set $Y\subseteq \Lambda$ and a positive integer $N$ such that $f^{N}(Y)=Y$, $\Lambda=\bigcup_{l=0}^{N-1}f^{l}(Y)$ and  $(Y,f^N)$  has the exponential specification property  with respect to the exponent $(N\ln(\lambda'_1)^{-1},N\ln(\lambda'_2)^{-1})$ at any scale  for any $\lambda_1<\lambda'_1<1$ and $\lambda_2<\lambda'_2<1$. If further $(\Lambda,f)$ is mixing, then $Y=\Lambda$ and $N=1.$
\end{Prop}
\begin{proof}
	By spectral decomposition \cite[Theorem 18.3.1]{KatHas} there is a non-empty compact set $Y\subseteq \Lambda$ and a positive integer $N$ such that $f^{N}(Y)=Y$, $\Lambda=\bigcup_{l=0}^{N-1}f^{l}(Y)$ and  $(Y,f^N)$ is mixing. By  \cite[Theorem 18.3.9]{KatHas}, $(Y,f^N)$ has the  specification property at any scale. So by Proposition \ref{Prop-hype}, Proposition \ref{Prop-N-hype} and Proposition \ref{Prop-exp+spec}, $(Y,f^N)$  has the exponential specification property  with respect to the exponent $(N\ln(\lambda'_1)^{-1},N\ln(\lambda'_2)^{-1})$ at any scale for any $\lambda_1<\lambda'_1<1$ and $\lambda_2<\lambda'_2<1$.
\end{proof}
\begin{Thm}\label{Thm-hyper}
	Let $\Lambda\subset M$ be a transitive locally maximal uniformly $(\lambda_1,\lambda_2)$-hyperbolic set for  a diffeomorphism $f.$ Denote $L_2=\sup\limits_{x\in\Lambda}\sup\limits_{v\in T_xM,v\neq0}\frac{|Df(v)|}{|v|},$ $L_1=\sup\limits_{x\in\Lambda}\sup\limits_{v\in T_xM,v\neq0}\frac{|Df^{-1}(v)|}{|v|}.$
	\begin{enumerate}
		\item[(1)] If $\overline{\tau}<-\ln\lambda_1$ and $\cap_{j=1}^{\infty}\mathrm{Ind}'(\mathcal{Z}_j,S_j,Y)\neq\emptyset$, then
		the set
		$\cap_{j=1}^{\infty}\mathfrak{S}(f,\varphi_j,\mathcal{Z}_j,S_j)$ is non-empty and  $$\frac{\ln\lambda_1^{-1}\ln\lambda_2^{-1}-\overline{\tau}\ln\lambda_2^{-1} }{\ln\lambda_1^{-1}\ln\lambda_2^{-1}+\overline{\tau}\ln\lambda_1^{-1}}h_{top}(f,X)\leq h_{top}(f,\cap_{j=1}^{\infty}\mathfrak{S}(f,\varphi_j,\mathcal{Z}_j,S_j))\leq \frac{\ln L_1\ln L_2-\underline{\tau}\ln L_2}{\ln L_1\ln L_2+\underline{\tau}\ln L_1}h_{top}(f,X),$$
		\begin{equation*}
			\begin{split}
				(\frac{1}{\ln L_1}+\frac{1}{\ln L_2}\frac{\ln\lambda_1^{-1}\ln\lambda_2^{-1}-\overline{\tau}\ln\lambda_2^{-1} }{\ln\lambda_1^{-1}\ln\lambda_2^{-1}+\overline{\tau}\ln\lambda_1^{-1}})h_{top}(f,X)&\leq \operatorname{dim}_H \cap_{j=1}^{\infty}\mathfrak{S}(f,\varphi_j,\mathcal{Z}_j,S_j)\\
				&\leq (\frac{1}{\ln\lambda_1^{-1}}+\frac{1}{\ln\lambda_2^{-1}}\frac{\ln L_1\ln L_2-\underline{\tau}\ln L_2}{\ln L_1\ln L_2+\underline{\tau}\ln L_1})h_{top}(f,X).
			\end{split}
		\end{equation*}
		When $(\Lambda,f)$ is mixing, the condition $\cap_{j=1}^{\infty}\mathrm{Ind}'(\mathcal{Z}_j,S_j,Y)\neq\emptyset$ can be omitted.    When $(\Lambda,f)$ is mixing and $S_j=\mathbb{N}$ for any $j\geq 1,$ the results  hold if we replace $\overline{\tau}$ by $\underline{\tau}.$ Moreover, if  $L_1=\lambda_1^{-1}, L_2=\lambda_2^{-1}$, we have
		$$h_{top}(f,\cap_{j=1}^{\infty}\mathfrak{S}(f,\varphi_j,\mathcal{Z}_j,S_j))= \frac{\ln L_1\ln L_2-\underline{\tau}\ln L_2}{\ln L_1\ln L_2+\underline{\tau}\ln L_1}h_{top}(f,X),$$
		$$\operatorname{dim}_H \cap_{j=1}^{\infty}\mathfrak{S}(f,\varphi_j,\mathcal{Z}_j,S_j)= \frac{1}{\ln L_1}\frac{\ln L_1+\ln L_2}{\ln L_2+\underline{\tau}}h_{top}(f,X).$$
		\item[(2)] If $\underline{\tau}=\ln L_1,$ then $$h_{top}(f,\cap_{j=1}^{\infty}\mathfrak{S}(f,\varphi_j,\mathcal{Z}_j,S_j))=0,$$
		$$\operatorname{dim}_H \cap_{j=1}^{\infty}\mathfrak{S}(f,\varphi_j,\mathcal{Z}_j,S_j)\leq \frac{1}{\ln\lambda^{-1}_1}h_{top}(f,X).$$
		\item[(3)] If $ \underline{\tau}>\ln L_1,$ then $$h_{top}(f,\cap_{j=1}^{\infty}\mathfrak{S}(f,\varphi_j,\mathcal{Z}_j,S_j))=0,$$
		$$\operatorname{dim}_H \cap_{j=1}^{\infty}\mathfrak{S}(f,\varphi_j,\mathcal{Z}_j,S_j)=0.$$
	\end{enumerate}
\end{Thm}

Hyperbolic automorphisms of torus are the prime examples of hyperbolic dynamical systems. The action of a matrix $A \in S L(d, \mathbb{Z})$ on $\mathbb{R}^d$ induces an automorphism of the torus $\mathbb{T}^d=\mathbb{R}^d / \mathbb{Z}^d$ as following
$$
f_A(x)=A x \quad \bmod 1.
$$
$f_A$ or $A$ is called hyperbolic  if the matrix $A$ has no eigenvalues on the unit circle.
\begin{Thm}\label{Thm-hyper-torus}
	Let $A \in S L(d, \mathbb{Z})$  be a hyperbolic matrix. Assume that $A$ has only two different eigenvalues $\lambda_s,\lambda_u$  with $0<|\lambda_s|<1<|\lambda_u|.$
	\begin{enumerate}
		\item[(1)] If $\underline{\tau}<-\ln|\lambda_s|$, then
		the set
		$\cap_{j=1}^{\infty}\mathfrak{S}(f_A,\varphi_j,\mathcal{Z}_j,\mathbb{N})$ is non-empty and
		$$h_{top}(f_A,\cap_{j=1}^{\infty}\mathfrak{S}(f_A,\varphi_j,\mathcal{Z}_j,\mathbb{N}))= d_s\frac{\ln |\lambda_s|^{-1}\ln |\lambda_u|-\underline{\tau}\ln |\lambda_u|}{\ln |\lambda_u|+\underline{\tau}},$$
		$$\operatorname{dim}_H \cap_{j=1}^{\infty}\mathfrak{S}(f_A,\varphi_j,\mathcal{Z}_j,\mathbb{N})= d_s\frac{\ln |\lambda_s|^{-1}+\ln |\lambda_u|}{\ln |\lambda_u|+\underline{\tau}},$$
		where $d_s$ is the multiplicity of $\lambda_s.$
		\item[(2)] If $\underline{\tau}=-\ln|\lambda_s|,$ then $$h_{top}(f_A,\cap_{j=1}^{\infty}\mathfrak{S}(f_A,\varphi_j,\mathcal{Z}_j,\mathbb{N}))=0,$$
		$$\operatorname{dim}_H \cap_{j=1}^{\infty}\mathfrak{S}(f_A,\varphi_j,\mathcal{Z}_j,\mathbb{N})\leq d_s.$$
		\item[(3)] If $ \underline{\tau}>-\ln|\lambda_s|,$ then $$h_{top}(f_A,\cap_{j=1}^{\infty}\mathfrak{S}(f_A,\varphi_j,\mathcal{Z}_j,\mathbb{N}))=0,$$
		$$\operatorname{dim}_H \cap_{j=1}^{\infty}\mathfrak{S}(f_A,\varphi_j,\mathcal{Z}_j,\mathbb{N})=0.$$
	\end{enumerate}
\end{Thm}

\subsubsection{Expanding maps}
Let $M$ be a compact manifold and $f: M \rightarrow M$ be a map of class $C^1$. We say that $f$ is $\lambda$-expanding for some $\lambda>1$ if
$$
|Df (v)| \geq \lambda|v| \quad \forall x \in M, v \in T_x M.
$$
From \cite[Example 11.2.1]{VO-2016}, there exists $\rho>0$ such that for any $p\in M$ and any $x, y \in B(p, \rho)$ one has
$$
d(f(x), f(y)) \geq \lambda d(x, y),
$$
where $B(p, \rho):=\{z\in M: d(p,z)<\rho\}.$
It follows that $f: M \rightarrow M$ is $\ln \lambda$-hyperbolic.
\begin{Prop}\label{Prop-exp}
	If $f: M \rightarrow M$ is a $\lambda$-expanding $C^1$ map, then  $f: M \rightarrow M$ is $\ln \lambda$-hyperbolic.
\end{Prop}
\begin{proof}
	Fix $\varepsilon_h>0.$ Denote  $\varepsilon=\min\{\rho,\varepsilon_h\}$.  Fix $x,y\in \Lambda$ and $n\in\mathbb{N}$  such that $d_n(x,y)<\varepsilon.$ Since $d(f^{i}(x),f^{i}(y))<\varepsilon\leq \rho$ for any $0\leq i\leq n-1,$ we have $d(f^{i}(x),f^{i}(y))\leq \lambda^{-1}d(f^{i+1}(x),f^{i+1}(y))$  for any $0\leq i\leq n-2.$ It follows that $d(f^{i}(x),f^{i}(y))\leq \lambda^{i-n+1}d(f^{n-1}(x),f^{n-1}(y)) <\lambda^{i-n+1}\varepsilon<\varepsilon_h e^{-(n-1-i)\ln\lambda}$
	for any $0\leq i\leq n-1.$
\end{proof}
Similar to Proposition \ref{Prop-hyper-spec}, we have the following result.
\begin{Prop}\label{Prop-exp-spec}
	Let $f: M \rightarrow M$ be a transitive $\lambda$-expanding $C^1$ map. Then there is a non-empty compact set $Y\subseteq M$ and a positive integer $N$ such that $f^{N}(Y)=Y$, $M=\bigcup_{l=0}^{N-1}f^{l}(Y)$ and  $(Y,f^N)$  has the exponential specification property  with respect to the exponent $(+\infty,N\ln\lambda)$ at any scale. If further $(\Lambda,f)$ is mixing,  then $Y=M$ and $N=1.$
\end{Prop}
\begin{proof}
	By spectral decomposition \cite[Theorem 11.2.15]{VO-2016} there is a non-empty compact set $Y\subseteq M$ and a positive integer $N$ such that $f^{N}(Y)=Y$, $M=\bigcup_{l=0}^{N-1}f^{l}(Y)$ and  $(Y,f^N)$ is mixing. By  \cite[Proposition 11.3.1]{VO-2016}, $(Y,f^N)$ has the  specification property at any scale. So by Proposition \ref{Prop-exp}, Proposition \ref{Prop-N-hype} and Proposition \ref{Prop-exp+spec}, $(Y,f^N)$  has the  exponential  specification property  with respect to the exponent $(+\infty,N\ln\lambda)$ at any scale.
\end{proof}
\begin{Thm}\label{Thm-exp}
	Let $f: M \rightarrow M$ be a transitive $\lambda$-expanding $C^1$ map.  Denote $L=\sup\limits_{x\in\Lambda}\sup\limits_{v\in T_xM,v\neq0}\frac{|Df(v)|}{|v|}.$
	If  $\cap_{j=1}^{\infty}\mathrm{Ind}'(\mathcal{Z}_j,S_j,Y)\neq\emptyset$, then
	the set
	$\cap_{j=1}^{\infty}\mathfrak{S}(f,\varphi_j,\mathcal{Z}_j,S_j)$ is non-empty and  $$\frac{\ln\lambda }{\ln\lambda+\overline{\tau}}h_{top}(f,X)\leq h_{top}(f,\cap_{j=1}^{\infty}\mathfrak{S}(f,\varphi_j,\mathcal{Z}_j,S_j))\leq \frac{\ln L}{\ln L+\underline{\tau}}h_{top}(f,X),$$
	\begin{equation*}
		\begin{split}
			\frac{1}{\ln L}\frac{\ln\lambda }{\ln\lambda+\overline{\tau}}h_{top}(f,X)\leq \operatorname{dim}_H \cap_{j=1}^{\infty}\mathfrak{S}(f,\varphi_j,\mathcal{Z}_j,S_j)
			\leq \frac{1}{\ln\lambda^{-1}}\frac{\ln L}{\ln L+\underline{\tau}}h_{top}(f,X).
		\end{split}
	\end{equation*}
	When $(\Lambda,f)$ is mixing, the condition $\cap_{j=1}^{\infty}\mathrm{Ind}'(\mathcal{Z}_j,S_j,Y)\neq\emptyset$ can be omitted.     When $(\Lambda,f)$ is mixing and $S_j=\mathbb{N}$ for any $j\geq 1,$ the results  hold if we replace $\overline{\tau}$ by $\underline{\tau}.$ Moreover, if  $L=\lambda^{-1}$, we have
	$$h_{top}(f,\cap_{j=1}^{\infty}\mathfrak{S}(f,\varphi_j,\mathcal{Z}_j,S_j))= \frac{\ln L}{\ln L+\underline{\tau}}h_{top}(f,X),$$
	$$\operatorname{dim}_H \cap_{j=1}^{\infty}\mathfrak{S}(f,\varphi_j,\mathcal{Z}_j,S_j)=\frac{1}{\ln L+\underline{\tau}}h_{top}(f,X).$$
\end{Thm}
\begin{proof}
	By the definition of $L$, $f:M\to M$ is a $L$-Lipschitz map. So by Proposition \ref{Prop-exp-spec}, Corollary \ref{theorem-shrinking-2}, Theorem \ref{Thm-upper-bound} and Remark \ref{Rem-under}, we obtain Theorem \ref{Thm-exp}.
\end{proof}

Let $f_A: \mathbb{T}^d \rightarrow \mathbb{T}^d$ be the linear endomorphism of the torus induced by some matrix $A$ with integer coefficients and determinant different from zero. Assume that all the eigenvalues  of $A$ are  $0<|\lambda_1|\leq |\lambda_2|\leq \ldots\leq  |\lambda_d|$.
$f_A$ or $A$ is called expanding  if $1<|\lambda_1|\leq |\lambda_2|\leq \ldots\leq  |\lambda_d|$.
\begin{Thm}\label{Thm-exp-torus}
	Let $A \in G L(d, \mathbb{Z})$  be an expanding matrix.  Then
	the set
	$\cap_{j=1}^{\infty}\mathfrak{S}(f_A,\varphi_j,\mathcal{Z}_j,\mathbb{N})$ is non-empty and  $$\frac{\ln|\lambda_1| }{\ln|\lambda_1|+\underline{\tau}}\sum_{i=1}^{d}\ln|\lambda_i|\leq h_{top}(f_A,\cap_{j=1}^{\infty}\mathfrak{S}(f_A,\varphi_j,\mathcal{Z}_j,\mathbb{N}))\leq \frac{\ln |\lambda_d|}{\ln |\lambda_d|+\underline{\tau}}\sum_{i=1}^{d}\ln|\lambda_i|,$$
	\begin{equation*}
		\begin{split}
			\frac{1}{\ln |\lambda_d|}\frac{\ln|\lambda_1| }{\ln|\lambda_1|+\underline{\tau}}\sum_{i=1}^{d}\ln|\lambda_i|\leq \operatorname{dim}_H \cap_{j=1}^{\infty}\mathfrak{S}(f_A,\varphi_j,\mathcal{Z}_j,\mathbb{N})
			\leq \frac{1}{\ln|\lambda_1|}\frac{\ln |\lambda_d|}{\ln |\lambda_d|+\underline{\tau}}\sum_{i=1}^{d}\ln|\lambda_i|.
		\end{split}
	\end{equation*}
	 In particular, if  $|\lambda_1|=|\lambda_d|$, we have
	$$h_{top}(f_A,\cap_{j=1}^{\infty}\mathfrak{S}(f_A,\varphi_j,\mathcal{Z}_j,\mathbb{N}))= d\frac{(\ln |\lambda_d|)^{2}}{\ln |\lambda_d|+\underline{\tau}},$$
	$$\operatorname{dim}_H \cap_{j=1}^{\infty}\mathfrak{S}(f_A,\varphi_j,\mathcal{Z}_j,\mathbb{N})=d\frac{\ln |\lambda_d|}{\ln |\lambda_d|+\underline{\tau}}.$$
\end{Thm}

\subsubsection{Symbolic dynamics}
We first consider two-sided symbolic dynamics.
For any finite alphabet $A$, the set $A^{\mathbb{Z}}=\{\dots x_{-2}x_{-1}x_0x_{1}x_{2}\cdots : x_{i}\in A\}$ is called two-sided full symbolic space. The shift action on two-side full symbolic space is defined by $(\sigma(x))_{n}=x_{n+1}$, where $x =\left(x_{n}\right)_{n=-\infty}^{\infty}.$
$(A^{\mathbb{Z}},\sigma)$ forms a dynamical system under the discrete product topology which we call a shift. We equip $A^{\mathbb{Z}}$ with a compatible metric $d$ given by
\begin{equation}
	d(\omega,\gamma):=
	\begin{cases}
		e^{-\min\{|k|:k\in\mathbb{Z}, \omega_{k}\neq \gamma_{k}\}},&\omega\neq \gamma,\\
		0,&\omega= \gamma.
	\end{cases}
\end{equation}
A closed subset $X\subseteq A^{\mathbb{Z}}$ is called a two-sided subshift if it is invariant under the shift action $\sigma$.
From the definition of metric $d$, we have
$\frac{1}{e}d(\omega,\gamma)\leq d(\sigma(\omega),\sigma(\gamma))\leq ed(\omega,\gamma)$ for any $\omega,\gamma\in A^{\mathbb{N^{+}}}.$
It follows that $\sigma:\ A^{\mathbb{N^{+}}}\rightarrow A^{\mathbb{N^{+}}}$ is $(1,1)$-hyperbolic and $(e,e)$-bi-Lipschitz. (see, for example, \cite[Lemma 6.4]{DHT})
\begin{Prop}\label{Prop-two-side}
	$\sigma:\ A^{\mathbb{Z}}\rightarrow A^{\mathbb{Z}}$ is $(1, 1)$-hyperbolic and $(e,e)$-bi-Lipschitz.
\end{Prop}

\begin{Thm}\label{Thm-two-side}
	Let $X\subseteq A^{\mathbb{Z}}$ be a  two-sided subshift.  Assume that there is a non-empty compact set $Y\subseteq X$ and a positive integer $N$ such that $f^{N}(Y)=Y$, $X=\bigcup_{l=0}^{N-1}f^{l}(Y)$ and  $(Y,f^N)$  has the non-uniform specification property.
	\begin{enumerate}
		\item[(1)] If $\overline{\tau}<1$ and $\cap_{j=1}^{\infty}\mathrm{Ind}'(\mathcal{Z}_j,S_j,Y)\neq\emptyset$, then
		the set
		$\cap_{j=1}^{\infty}\mathfrak{S}(f,\varphi_j,\mathcal{Z}_j,S_j)$ is non-empty and  $$\frac{1-\overline{\tau} }{1+\overline{\tau}}h_{top}(f,X)\leq h_{top}(f,\cap_{j=1}^{\infty}\mathfrak{S}(f,\varphi_j,\mathcal{Z}_j,S_j))\leq \frac{1-\underline{\tau}}{1+\underline{\tau}}h_{top}(f,X),$$
		\begin{equation*}
			\begin{split}
				\frac{2 }{1+\overline{\tau}}h_{top}(f,X)\leq \operatorname{dim}_H \cap_{j=1}^{\infty}\mathfrak{S}(f,\varphi_j,\mathcal{Z}_j,S_j)
				\leq \frac{2}{1+\underline{\tau}}h_{top}(f,X).
			\end{split}
		\end{equation*}
		When $Y=X$, $N=1$ and $S_j=\mathbb{N}$ for any $j\geq 1,$ we have
		$$h_{top}(f,\cap_{j=1}^{\infty}\mathfrak{S}(f,\varphi_j,\mathcal{Z}_j,S_j))= \frac{1-\underline{\tau}}{1+\underline{\tau}}h_{top}(f,X),$$
		$$\operatorname{dim}_H \cap_{j=1}^{\infty}\mathfrak{S}(f,\varphi_j,\mathcal{Z}_j,S_j)= \frac{2}{1+\underline{\tau}}h_{top}(f,X).$$
		\item[(2)] If $\underline{\tau}=1,$ then $$h_{top}(f,\cap_{j=1}^{\infty}\mathfrak{S}(f,\varphi_j,\mathcal{Z}_j,S_j))=0,$$
		$$\operatorname{dim}_H \cap_{j=1}^{\infty}\mathfrak{S}(f,\varphi_j,\mathcal{Z}_j,S_j)\leq h_{top}(f,X).$$
		\item[(3)] If $ \underline{\tau}>1,$ then $$h_{top}(f,\cap_{j=1}^{\infty}\mathfrak{S}(f,\varphi_j,\mathcal{Z}_j,S_j))=0,$$
		$$\operatorname{dim}_H \cap_{j=1}^{\infty}\mathfrak{S}(f,\varphi_j,\mathcal{Z}_j,S_j)=0.$$
	\end{enumerate}
\end{Thm}
\begin{proof}
	By Proposition \ref{Prop-two-side}, Proposition \ref{Prop-N-hype} and Proposition \ref{Prop-exp+spec}, $(Y,f^N)$  has exponential non-uniform specification property  with respect to the exponent $(N,N)$. So by Corollary \ref{theorem-shrinking-2}, Theorem \ref{Thm-upper-bound} and Remark \ref{Rem-under}, we obtain Theorem \ref{Thm-two-side}.
\end{proof}
Theorem  \ref{Thm-two-side} applies in the following examples:
\begin{enumerate}
	\item two-sided shifts of finite type.
	\item two-sided sofic shifts: A shift which can be displayed
	as a factor of a shift of finite type is said to
	be a sofic shift. From \cite[Corollary 40]{KLO2016}, for every transitive  sofic shift (hence every transitive shift of finite type) $(X,\sigma)$ there is a non-empty compact set $Y\subseteq X$ and a positive integer $N$ such that $f^{N}(Y)=Y$, $X=\bigcup_{l=0}^{N-1}f^{l}(Y)$ and  $(Y,f^N)$  has specification property.  Weiss \cite{Weiss} proved that every mixing sofic shift (hence every mixing shift of finite type) has the  specification property.
	\item some two-sided shifts with non-uniform specification property.  Readers can refer to \cite{KLO2016,Pavlov} for examples with non-uniform specification property.
\end{enumerate}

Now we consider one-sided symbolic dynamics.
For any finite alphabet $A$, the set $A^{\mathbb{N^{+}}}=\{x_{1}x_{2}\cdots : x_{i}\in A\}$ is called one-sided full symbolic space. The shift action on one-sided full symbolic space is defined by
$$\sigma:\ A^{\mathbb{N^{+}}}\rightarrow A^{\mathbb{N^{+}}},\ \ \ x_{1}x_{2}\cdots\mapsto x_{2}x_{3}\cdots.$$
$(A^{\mathbb{N^{+}}},\sigma)$ forms a dynamical system under the discrete product topology which we call a shift. We equip $A^{\mathbb{N^{+}}}$ with a compatible metric $d$ given by
\begin{equation}
	d(\omega,\gamma):=
	\begin{cases}
		e^{-\min\{k\in\mathbb{N^{+}}:\omega_{k}\neq \gamma_{k}\}},&\omega\neq \gamma,\\
		0,&\omega= \gamma.
	\end{cases}
\end{equation}
A closed subset $X\subseteq A^{\mathbb{N^{+}}}$  is called a one-sided subshift if it is invariant under the shift action $\sigma$.
From the definition of metric $d$, we have
$d(\sigma(\omega),\sigma(\gamma))=nd(\omega,\gamma)$ for any $\omega,\gamma\in A^{\mathbb{N^{+}}}.$
It follows that $\sigma:\ A^{\mathbb{N^{+}}}\rightarrow A^{\mathbb{N^{+}}}$ is $1$-hyperbolic and $e$-Lipschitz.
\begin{Prop}\label{Prop-one-side}
	$\sigma:\ A^{\mathbb{N^{+}}}\rightarrow A^{\mathbb{N^{+}}}$ is $1$-hyperbolic and $e$-Lipschitz.
\end{Prop}
\begin{Thm}\label{Thm-one-side}
	Let $X\subseteq A^{\mathbb{N}^+}$ be a  one-sided subshift.  Assume that there is a non-empty compact set $Y\subseteq X$ and a positive integer $N$ such that $f^{N}(Y)=Y$, $X=\bigcup_{l=0}^{N-1}f^{l}(Y)$ and  $(Y,f^N)$  has the non-uniform specification property.
	If  $\cap_{j=1}^{\infty}\mathrm{Ind}'(\mathcal{Z}_j,S_j,Y)\neq\emptyset$, then
	the set
	$\cap_{j=1}^{\infty}\mathfrak{S}(f,\varphi_j,\mathcal{Z}_j,S_j)$ is non-empty and  $$\frac{1 }{1+\overline{\tau}}h_{top}(f,X)\leq h_{top}(f,\cap_{j=1}^{\infty}\mathfrak{S}(f,\varphi_j,\mathcal{Z}_j,S_j))\leq \frac{1}{1+\underline{\tau}}h_{top}(f,X),$$
	\begin{equation*}
		\begin{split}
			\frac{1 }{1+\overline{\tau}}h_{top}(f,X)\leq \operatorname{dim}_H \cap_{j=1}^{\infty}\mathfrak{S}(f,\varphi_j,\mathcal{Z}_j,S_j)
			\leq \frac{1}{1+\underline{\tau}}h_{top}(f,X).
		\end{split}
	\end{equation*}
	When $Y=X,$ $N=1$  and $S_j=\mathbb{N}$ for any $j\geq 1,$ we have
	$$h_{top}(f,\cap_{j=1}^{\infty}\mathfrak{S}(f,\varphi_j,\mathcal{Z}_j,S_j))= \frac{1}{1+\underline{\tau}}h_{top}(f,X),$$
	$$\operatorname{dim}_H \cap_{j=1}^{\infty}\mathfrak{S}(f,\varphi_j,\mathcal{Z}_j,S_j)=\frac{1}{1+\underline{\tau}}h_{top}(f,X).$$
\end{Thm}
\begin{proof}
	By Proposition \ref{Prop-one-side}, Proposition \ref{Prop-N-hype} and Proposition \ref{Prop-exp+spec}, $(Y,f^N)$  has the exponential non-uniform specification property  with respect to the exponent $(+\infty,N)$. So by Corollary \ref{theorem-shrinking-2}, Theorem \ref{Thm-upper-bound} and Remark \ref{Rem-under}, we obtain Theorem \ref{Thm-one-side}.
\end{proof}

Similar as Theorem  \ref{Thm-two-side},  Theorem  \ref{Thm-one-side} applies in the following examples:
\begin{enumerate}
	\item one-sided shifts of finite type.
	\item one-sided sofic shifts.
	\item some one-sided shifts with non-uniform specification property.  Readers can refer to \cite{KLO2016,Pavlov} for examples with non-uniform specification property.
\end{enumerate}

\textbf{Outline of the paper.}
Section \ref{sec-preliminaries} reviews definitions of entropy, Hausdorff dimension, and exponential specification property to make precise statements of the theorems and their proofs.
In Section \ref{sec-lower-bounds}, we concentrate on the lower bounds for topological entropy and Hausdorff dimension and give the proofs of Theorem \ref{theorem-shrinking-1}, Corollary \ref{theorem-shrinking-2} and Corollary \ref{Cor-covering}. In Section \ref{sec-upper-bounds}, we concentrate on the upper bounds for topological entropy and Hausdorff dimension and give the proof of Theorem  \ref{Thm-upper-bound}. In Section \ref{sec-app}, we give the proofs  of Theorem \ref{Thm-hyper},  \ref{Thm-hyper-torus} and \ref{Thm-exp-torus}.

\section{Preliminaries}\label{section-preliminaries}\label{sec-preliminaries}
In the present section, we introduce the necessary terminology and recall some results that will be needed.
Let $\mathbb{Z}$, $\mathbb{N}$, $\mathbb{N^{+}}$ denote the sets of integers, non-negative integers, and positive integers, respectively.   $\sharp Y$ denotes the cardinality of the set $Y.$

\subsection{Topological entropy and metric entropy}
First, we recall the classical topological entropy for $(X, f)$, introduced by Adler et al. \cite{Adler-Konheim-McAndrew} using open covers and  by Bowen \cite{Bowen-1971} using separated and spanning sets.
\begin{Def}\label{Def-entropy-comp}
	Let $(X, f)$ be a dynamical system. For $n \in \mathbb{N}$ and $\varepsilon>0$, a subset $E \subset X$ is called an $(n, \varepsilon)$-separated set in $X$ if for any distinct points $x, y$ in $E$, there is $k \in\{0, \cdots, n-1\}$ such that
	$
	d(f^k(x), f^k(y))>\varepsilon .
	$
	Denote by $s(X, n, \varepsilon)$ the maximal cardinality of $(n, \varepsilon)$-separated subsets of $X$.
	Then the \emph{topological entropy of $f$ on $X$} is given by
	$$
	h_{top}(f, X):=\lim _{\varepsilon \rightarrow 0} \limsup _{n \rightarrow \infty} \frac{\ln s(X, n, \varepsilon)}{n} .
	$$
\end{Def}

For noncompact sets, Bowen also developed a satisfying definition via dimension language \cite{Bowen-1973}, which is a cross between topological entropy and Hausdorff dimension.
\begin{Def}
	Let $(X, f)$ be a dynamical system and $Z\subseteq X$. For $s\geq 0$, $N\in\mathbb{N}^+$, and $\varepsilon >0$, define
	$$\mathcal{M}_{N,\varepsilon}^{s}(Z):=\inf \sum_{i}\exp(-sn_{i})$$
	where the infimum is taken over all finite or countable families $\{B_{n_{i}}(x_{i},\varepsilon,f)\}$ such that $x_{i}\in X$, $n_{i}\geq N$ and $\bigcup_{i}B_{n_{i}}(x_{i},\varepsilon,f)\supseteq Z$.
	We set
	$\mathcal{M}_{\varepsilon}^{s}(Z):=\lim\limits_{N\to \infty}\mathcal{M}_{N,\varepsilon}^{s}(Z)$
	and define
	$$h_{top}(f,Z,\varepsilon):=\inf\{s:\mathcal{M}_{\varepsilon}^{s}(Z)=0\}=\sup\{s:\mathcal{M}_{\varepsilon}^{s}(Z)=\infty\}.$$
	The \emph{Bowen topological entropy of $f$ on $Z$} is defined as $$h_{top}(f,Z):=\lim_{\varepsilon\to\infty}h_{top}(f,Z,\varepsilon).$$
	When $Z=X,$ $h_{top}(f,Z)$ equals to the  classical  topological entropy in Definition \ref{Def-entropy-comp}.
\end{Def}
Some basic properties of Bowen topological entropy are as follows:
\begin{Prop}\cite[Proposition 2]{Bowen-1973}\label{basic properties}
	\begin{description}
		\item[(1)] $h_{top}(f, \bigcup_{i=1}^{\infty} Z_{i}) = \sup\limits _{i \geq 1} h_{top}\left(f, Z_{i}\right).$
		\item[(2)] $h_{top}(f^m, Z) = m h_{top}\left(f, Z\right)$ for any $m\in\mathbb{N^{+}}.$
	\end{description}
\end{Prop}

\begin{Def}
	We call $(X,  \B,  \mu)$ a probability space if $\B$ is a Borel $\sigma$-algebra on $X$ and $\mu$ is a probability measure on $X$.   For a finite measurable partition $\xi=\{A_1,  \cdots,  A_n\}$ of a probability space $(X,  \B,  \mu)$,   define
	$$H_\mu(\xi)=-\sum_{i=1}^n\mu(A_i)\log\mu(A_i).  $$
	Let $f:X\to X$ be a continuous map preserving $\mu$.   We denote by $\bigvee_{i=0}^{n-1}f^{-i}\xi$ the partition whose element is the set $\bigcap_{i=0}^{n-1}f^{-i}A_{j_i},  1\leq j_i\leq n$.   Then the following limit exists:
	$$h_\mu(f,  \xi)=\lim_{n\to\infty}\frac1n H_\mu\left(\bigvee_{i=0}^{n-1}f^{-i}\xi\right)$$
	and we define \emph{the metric entropy of $\mu$} as
	$$h_{\mu}(f):=\sup\{h_\mu(f,  \xi):\xi~\textrm{is a finite measurable partition of X}\}. $$
\end{Def}

Following the idea of Brin and Katok \cite{BK}, Feng and Huang introduced the measure-theoretical lower entropy  of $\mu$ for each $\mu\in\mathcal{M}(X)$, where  $\mathcal{M}(X)$ is  the set of all Borel probability measures on $X.$
\begin{Def}
	Let $\mu \in \mathcal{M}(X)$. The \emph{measure-theoretical lower entropy of $\mu$} is defined  by
	$$
	\underline{h}_\mu(f):=\int \underline{h}_\mu(T, x) d \mu,
	$$
	where
	\begin{equation*}
			\underline{h}_\mu(f, x):=\lim _{\varepsilon \rightarrow 0} \liminf _{n \rightarrow+\infty}-\frac{1}{n} \log \mu\left(B_n(x, \varepsilon)\right).
	\end{equation*}
\end{Def}
Feng and Huang in \cite{Feng-Huang} give the following variational principles for topological entropies of subsets and the measure-theoretical lower entropy.
\begin{Lem}\label{lemma-aa}\cite[Theorem 1.2 and 1.3]{Feng-Huang}
	Let $(X, f)$ be a  dynamical system. If $Z\subset X$  is non-empty and compact, then
	$
			h_{top}(f,Z)= \sup\{\underline{h}_{\mu}(f):\mu\in\mathcal{M}(X), \mu(Z)=1\}.
	$
\end{Lem}
From Lemma \ref{lemma-aa}, we have the following corollary.
\begin{Cor}\label{Cor-aa}
	Let $(X, f)$ be a  dynamical system and $Z\subset X$  be a non-empty  compact set, if there is $\mu\in\mathcal{M}(X)$ and $\varepsilon>0$   such that $\mu(Z)=1$ and
	$
		\mu\left(B_n(x, \varepsilon)\right)\leq e^{-nh}.
$
	for any $x\in Z$ and large enough $n\in\mathbb{N^{+}},$ then $h_{top}(f,Z)\geq h.$
\end{Cor}

For $\delta>0$,   $\eps>0$ and $n\in\N$,   two points $x$ and $y$ are $(\delta,  n,  \eps)$-separated with respect to $f$ if
$$\sharp\{0\leq j\leq n-1:d(f^{j}(x),  f^{j}(y))>\eps\}\geq\delta n. $$
A subset $E\subset X$ is called a  $(\delta,  n,  \eps)$-separated set of $X$ with respect to $f$  if any pair of different points of $E$ are $(\delta,  n,  \eps)$-separated  with respect to $f$.
\begin{Lem}\label{lemma-A}
	Let $(X, f)$ be a  dynamical system. Then for any $h^*<h_{top}(f,X)$, there exist $\delta^*>0$ and $\varepsilon^*>0$ so that there exists $n^* \in \mathbb{N}$ such that for any $n \geq n^*$, there exists $\Gamma_n \subseteq X$ which is $\left(\delta^*, n, \varepsilon^*\right)$-separated with respect to $f$ and satisfies $\ln \sharp \Gamma_n \geq n h^*$.
\end{Lem}
\begin{proof}
	By the clasical variational principle \cite[Corollary 8.6.1]{Walters}, there is an ergodic measure $\mu$ such that $h_{\mu}(f)>h^*.$ From \cite[Proposition 2.1]{PS2005},  there exist $\delta^*>0$ and $\varepsilon^*>0$ so that there exists $n^* \in \mathbb{N}$ such that for any $n \geq n^*$, there exists $\Gamma_n \subseteq X$ which is $\left(\delta^*, n, \varepsilon^*\right)$-separated with respect to $f$ and satisfies $\ln \sharp \Gamma_n \geq n h_{\mu}(f)>nh^*$.
\end{proof}

\subsection{Hausdorff dimension}
\begin{Def}
	Let $X$ be a compact metric space equipped with a metric $d$. Given a subset $Z$ of $X$, for $s \geq 0$ and $\delta>0$, define
	$$
	\mathcal{H}_\delta^s(Z):=\inf \left\{\sum_i\left|U_i\right|^s: Z \subset \bigcup_i U_i,\left|U_i\right| \leq \delta\right\}
	$$
	where $|\cdot|$ denotes the diameter of a set. The quantity $\mathcal{H}^s(Z):=\lim\limits _{\delta \rightarrow 0} \mathcal{H}_\delta^s(Z)$ is called the s-dimensional Hausdorff measure of $Z$. Define the \emph{Hausdorff dimension of $Z$}, denoted by $\operatorname{dim}_H Z$, as follows:
	$$
	\operatorname{dim}_H Z:=\inf \left\{s: \mathcal{H}^s(Z)=0\right\}=\sup \left\{s: \mathcal{H}^s(Z)=\infty\right\}.
	$$
\end{Def}
\begin{Lem}\cite{Falconer}\label{Lem-mass}
	Let  $\mu$ be a probability measure supported on a measurable set $E$.  Suppose there are positive constants $c$ and $r_0$ such that
	$
	\mu(B(x,r)) \leq c r^s
	$
	for any $x\in E$ and $0<r \leq r_0$, where $B(x,r):=\{y\in X:d(x,y)<r\}$. Then $\operatorname{dim}_H E \geq s$.
\end{Lem}

\begin{Lem}\label{lem-upper-bound}
	Suppose that $(X,f)$ is a dynamical system. 
	\begin{enumerate}
		\item[(1)] 
		If $f:X\to X$ is  $\lambda$-hyperbolic, then
		$
			\operatorname{dim}_H X \leq \frac{1}{\lambda}h_{top}(f,X).
		$
		\item[(2)] 
			If  $f:X\to X$ is a $(\lambda_1,\lambda_2)$-hyperbolic homeomorphism, then
			$
				\operatorname{dim}_H X \leq (\frac{1}{\lambda_1}+\frac{1}{\lambda_2})h_{top}(f,X).
		$
	\end{enumerate}
\end{Lem}
\begin{proof}
	Fix $\delta>0.$ Then there is $\varepsilon>0$ such that
$
		\limsup\limits _{n \rightarrow \infty} \frac{\ln s(X, n, \frac{\varepsilon}{2})}{n}<h_{top}(f,X)+\delta,
$
	And there is $N^*\in\mathbb{N}$ such that for any $n\geq N^*,$
	\begin{equation}\label{lem-eq-upper-3}
		s(X, n, \frac{\varepsilon}{2})<e^{n(h_{top}(f,X)+2\delta)}.
	\end{equation}
For any $n\in\mathbb{N},$ let $\Gamma_{n,\frac{\varepsilon}{2}}$ be an $(n, \frac{\varepsilon}{2})$-separated set in $X$ with $\sharp \Gamma_{n,\frac{\varepsilon}{2}}=s(X, n, \frac{\varepsilon}{2}).$ Then we have $X=\bigcup_{w\in \Gamma_{n,\frac{\varepsilon}{2}}}B_{n}(w,\frac{\varepsilon}{2},f).$

(1) If $f:X\to X$ is  $\lambda$-hyperbolic, we also  require that $0<\varepsilon<\varepsilon_h$ for some $\varepsilon_h>0$ such that for any $x\in X$ and $n\in\mathbb{N},$ one has $B_{n}(x,\varepsilon,f)\subset B_n(x,\varepsilon_h,\lambda,f).$ Then
we have $B_{n}(w,\varepsilon,f)\subset B_{n}(w,\varepsilon_h,\lambda,f)\subset B(w,\varepsilon_h e^{-(n-1)\lambda})$.
Hence, for any $n\geq N^*$, we get an  cover of $X$ as:
$$X=\bigcup_{w\in \Gamma_{n,\frac{\varepsilon}{2}}} B(w,\varepsilon_h e^{-(n-1)\lambda}).$$
Denote $\delta_n:=\varepsilon_h e^{-(n-1)\lambda}.$ Then $\lim\limits_{n\to\infty}\delta_n=0.$
Let $s_h:=\frac{1}{\lambda}(h_{top}(f,X)+3\delta).$ Then for any $n\geq N^*,$ by (\ref{lem-eq-upper-3}) we have
\begin{equation*}
	\begin{split}
		\mathcal{H}_{\delta_n}^{s_h}(X)\leq\sum_{w\in \Gamma_{n,\frac{\varepsilon}{2}}} \varepsilon_h^{s_h} e^{-(n-1)\lambda s_h}
		\leq \varepsilon_h^{s_h} e^{h_{top}(f,X)+3\delta}\sum_{w\in \Gamma_{n,\frac{\varepsilon}{2}}}  e^{-n(h_{top}(f,X)+3\delta)}
		\leq \varepsilon_h      ^{s_h} e^{h_{top}(f,X)+3\delta} e^{-n\delta}.
	\end{split}
\end{equation*}
It implies $\mathcal{H}^{s_h}(X):=\lim\limits _{n \rightarrow \infty}\mathcal{H}_{\delta_n}^{s_h}(X)=0.$ So $s_h=\frac{1}{\lambda}(h_{top}(f,X)+3\delta)\geq \operatorname{dim}_H X.$
Let $\delta\to 0,$ we obtain $\operatorname{dim}_H X \leq \frac{1}{\lambda}h_{top}(f,X).$

(2) 
If  $f:X\to X$ is  a $(\lambda_1,\lambda_2)$-hyperbolic homeomorphism, we also  require that $0<\varepsilon<\varepsilon_h$ for some $\varepsilon_h>0$ such that for any $x\in X$ and $n\in\mathbb{N},$ one has $B_{n}(x,\varepsilon,f)\subset B_n(x,\varepsilon_h,\lambda_1,\lambda_2,f).$ Take $\tilde{N}^*\geq N^*$ such that $\frac{\lambda_2n}{\lambda_1}\geq N^*+1$ for any $n\geq \tilde{N}^*.$ Denote $\psi(n)=\lceil\frac{\lambda_2n}{\lambda_1}\rceil.$
Then we have  $\psi(n)\geq N^*$ and
\begin{equation}\label{lem-equation-EE}
	\frac{\lambda_2n}{\lambda_1}\leq \psi(n)\leq \frac{\lambda_2n}{\lambda_1}+1.
\end{equation}
Since  $X=\bigcup_{w\in \Gamma_{\psi(n),\frac{\varepsilon}{2}}} B_{\psi(n)}(w,\frac{\varepsilon}{2},f)$, then $X=\bigcup_{w\in \Gamma_{\psi(n),\frac{\varepsilon}{2}}}f^{\psi(n)}B_{\psi(n)}(w,\frac{\varepsilon}{2},f)$ and thus
$$X= \bigcup_{w\in \Gamma_{n,\frac{\varepsilon}{2}}} \bigcup_{\tilde{w}\in \Gamma_{\psi(n),\frac{\varepsilon}{2}}}B_{n}(w,\frac{\varepsilon}{2},f)\cap f^{\psi(n)}B_{\psi(n)}(\tilde{w},\frac{\varepsilon}{2},f).$$
When $B_{n}(w,\frac{\varepsilon}{2},f)\cap f^{\psi(n)}B_{\psi(n)}(\tilde{w},\frac{\varepsilon}{2},f)\neq\emptyset,$ choose $w''\in B_{n}(w,\frac{\varepsilon}{2},f)\cap f^{\psi(n)}B_{\psi(n)}(\tilde{w},\frac{\varepsilon}{2},f).$ Then  $$B_{n}(w,\frac{\varepsilon}{2},f)\cap f^{\psi(n)}B_{\psi(n)}(\tilde{w},\frac{\varepsilon}{2},f)\subset B_{n}(w'',\varepsilon,f)\cap B_{\psi(n)}(w'',\varepsilon,f^{-1})=f^{\psi(n)}B_{n+\psi(n)}(f^{-\psi(n)}(w''),\varepsilon,f).$$
Since $f:X\to X$ is  $(\lambda_1,\lambda_2)$-hyperbolic,
then
we have
\begin{equation*}
	\begin{split}
		&B_{n+\psi(n)}(f^{-\psi(n)}(w''),\varepsilon,f)\\
		\subset& B_{n+\psi(n)}(f^{-\psi(n)}(w''),\varepsilon_h,\lambda_1,\lambda_2,f)\\
		\subset &\{y\in X:d(f^{\psi(n)}(f^{-\psi(n)}(w'')),f^{\psi(n)}(y))< \varepsilon_h e^{-\min\{\psi(n)\lambda_1,(n+\psi(n)-1-\psi(n))\lambda_2\} }\}\\
		=&\{y\in X:d(w'',f^{\psi(n)}(y))< \varepsilon_h e^{-\min\{\psi(n)\lambda_1,(n-1)\lambda_2\} }\}\\
		=&\{y\in X:d(w'',f^{\psi(n)}(y))< \varepsilon_h e^{-(n-1)\lambda_2 }\}\qquad\qquad\qquad\qquad\qquad\qquad~~~(\text{using }  (\ref{equation-EE}))\\
		=&f^{-\psi(n)}B(w'',\varepsilon_h e^{-(n-1)\lambda_2}).
	\end{split}
\end{equation*}
Hence, $B_{n}(w,\frac{\varepsilon}{2},f)\cap f^{\psi(n)}B_{\psi(n)}(\tilde{w},\frac{\varepsilon}{2},f)\subset B(w'',\varepsilon_h e^{-(n-1)\lambda_2}).$
For any $N\geq \tilde{N}^*$, we get an  cover of $X$:
$$X=\bigcup_{w\in \Gamma_{n,\frac{\varepsilon}{2}}} \bigcup_{\tilde{w}\in \Gamma_{\psi(n),\frac{\varepsilon}{2}}}B(w'',\varepsilon_h e^{-(n-1)\lambda_2}).$$
Denote $\delta_n:=\varepsilon_h e^{-(n-1)\lambda_2}.$ Then  $\lim\limits_{n\to\infty}\delta_n=0.$
Let $s_h:=\frac{1}{\lambda_2}(h_{top}(f,X)+3\delta)+\frac{1}{\lambda_1}(h_{top}(f,X)+2\delta).$ Then for any $n\geq \tilde{N}^*,$  we have
\begin{equation*}
	\begin{split}
		\mathcal{H}_{\delta_n}^{s_h}(X)
		\leq &\sum_{w\in \Gamma_{n,\frac{\varepsilon}{2}}} \sum_{\tilde{w}\in \Gamma_{\psi(n),\frac{\varepsilon}{2}}}\varepsilon_h^{s_h} e^{-(n-1)\lambda_2 s_h}\\
		\leq &\varepsilon_h^{s_h} e^{\lambda_2s_h}\sum_{w\in \Gamma_{n,\frac{\varepsilon}{2}}}  e^{-n\lambda_2 s_h}e^{\psi(n)(h_{top}(f,X)+2\delta)}\\
		=& \varepsilon_h^{s_h} e^{\lambda_2s_h}\sum_{w\in \Gamma_{n,\frac{\varepsilon}{2}}}  e^{-n(h_{top}(f,X)+3\delta)}e^{(-\frac{\lambda_2n}{\lambda_1}+\psi(n))(h_{top}(f,X)+2\delta)}\\
		\leq & \varepsilon_h^{s_h} e^{\lambda_2s_h}e^{h_{top}(f,X)+2\delta}\sum_{w\in \Gamma_{n,\frac{\varepsilon}{2}}}  e^{-n(h_{top}(f,X)+3\delta)}~~~~~\qquad\quad\qquad\qquad\qquad\qquad\qquad~~~~~~~~~~~~~~~(\text{using }  (\ref{lem-equation-EE}))\\
		\leq & \varepsilon_h^{s_h} e^{\lambda_2s_h}e^{h_{top}(f,X)+2\delta}e^{-n\delta}.~~~~~~~\quad\qquad\qquad\qquad\qquad\qquad\qquad\qquad\qquad\qquad\qquad~~~~~~~~~~~~(\text{using }  (\ref{lem-eq-upper-3}))
	\end{split}
\end{equation*}
It implies $\mathcal{H}^{s_h}(X):=\lim\limits _{N \rightarrow \infty}\mathcal{H}_{\delta_n}^{s_h}(X)=0.$ So $s_h=\frac{1}{\lambda_2}(h_{top}(f,X)+3\delta)+\frac{1}{\lambda_1}(h_{top}(f,X)+2\delta)\geq \operatorname{dim}_H X.$
Let $\delta\to 0,$ we obtain  $\operatorname{dim}_H X \leq (\frac{1}{\lambda_1}+\frac{1}{\lambda_2})h_{top}(f,X).$
\end{proof}

\subsection{Exponential specification property}\label{sec-specfication}
First, we introduce exponential Bowen ball. Let $x\in X,$ $n\in\mathbb{N^{+}}$ and $\varepsilon,\lambda_1,\lambda_2>0,$ we denote the exponential $(n,\varepsilon)$-Bowen ball centered in $x$ with respect to the exponent $(\lambda_1,\lambda_2)$  by
$$B_{n}(x,\varepsilon,\lambda_1,\lambda_2,f):=\{y\in X:d(f^{i}(x),f^i(y))< \varepsilon e^{-\min\{i\lambda_1,(n-1-i)\lambda_2\} }\text{ for any }0\leq i\leq n-1\},$$
$$\overline{B}_{n}(x,\varepsilon,\lambda_1,\lambda_2,f):=\{y\in X:d(f^{i}(x),f^i(y))\leq  \varepsilon e^{-\min\{i\lambda_1,(n-1-i)\lambda_2\} }\text{ for any }0\leq i\leq n-1\}.$$

Now we define the exponential (non-uniform) specification property and recall the classic (non-uniform) specification property.
\begin{Def}\label{def-exp-nonuni-speci}
	We say that a dynamical system $(X,f)$  satisfies the \emph{exponential non-uniform specification property with respect to the exponent $(\lambda_1,\lambda_2)$  at scale $\varepsilon$}, if  for every $n \geq 1$ there exists a positive integer $p(n, \varepsilon)$ so that
	$$
	\limsup _{n \rightarrow \infty} \frac{1}{n} p(n, \varepsilon)=0
	$$
	and the following holds: given points $x_1, \cdots, x_k$ in $X$ and positive integers $n_1, \cdots, n_k$, if $p_i \geq p\left( n_i, \varepsilon\right)$ then there exists $z \in X$ such that
	\begin{equation*}
		z \in B_{n_1}\left(x_1,  \varepsilon,\lambda_1,\lambda_2,f\right) \text{ and } f^{n_1+p_1+\cdots+n_{i-1}+p_{i-1}}(z) \in B_{n_i}\left(x_i, \varepsilon,\lambda_1,\lambda_2,f\right) \text{ for every  }2 \leq i \leq k.
	\end{equation*}
	When $B_{n_i}\left(x_i, \varepsilon,\lambda_1,\lambda_2,f\right) $ is replaced by $B_{n_i}\left(x_i, \varepsilon,f\right),$ we say $(X,f)$  satisfies the \emph{non-uniform specification property at scale $\varepsilon$.}
\end{Def} 
\begin{Def}\label{def-nonuni-speci}
	We say that a dynamical system $(X,f)$  satisfies the \emph{exponential specification property with respect to the exponent $(\lambda_1,\lambda_2)$  at scale $\varepsilon$}, if  for there exists a positive integer $p_\varepsilon $ so that the following holds: given points $x_1, \cdots, x_k$ in $X$ and positive integers $n_1, \cdots, n_k$, if $p_i \geq p_\varepsilon$ then there exists $z \in X$ such that
	\begin{equation*}
		z \in B_{n_1}\left(x_1,  \varepsilon,\lambda_1,\lambda_2,f\right) \text{ and } f^{n_1+p_1+\cdots+n_{i-1}+p_{i-1}}(z) \in B_{n_i}\left(x_i, \varepsilon,\lambda_1,\lambda_2,f\right) \text{ for every  }2 \leq i \leq k.
	\end{equation*}
	When $B_{n_i}\left(x_i, \varepsilon,\lambda_1,\lambda_2,f\right) $ is replaced by $B_{n_i}\left(x_i, \varepsilon,f\right),$ we say $(X,f)$  satisfies the \emph{specification property at scale $\varepsilon$.}
\end{Def}
\begin{Rem}
	When $f:X\to X$ is a homeomorphism and $(X,f)$ satisfies the exponential non-uniform specification property with respect to the exponent $(\lambda_1,\lambda_2)$  at scale $\varepsilon$, the following holds: given points $x_{-k_2},\dots,x_{-1},$$x_1, \cdots, x_{k_1}$ in $X$ and positive integers $n_{-k_2},\dots,n_{-1},n_1, \cdots, n_{n_k}$, if $p_i \geq p\left( n_i, \varepsilon\right)$ then there exists $z \in X$ such that
	\begin{equation*}
		z \in B_{n_1}\left(x_1,  \varepsilon,\lambda_1,\lambda_2,f\right) \text{ and } f^{n_1+p_1+\cdots+n_{i-1}+p_{i-1}}(z) \in B_{n_i}\left(x_i, \varepsilon,\lambda_1,\lambda_2,f\right) \text{ for every  }2 \leq i \leq k_1,
	\end{equation*}
	\begin{equation*}
		f^{-n_{-1}-p_{-1}-\cdots-n_{-i}-p_{-i}}(z) \in B_{n_{-i}}\left(x_{-i}, \varepsilon,\lambda_1,\lambda_2,f\right) \text{ for every  }1 \leq i \leq k_2.
	\end{equation*}
\end{Rem}
By Definition \ref{def-exp-nonuni-speci}, \ref{def-nonuni-speci} and the definition of $(\lambda_1,\lambda_2)$-hyperbolic(see Subsection \ref{subsection_upperbounds}), we have the following proposition.
\begin{Prop}\label{Prop-exp+spec}
	Suppose a dynamical system $(X,f)$  satisfies the (non-uniform) specification property at  any scale $\varepsilon.$ If  $(X,f)$ is  $(\lambda_1,\lambda_2)$-hyperbolic, then $(X,f)$  satisfies the exponential (non-uniform) specification property with respect to the exponent $(\lambda_1,\lambda_2)$ at any scale.
\end{Prop}
\begin{Prop}\label{Prop-N-hype}
	If a dynamical system $(X,f)$ is $(\lambda_1,\lambda_2)$-hyperbolic, then $(X,f^N)$ is $(N\lambda_1,N\lambda_2)$-hyperbolic.
\end{Prop}

As readers will realize from Definition \ref{def-exp-nonuni-speci}, \ref{def-nonuni-speci} and Proposition \ref{Prop-exp+spec}, different specification properties have the relationships shown in Figure \ref{Fig-1}.
\begin{figure}\caption{The relationships between different specification properties}\label{Fig-1}
	\centering
	\includegraphics[width=16cm]{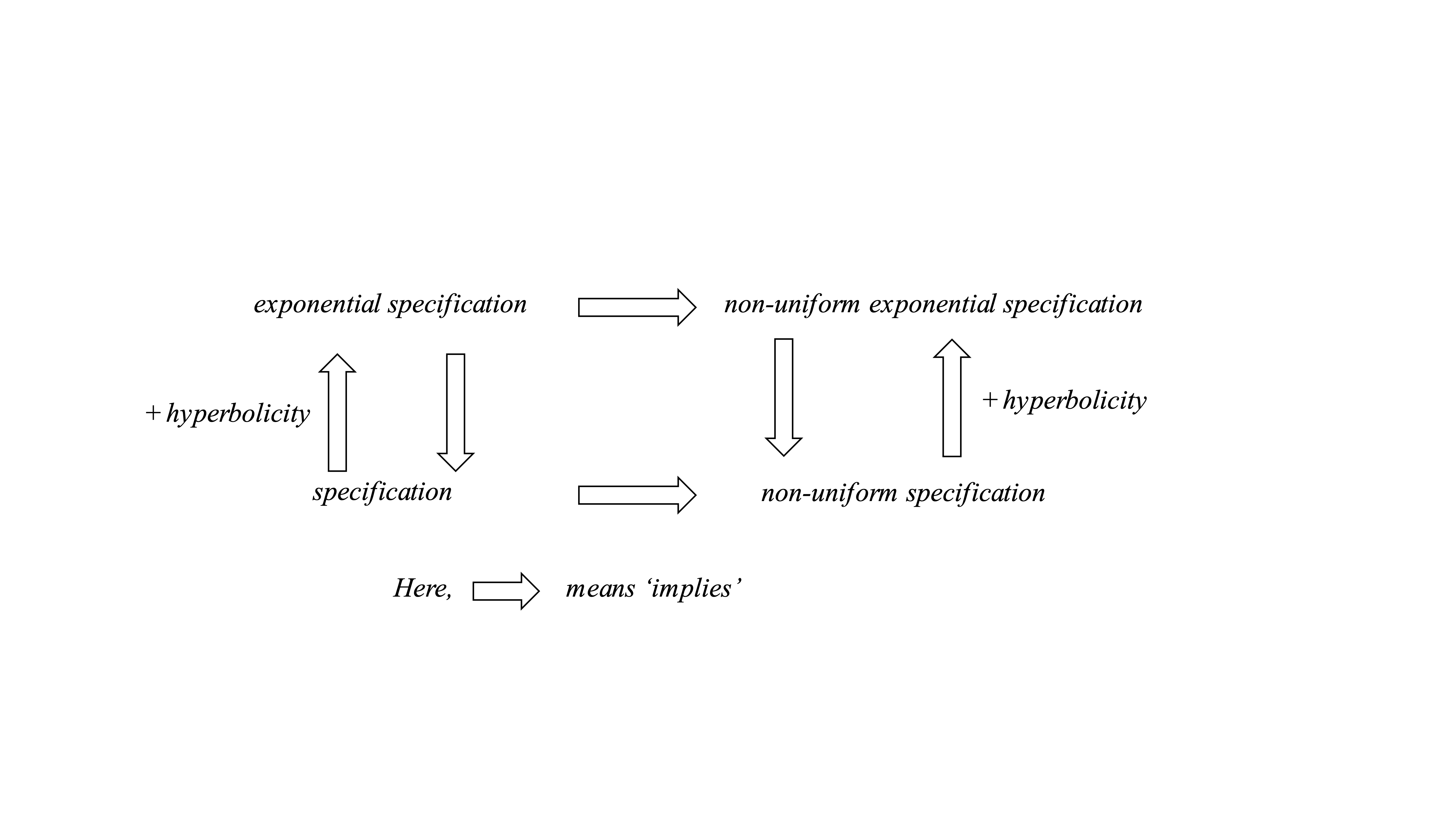}
\end{figure}

\section{Lower bounds }\label{sec-lower-bounds}
\subsection{Proof of Theorem \ref{theorem-shrinking-1} }
We are intended to prove  Theorem \ref{theorem-shrinking-1}  by the following strategy. First, we choose appropriate parameters and  orbit segments (Steps 1 and 2).  Then we construct a Cantor subset $G$ of  $\cap_{j=1}^{\infty}\mathfrak{S}(f,\varphi_j,\mathcal{Z}_j,S_j)$ and a probability measure $\mu$ supported on  $G$ by gluing these orbit segments using exponential non-uniform specification property (Steps 3 and 4).
Finally we estimate  the $\mu$ measure of dynamical ball $B_n(z, \varepsilon,f)$ and  ball $B(z,\varepsilon)$, and obtain topological entropy and Hausdorff dimension of $\cap_{j=1}^{\infty}\mathfrak{S}(f,\varphi_j,\mathcal{Z}_j,S_j)$ by the variational principles for topological entropies of subsets and  the mass distribution principle, respectively (Steps 5, 6, 7).
Now, we proceed with the proof as follows.
\newline

\noindent {\bf Step 1: preliminary choices}

$-$ When $h_{top}(f,X)=0,$ it's clear that the results of  Theorem \ref{theorem-shrinking-1} hold. Now, we assume that $h_{top}(f,X)>0.$

$-$ Fix $0<\eta<\min\{\frac{\lambda_1-\overline{\tau}}{2\lambda_1+1}, h_{top}(f,X),1\}.$

$-$ Choose $\tilde{n}\in\mathbb{N}$ such that $p(n,\varepsilon_0)<n\eta$ for any $n\geq \tilde{n}.$

$-$ By Lemma \ref{lemma-A}  there exist $\delta^*>0$ and $\varepsilon^*>0$ so that there exists $n^* \in \mathbb{N}$ such that for any $n \geq n^*$, there exists $\Gamma_n \subseteq \Lambda$ which is $\left(\delta^*, n, 3\varepsilon^*\right)$-separated  and satisfies $\ln \sharp \Gamma_n \geq n (h_{top}(f,X)-\eta)$.

$-$ Choose $\mathfrak{n}>\max\{\tilde{n},n^*\}$ such that $\mathfrak{n}\delta^*>\max\{\lceil \frac{\ln \varepsilon_0-\ln \varepsilon^*}{\lambda_1}\rceil+\lceil \frac{\ln \varepsilon_0-\ln \varepsilon^*}{\lambda_2}\rceil,0\}+1.$

$-$ Choose $\overline{n}_j\geq \tilde{n}$ such that $\varphi_j(n)\geq e^{-n(\overline{\tau}_j+\eta)}$ and $e^{\lambda_1\eta n}>\varepsilon_0 $  for any $n\geq \overline{n}_j.$
\newline

\noindent {\bf Step 2: choices of orbit segments}

Now we  choose appropriate  orbit segments. Before that, we give a useful lemma.
\begin{Lem}\label{Lemma-B}
	Let $\eta>0,$ $0<\alpha<\beta<1$ be three real numbers, $M,P$ be two positive integers and $S=\{s_i\}_{i=0}^{\infty}$ be a sequence of positive integers with $\lim\limits_{i \rightarrow \infty}s_i=\infty$. Then for any $n\in\mathbb{N}$   there exist  $s\in S$ and $m,N,t\in \mathbb{N}^{+}$ such that
	\begin{enumerate}
		\item[(1)] $s>n,m>n,$
		\item[(2)] $\alpha<\frac{m}{s}<\beta,$
		\item[(3)] $0<\frac{M}{s-M-m}<\eta.$
		\item[(4)] $s=M+m+NP+t$ with $0\leq t\leq P-1,$
	\end{enumerate}
\end{Lem}
\begin{proof}
	$\alpha<\frac{m}{s}<\beta$
	if and only if
	$
	\alpha s<m<\beta s.
	$
	This lemma is from the facts:
	$
	\lim\limits_{j \rightarrow \infty} s_j(\beta-\alpha)=\infty.
	$
\end{proof}

Recall that $0<\eta<\frac{\lambda_1-\overline{\tau}}{2\lambda_1+1}.$ Then we have $\frac{\overline{\tau}+\eta}{\lambda_1}+\eta<\frac{\overline{\tau}+\eta}{\lambda_1}+2\eta<1.$ Apply Lemma \ref{Lemma-B} to $\eta, \alpha=\frac{\overline{\tau}+\eta}{\lambda_1}+\eta, \beta=\frac{\overline{\tau}+\eta}{\lambda_1}+2\eta,$ $M=0,$ $P=\mathfrak{n}+p(\mathfrak{n},\varepsilon_0)$ and $S_1,$  there exist $s_1\in S_1$ and $m_1,N_1,t_{1}\in\mathbb{N^{+}}$ such that
\begin{equation*}
	\begin{split}
		&s_1>\overline{n}_1, m_1>\tilde{n}.\\
		&\frac{\overline{\tau}+\eta}{\lambda_1}+\eta<\frac{m_1}{s_1}<\frac{\overline{\tau}+\eta}{\lambda_1}+2\eta,\\
		&s_1=N_1(\mathfrak{n}+p(\mathfrak{n},\varepsilon_0))+m_1+t_1 \text{ with } 0\leq t_1\leq \mathfrak{n}+p(\mathfrak{n},\varepsilon_0)-1\leq \mathfrak{n}+\eta \mathfrak{n}-1\leq 2\mathfrak{n}-1.
	\end{split}
\end{equation*}
Let $m'_1:=\lceil \frac{\lambda_1m_1}{\lambda_2} \rceil+1.$ Then
$
	\frac{\lambda_1m_1}{\lambda_2}+1\leq m'_1\leq \frac{\lambda_1m_1}{\lambda_2}+2.
$
Denote $M_1:=N_1(\mathfrak{n}+p(\mathfrak{n},\varepsilon_0))+t_1+m_1+m'_1+p(m_1+m_1',\varepsilon_0),$ and take $z_1\in X$ such that $f^{m_1}(z_1)=z_1^{s_1}.$ Then we have
$
	M_1=s_1+m'_1+p(m_1+m_1',\varepsilon_0).
$

Next, for any $k\geq 2$, there is  $u\geq 1$ and $1\leq v\leq u+1$ such that $k=\frac{1}{2}u(u+1)+v.$ Apply Lemma \ref{Lemma-B} to $\eta, \alpha=\frac{\overline{\tau}+\eta}{\lambda_1}+\eta, \beta=\frac{\overline{\tau}+\eta}{\lambda_1}+2\eta,$ $M=M_{k-1},$ $P=\mathfrak{n}+p(\mathfrak{n},\varepsilon_0)$ and $S_v,$  there exist $s_k\in S_v$ and $m_k,N_k,t_{k}\in\mathbb{N^{+}}$ such that
\begin{subequations}
	\begin{align}
		&s_k>\max\{s_{k-1},\overline{n}_v\}, m_k>\tilde{n}, \label{equation-A1}\\
		&\frac{\overline{\tau}+\eta}{\lambda_1}+\eta<\frac{m_k}{s_k}<\frac{\overline{\tau}+\eta}{\lambda_1}+2\eta, \label{equation-A2}\\
		&0<\frac{M_{k-1}}{s_k-M_{k-1}-m_k}<\eta, \label{equation-A4}\\
		&s_k=M_{k-1}+N_k(\mathfrak{n}+p(\mathfrak{n},\varepsilon_0))+m_k+t_k \text{ with } 0\leq t_k\leq \mathfrak{n}+p(\mathfrak{n},\varepsilon_0)-1\leq 2\mathfrak{n}-1. \label{equation-A3}
	\end{align}
\end{subequations}
Let $m'_k:=\lceil \frac{\lambda_1m_k}{\lambda_2} \rceil+1.$ Then
\begin{equation}\label{equation-AO}
	\frac{\lambda_1m_k}{\lambda_2}+1\leq m'_k\leq \frac{\lambda_1m_k}{\lambda_2}+2.
\end{equation}
Denote $M_k:=M_{k-1}+N_k(\mathfrak{n}+p(\mathfrak{n},\varepsilon_0))+t_k+m_k+m'_k+p(m_k+m'_k,\varepsilon_0)$ and take $z_k\in X$ such that
\begin{equation}\label{equation-CO}
	f^{m_k}(z_k)=z_v^{s_k}.
\end{equation}
Then by (\ref{equation-A3}) we have
$
	M_k=s_k+m'_k+p(m_k+m_k',\varepsilon_0),
$
and by (\ref{equation-A4}) we have
\begin{equation}\label{equation-AP}
	M_{k-1}<\eta s_k<\eta M_k.
\end{equation}
By (\ref{equation-A2}), we have $M_k\geq s_k+m'_k\geq s_k+\frac{\lambda_1m_k}{\lambda_2}>s_k+\frac{\lambda_1}{\lambda_2}(\frac{\overline{\tau}+\eta}{\lambda_1}+\eta)s_k.$ It follows that
\begin{equation}\label{equation-G}
	s_k\leq \frac{\lambda_2}{\lambda_2+\overline{\tau}+\eta+\eta\lambda_1}M_k.
\end{equation}
Then we have
\begin{equation}\label{equation-AQ}
	\begin{split}
		&m_{k}+m'_{k}+p(m_{k}+m'_{k},\varepsilon_0)\\
		\leq &(1+\eta)(m_{k}+m'_{k})\qquad\qquad\qquad\qquad\qquad\qquad\qquad\qquad (\text{using } (\ref{equation-A1}) \text{ and the definition } \tilde{n})\\
		\leq &(1+\eta)(m_{k}+\frac{m_{k}\lambda_1}{\lambda_2}+2)\qquad\qquad\qquad\qquad\qquad\qquad\qquad\qquad\qquad(\text{using }  (\ref{equation-AO}))\\
		\leq & (1+\eta)(1+\frac{\lambda_1}{\lambda_2})(\frac{\overline{\tau}+\eta}{\lambda_1}+2\eta)s_{k}+4\qquad\qquad\qquad\qquad\qquad\qquad\qquad(\text{using }  (\ref{equation-A2}))\\
		\leq &(1+\eta)(1+\frac{\lambda_1}{\lambda_2})(\frac{\overline{\tau}+\eta}{\lambda_1}+2\eta)\frac{\lambda_2}{\lambda_2+\overline{\tau}+\eta+\eta\lambda_1}M_{k}+4\qquad\qquad\qquad(\text{using }  (\ref{equation-G})).
	\end{split}
\end{equation}

Now, we define a sequences $\{n_{j}'\}_{j=1}^{+\infty}$ as
$$
n_{j}'=\left\{\begin{array}{ll}
	\mathfrak{n}+p(\mathfrak{n},\varepsilon_0), & \text { for } j=N_{1}+N_{2}+\dots+N_{k-1}+k-1+q\ \mathrm{with}\  1\leq q \leq N_{k}-1,\\
	\mathfrak{n}+p(\mathfrak{n},\varepsilon_0)+t_k, & \text { for } j=N_{1}+N_{2}+\dots+N_{k-1}+N_k+k-1,\\
	m_k+m'_k+p(m_k+m'_k,\varepsilon_0), & \text { for } j=N_{1}+N_{2}+\dots+N_{k}+k,
\end{array}\right.
$$
where $N_0=0.$
Let $M'_0:=0$ and $M'_{j}:=\sum_{i=1}^{j}n_i'$ for any $j\geq 1.$ Then we have
\begin{equation}\label{equation-AM}
	M_k=M'_{N_1+\dots+N_k+k}=M'_{N_1+\dots+N_{k-1}+k-1}+N_k(\mathfrak{n}+p(\mathfrak{n},\varepsilon_0))+t_k+m_k+m'_k+p(m_k+m'_k,\varepsilon_0).
\end{equation}
Denote $\Gamma_{k}:=\Gamma_{\mathfrak{n}}^{N_k}=\Gamma_{\mathfrak{n}}\times \cdots \times\Gamma_{\mathfrak{n}}=\{(y_1^k,\dots,y_{N_k}^k):y_j^k\in\Gamma_{\mathfrak{n}}\text{ for } 1\leq j\leq N_k\}.$ Then $|\Gamma_{k}|=|\Gamma_{\mathfrak{n}}|^{N_k}.$

The  points $y_j^k$ with length $\mathfrak{n}$, denoted by $(y_j^k,\mathfrak{n}),$ and points $z_k$ with length $m_k+m'_k,$ denoted by $(z_k, m_k+m'_k),$  are the desired orbit segments as shown in Figure \ref{Fig-2}. We will use them to construct a Cantor subset $G$ of  $\cap_{j=1}^{\infty}\mathfrak{S}(f,\varphi_j,\mathcal{Z}_j,S_j)$.
\newline

\noindent {\bf Step 3: construction of a Cantor set $G$}

Now, we construct  a Cantor subset $G$ of  $\cap_{j=1}^{\infty}\mathfrak{S}(f,\varphi_j,\mathcal{Z}_j,S_j)$.
Let $k\geq 1.$ For any $(\mathbf{y}_1,\dots,\mathbf{y}_{k})\in  \Gamma_1\times\dots\times\Gamma_k$ with $\mathbf{y}_i=(y_1^i,\dots,y_{N_i}^i)\in \Gamma_i,$
denote $$G(\mathbf{y}_i):=\left(\cap_{j=1}^{N_{i}}f^{-M'_{N_0+\dots+N_{i-1}+i+j-2}}\overline{B}_{\mathfrak{n}}\left(y_j^i,  \varepsilon_0,\lambda_1,\lambda_2,f\right)\right)\bigcap f^{-M'_{N_1+\dots+N_i+i-1}}\overline{B}_{m_i+m'_i}\left(z_i,  \varepsilon_0,\lambda_1,\lambda_2,f\right).$$
By exponential non-uniform specification property  of $(X,f),$
\begin{equation*}
	\begin{split}
		G(\mathbf{y}_1,\dots,\mathbf{y}_k)
		:=\bigcap_{i=1}^{k} G(\mathbf{y}_i)
	\end{split}
\end{equation*}
is a non-empty closed set.
In Figure \ref{Fig-2} we give an  illustration of canonical points in $G(\mathbf{y}_1,\dots,\mathbf{y}_k).$ The points in $G(\mathbf{y}_1,\dots,\mathbf{y}_k)$ trace exponentially the orbit segments $(y_1^1,\mathfrak{n}),\dots,(y_{N_1}^1,\mathfrak{n}),$ $(z_1, m_1+m'_1),\dots $ $(y_1^k,\mathfrak{n}),\dots,(y_{N_k}^k,\mathfrak{n}),$ $(z_k, m_k+m'_k)$  one by one, and spend time  $p(\mathfrak{n},\varepsilon_0)$ or $p(m_k+m'_k,\varepsilon_0)$ between consecutive segments.
	\begin{figure}\caption{Illustration of canonical points in $G(\mathbf{y}_1,\dots,\mathbf{y}_k)$}\label{Fig-2}
		\centering
		\includegraphics[width=16cm]{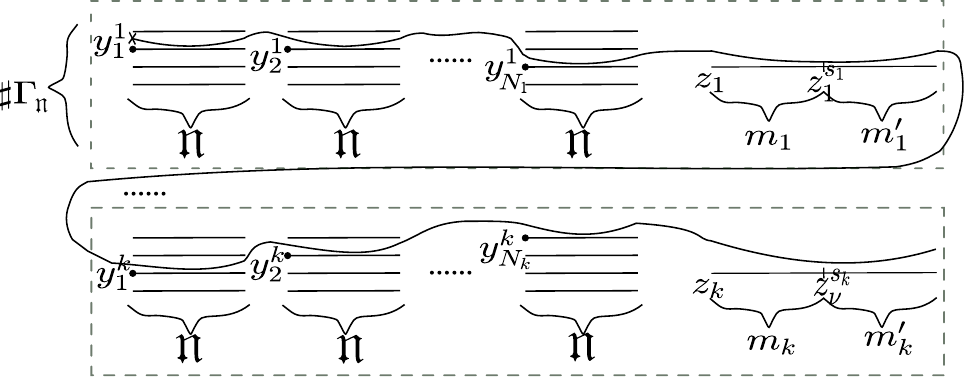}
	\end{figure}

Let $$G_k:=\bigcup_{(\mathbf{y}_1,\dots,\mathbf{y}_k)\in \Gamma_1\times\dots\times\Gamma_k}G(\mathbf{y}_1,\dots,\mathbf{y}_k),$$ then we have $G_{k+1}\subset G_{k}$.
Let
\begin{equation*}
	\qquad\qquad\qquad\qquad\qquad\qquad G:=\bigcap_{k\geq 1}G_k\qquad\qquad\qquad\qquad\qquad\qquad(\star)
\end{equation*}
Then $G$ is a non-empty closed set.
\newline

\noindent {\bf Step 4: verification of $G\subset \cap_{j=1}^{\infty}\mathfrak{S}(f,\varphi_j,\mathcal{Z}_j,S_j)$}

As shown in Figure \ref{Fig-2}, every point $z\in G$ trace exponentially every orbit segments $(z_k, m_k+m'_k).$  Recall  from (\ref{equation-CO}) that $f^{m_k}(z_k)=z_v^{s_k}$, exponentia tracing implies the orbit of $z$ is in fact very close to $z_v^{s_k}$ and thus visit the target ball $B(z_v^{s_k},\varphi_v(s_k)).$

Let  $z\in G$ and $j\geq 1.$ Then for any $u\geq j,$ we have $$z\in G\subset G_{k_{u,j}}\subset f^{-M'_{N_1+\dots+N_{k_{u,j}}+k_{u,j}-1}}\overline{B}_{ m_{k_{u,j}}+m'_{k_{u,j}}}\left(z_{k_{u,j}},  \varepsilon_0,\lambda_1,\lambda_2,f\right)$$
where $k_{u,j}:=\frac{1}{2}u(u+1)+j.$
By (\ref{equation-AO}), we have
$(m'_{k_{u,j}}-1)\lambda_2\geq m_{k_{u,j}}\lambda_1.$
Then by the definition of  exponential Bowen ball, we have $$d(f^{M'_{N_1+\dots+N_{k_{u,j}}+k_{u,j}-1}+m_{k_{u,j}}}(z),f^{m_{k_{u,j}}}(z_{k_{u,j}}))<\varepsilon_0 e^{-\min\{m_{k_{u,j}}\lambda_1,(m_{k_{u,j}}+m'_{k_{u,j}}-1-m_{k_{u,j}})\lambda_2\} }   \leq \varepsilon_0e^{-m_{k_{u,j}}\lambda_1}.$$
Note that $M'_{N_1+\dots+N_{k_{u,j}}+k_{u,j}-1}=M_{k_{u,j}-1}+(\mathfrak{n}+p(\mathfrak{n},\varepsilon))N_{k_{u,j}}+t_{u,j}.$ Then we have $M'_{N_1+\dots+N_{k_{u,j}}+k_{u,j}-1}+m_{k_{u,j}}=s_{k_{u,j}}\in S_j$ by (\ref{equation-A3}). Thus by $f^{m_{k_{u,j}}}(z_{k_{u,j}})=z_j^{s_{k_{u,j}}},$ we have
\begin{equation*}
	\begin{split}
			d(f^{s_{k_{u,j}}}(z),z_j^{s_{k_{u,j}}})&<\varepsilon_0  e^{-m_{k_{u,j}}\lambda_1}.\\
		&<\varepsilon_0e^{-s_{k_{u,j}}(\overline{\tau}+\eta)-s_{k_{u,j}}\lambda_1\eta}\qquad\qquad\qquad\qquad\qquad\qquad\qquad(\text{using }  (\ref{equation-A2})) \\
		&\leq e^{-s_{k_{u,j}}(\overline{\tau}_j+\eta)}<\varphi_j(s_{k_{u,j}}).\qquad\qquad\qquad\qquad(\text{using }  (\ref{equation-A1}) \text{ and the definition of } \overline{n}_j)
	\end{split}
\end{equation*}
By the arbitrariness of $u$,  we obtain $z\in \mathfrak{S}(f,\varphi_j,\mathcal{Z}_j,S_j)$ and thus $G\subset \cap_{j=1}^{\infty}\mathfrak{S}(f,\varphi_j,\mathcal{Z}_j,S_j).$
\newline

\noindent {\bf Step 5: estimate of topological entropy of $\cap_{j=1}^{\infty}\mathfrak{S}(f,\varphi_j,\mathcal{Z}_j,S_j)$}

Now we construct a probability measure $\mu$ supported on  $G$,  estimate  the $\mu$ measure of dynamical ball $B_n(z, \varepsilon,f)$ and obtain the topological entropy of $\cap_{j=1}^{\infty}\mathfrak{S}(f,\varphi_j,\mathcal{Z}_j,S_j)$ by the variational principles for topological entropies of subsets.

Let $\mathbf{y}_i^1=(y_1^{i,1},\dots,y_{N_i}^{i,1}),\ \mathbf{y}_i^2=(y_1^{i,2},\dots,y_{N_i}^{i,2})\in \Gamma_i$ for each $1\leq i\leq k$ with $(\mathbf{y}_1^1,\dots,\mathbf{y}_k^1)\neq (\mathbf{y}_1^2,\dots,\mathbf{y}_k^2).$ If $y_{\tilde{j}}^{\tilde{i},1}\neq y_{\tilde{j}}^{\tilde{i},2}$ for some $1\leq \tilde{i}\leq k$ and some $1\leq \tilde{j}\leq N_{\tilde{i}},$ then $y_{\tilde{j}}^{\tilde{i},1}$ and $y_{\tilde{j}}^{\tilde{i},2}$ are $\left(\delta^*, \mathfrak{n}, 3\varepsilon^*\right)$-separated with respect to $f.$ This implies
\begin{equation}\label{equation-E}
	\sharp \{0\leq j\leq \mathfrak{n}-1:d(f^{j}(y_{\tilde{j}}^{\tilde{i},1}),  f^{j}(y_{\tilde{j}}^{\tilde{i},2}))>3\varepsilon^*\}\geq\delta^* \mathfrak{n}.
\end{equation}
Take $z^1\in G(\mathbf{y}_1^1,\dots,\mathbf{y}_k^1)$ and $z^2\in G(\mathbf{y}_1^2,\dots,\mathbf{y}_k^2),$ we have $$f^{M'_{N_0+\dots+N_{\tilde{i}-1}+\tilde{i}+\tilde{j}-2}}(z^1)\in \overline{B}_{\mathfrak{n}}\left(y_{\tilde{j}}^{\tilde{i},1},  \varepsilon_0,\lambda_1,\lambda_2,f\right),$$
$$f^{M'_{N_0+\dots+N_{\tilde{i}-1}+\tilde{i}+\tilde{j}-2}}(z^2)\in \overline{B}_{\mathfrak{n}}\left(y_{\tilde{j}}^{\tilde{i},2},  \varepsilon_0,\lambda_1,\lambda_2,f\right).$$
Let $\mathfrak{N}_1:=\max\{\lceil \frac{\ln \varepsilon_0-\ln \varepsilon^*}{\lambda_1}\rceil,0\},$ $\mathfrak{N}_2:=\max\{\lceil \frac{\ln \varepsilon_0-\ln \varepsilon^*}{\lambda_2}\rceil,0\}.$ Then for any $x\in X$ we have $$\overline{B}_{\mathfrak{n}}\left(x,  \varepsilon_0,\lambda_1,\lambda_2,f\right)\subset f^{-\mathfrak{N}_1}\overline{B}_{\mathfrak{n}-\mathfrak{N}_1-\mathfrak{N}_2}\left(f^{\mathfrak{N}_1}(x),  \varepsilon^*,f\right).$$
Thus $$f^{M'_{N_0+\dots+N_{\tilde{i}-1}+\tilde{i}+\tilde{j}-2}}(z^1)\in f^{-\mathfrak{N}_1}\overline{B}_{\mathfrak{n}-\mathfrak{N}_1-\mathfrak{N}_2}\left(f^{\mathfrak{N}_1}(y_{\tilde{j}}^{\tilde{i},1}),  \varepsilon^*,f\right),$$
$$f^{M'_{N_0+\dots+N_{\tilde{i}-1}+\tilde{i}+\tilde{j}-2}}(z^2)\in f^{-\mathfrak{N}_1}\overline{B}_{\mathfrak{n}-\mathfrak{N}_1-\mathfrak{N}_2}\left(f^{\mathfrak{N}_1}(y_{\tilde{j}}^{\tilde{i},2}),  \varepsilon^*,f\right).$$
Recall $\mathfrak{n}\delta^*>\max\{\lceil \frac{\ln \varepsilon_0-\ln \varepsilon^*}{\lambda_1}\rceil+\lceil \frac{\ln \varepsilon_0-\ln \varepsilon^*}{\lambda_2}\rceil,0\}+1\geq\mathfrak{N}_1+\mathfrak{N}_2+1.$  Combining with (\ref{equation-E}), there is $\mathfrak{N}_1\leq t\leq \mathfrak{n}-\mathfrak{N}_2$ such that
 $d(f^{M'_{N_0+\dots+N_{\tilde{i}-1}+\tilde{i}+\tilde{j}-2}+t}(z^1),f^{M'_{N_0+\dots+N_{\tilde{i}-1}+\tilde{i}+\tilde{j}-2}+t}(z^2))>\varepsilon^*.$ So
\begin{equation}\label{equation-F}
	z^1\text{ and }z^2\text{ are }(M'_{N_0+\dots+N_{\tilde{i}-1}+\tilde{i}+\tilde{j}-1},\eps^*)\text{-separated with respect to } f.
\end{equation}
That is to say, the orbits of $z^1\text{ and }z^2$ separate in the place where it happens that $y_{\tilde{j}}^{\tilde{i},1}\neq y_{\tilde{j}}^{\tilde{i},2}.$

For any $k\geq 1$ and  $(\mathbf{y}_1,\dots,\mathbf{y}_{k})\in  \Gamma_1\times\dots\times\Gamma_k,$ we choose $z(\mathbf{y}_1,\dots,\mathbf{y}_{k})\in G(\mathbf{y}_1,\dots,\mathbf{y}_k).$ Then by (\ref{equation-F}), $F_{k}:=\{z(\mathbf{y}_1,\dots,\mathbf{y}_{k}):(\mathbf{y}_1,\dots,\mathbf{y}_{k})\in  \Gamma_1\times\dots\times\Gamma_k\}$ is a $(M'_{N_1+\dots+N_{k}+k},\eps^*)$-separated set with respect to $f$ and $\sharp F_k=\sharp \Gamma_1\times\dots\times\sharp\Gamma_k=(\sharp\Gamma_{\mathfrak{n}})^{N_1+\dots+N_k}$.
Define $\mu_{k}:=\frac{1}{\sharp F_k}\sum_{z\in F_{k}}\delta_{z}.$ Suppose $\mu=\lim\limits_{n\to\infty}\mu_{k_s}$ for some $k_s\to \infty$. For any fixed $p\in\mathbb{N^{+}}$ and all $p'\geq 0$,
since $F_{p+p'}\subseteq G_{p+p'},$ one has $\mu_{p+p'}(G_{p+p'})=1.$ Combining with $G_{p+p'}\subseteq G_{p},$ we have $\mu_{p+p'}(G_{p})=1.$
Then $\mu(G_{p})\geq \limsup\limits_{n\to\infty}\mu_{k_s}(G_{p})=1$. It follows that
$\mu(G)=\lim\limits_{p\to \infty}\mu(G_p)=1.$

Next, we estimate  the $\mu$ measure of dynamical ball $B_n(z , \varepsilon,f).$
Denote \begin{equation}\label{equation-DR}
	\overline{N}:=\max\{M'_{N_1+N_2+3},\frac{6\mathfrak{n} +2}{\eta},\frac{4\mathfrak{n} +3}{\eta}\}.
\end{equation}
Fix $z\in G$ and $n\geq \overline{N},$ there exists $i'\geq 3$ and $1\leq j'\leq N_{i'}+1$ such that
\begin{equation}\label{equation-AL}
	M'_{N_1+\dots+N_{i'-1}+i'-1+j'}\leq n\leq M'_{N_1+\dots+N_{i'-1}+i'-1+j'+1}.
\end{equation}
Since $z\in G\subseteq G_{i'},$ then there exists $\mathbf{y}_{i}^0=(y_1^{i,0},\dots,y_{N_{i}}^{i,0})\in \Gamma_{i}$ for each $1\leq i\leq i'$ such that
$z\in G(\mathbf{y}_1^0,\dots,\mathbf{y}_{i'}^0).$
Note that for any $k_s\geq i'$ and $(\mathbf{y}_1,\dots,\mathbf{y}_{k_s})\in  \Gamma_1\times\dots\times\Gamma_{k_s},$ if $(y_1^{1},\dots,y_{N_{1}}^{1},\dots,y_1^{i'},\dots,y_{j''}^{i'})\neq (y_1^{1,0},\dots,y_{N_{1}}^{1,0},\dots,y_1^{i',0},\dots,y_{j''}^{i',0})$ where $j''=\min\{j',N_{i'}\},$ then by (\ref{equation-F}) $$z_{k_s}(\mathbf{y}_1,\dots,\mathbf{y}_{k_s})\not\in B_{M'_{N_1+\dots+N_{i'-1}+i'-1+j''}}(z,\eps^*,f).$$  Then we have
\begin{equation}\label{equation-AG}
	\begin{split}
		\mu_{k_{s}}(B_{n}(z,\eps^*,f))\leq&\mu_{k_{s}}(B_{M'_{N_1+\dots+N_{i'-1}+i'-1+j''}}(z,\eps^*,f))\\
		\leq & \frac{1}{\sharp\Gamma_{1}\times\dots\times\sharp \Gamma_{k_s}}\times((\sharp\Gamma_{\mathfrak{n}})^{N_{i'}-j''}\times \sharp\Gamma_{i'+1}\times\dots\times\sharp\Gamma_{k_s})\\
		=&\frac{1}{\sharp\Gamma_{1}\times\dots\times\sharp \Gamma_{i'-1}\times (\sharp\Gamma_{\mathfrak{n}})^{j''}}\\
		=&\frac{1}{(\sharp\Gamma_{\mathfrak{n}})^{N_1+\dots+N_{i'-1}+j''}}\\
		\leq& e^{-(N_1+\dots+N_{i'-1}+j'')\mathfrak{n} (h_{top}(f,X)-\eta)}.
	\end{split}
\end{equation}

If $1\leq j'\leq N_{i'}-1,$ then  $j''=j'$ and we have
\begin{equation*}
	\begin{split}
		&n-(N_1+\dots+N_{i'-1}+j'')\mathfrak{n}\\
		\leq& M'_{N_1+\dots+N_{i'-1}+i'+j'}-(N_1+\dots+N_{i'-1}+j'')\mathfrak{n} \qquad\qquad\qquad\qquad\qquad (\text{using }  (\ref{equation-AL}))\\
		\leq& M'_{N_1+\dots+N_{i'-1}+i'+j'}-(N_{i'-1}+j')\mathfrak{n} \\
		= & M'_{N_1+\dots+N_{i'-2}+i'-2}+(N_{i'-1}+j')p(\mathfrak{n},\varepsilon_0)+\mathfrak{n}+p(\mathfrak{n},\varepsilon_0)+t_{i'-1}+t_{i'}\\
		&+m_{i'-1}+m'_{i'-1}+p(m_{i'-1}+m'_{i'-1},\varepsilon_0)\qquad\qquad\qquad\qquad\qquad\qquad\quad(\text{using }  (\ref{equation-AM}))\\
		\leq &M_{i'-2}+(N_{i'-1}+j')p(\mathfrak{n},\varepsilon_0)+2\mathfrak{n}+2\mathfrak{n}-1+2\mathfrak{n}-1\qquad\qquad\qquad(\text{using } (\ref{equation-AM}) \text{ and } (\ref{equation-A3}))\\
		&+m_{i'-1}+m'_{i'-1}+p(m_{i'-1}+m'_{i'-1},\varepsilon_0)\\
		\leq &\eta M_{i'-1}   +(N_{i'-1}+j')p(\mathfrak{n},\varepsilon_0)+6\mathfrak{n} -2     \qquad\qquad\qquad\qquad\qquad\qquad\qquad(\text{using }  (\ref{equation-AP}))\\
		&+m_{i'-1}+m'_{i'-1}+p(m_{i'-1}+m'_{i'-1},\varepsilon_0)\\
		\leq & \eta n+(N_{i'-1}+j')\eta\mathfrak{n}+6\mathfrak{n} -2 \qquad\qquad\qquad\qquad\qquad\qquad(\text{using } (\ref{equation-AL}) \text{ and the definition } \tilde{n})\\
		&+m_{i'-1}+m'_{i'-1}+p(m_{i'-1}+m'_{i'-1},\varepsilon_0)\\
		\leq &\eta n+\eta n+6\mathfrak{n} +2+(1+\eta)(1+\frac{\lambda_1}{\lambda_2})(\frac{\overline{\tau}+\eta}{\lambda_1}+2\eta)\frac{\lambda_2}{\lambda_2+\overline{\tau}+\eta+\eta\lambda_1}M_{i'-1}\quad(\text{using }  (\ref{equation-AQ}))\\
		\leq & 3\eta n+(1+\eta)(1+\frac{\lambda_1}{\lambda_2})(\frac{\overline{\tau}+\eta}{\lambda_1}+2\eta)\frac{\lambda_2}{\lambda_2+\overline{\tau}+\eta+\eta\lambda_1}n. \qquad\qquad\qquad\quad(\text{using }  (\ref{equation-DR}))
	\end{split}
\end{equation*}
It follows that
\begin{equation}\label{equation-H1}
	(N_1+\dots+N_{i'-1}+j'')\mathfrak{n}/n\geq 1-3\eta -(1+\eta)(1+\frac{\lambda_1}{\lambda_2})(\frac{\overline{\tau}+\eta}{\lambda_1}+2\eta)\frac{\lambda_2}{\lambda_2+\overline{\tau}+\eta+\eta\lambda_1}.
\end{equation}
If $j'= N_{i'} $ or $N_{i'}+1,$ then  $j''=N_{i'}$  and
$$n\leq M'_{N_1+\dots+N_{i'}+i'+1}=M_{i'}+\mathfrak{n}+p(\mathfrak{n},\varepsilon_0).$$
Thus we have
\begin{equation*}
	\begin{split}
		&1-(N_1+\dots+N_{i'-1}+j'')\mathfrak{n}/n\\
		\leq& 1-j''\mathfrak{n}/M'_{N_1+\dots+N_{i'}+i'+1}\qquad\qquad\qquad\qquad\qquad\qquad\qquad\qquad\qquad\qquad\qquad\quad (\text{using }  (\ref{equation-AL}))\\
		=  &\frac{M_{i'}+\mathfrak{n}+p(\mathfrak{n},\varepsilon_0)-N_{i'}\mathfrak{n}}{M_{i'}+\mathfrak{n}+p(\mathfrak{n},\varepsilon_0)}\qquad\qquad\qquad\qquad\qquad\qquad\qquad\qquad\qquad\qquad\qquad\quad(\text{using }  (\ref{equation-AM}))\\
		= &\frac{\mathfrak{n}+p(\mathfrak{n},\varepsilon_0)+M_{i'-1}+N_{i'}p(\mathfrak{n},\varepsilon_0)+t_{i'}+m_{i'}+m'_{i'}+p(m_{i'}+m'_{i'},\varepsilon_0)}{M_{i'}+\mathfrak{n}+p(\mathfrak{n},\varepsilon_0)}\qquad\quad\quad(\text{using }  (\ref{equation-AM}))\\
		\leq &\frac{2\mathfrak{n}+\eta M_{i'}+N_{i'}\eta \mathfrak{n}+2\mathfrak{n}-1+m_{i'}+m'_{i'}+p(m_{i'}+m'_{i'},\varepsilon_0)}{M_{i'}+\mathfrak{n}+p(\mathfrak{n},\varepsilon_0)}\qquad\qquad\quad(\text{using }  (\ref{equation-AP}) \text{ and } (\ref{equation-A3}))\\
		\leq &\frac{\eta M_{i'}+\eta M_{i'}+4\mathfrak{n}-1+m_{i'}+m'_{i'}+p(m_{i'}+m'_{i'},\varepsilon_0)}{M_{i'}+\mathfrak{n}+p(\mathfrak{n},\varepsilon_0)}\qquad\qquad\qquad\qquad(\text{using }  (\ref{equation-AP}) \text{ and } (\ref{equation-A3}))\\
		\leq &2\eta+\frac{4\mathfrak{n}-1}{M_{i'}+\mathfrak{n}+p(\mathfrak{n},\varepsilon_0)}+\frac{m_{i'}+m'_{i'}+p(m_{i'}+m'_{i'},\varepsilon_0)}{M_{i'}+\mathfrak{n}+p(\mathfrak{n},\varepsilon_0)}.\\
		\leq &2\eta+\frac{4\mathfrak{n}+3}{M_{i'}+\mathfrak{n}+p(\mathfrak{n},\varepsilon_0)}+(1+\eta)(1+\frac{\lambda_1}{\lambda_2})(\frac{\overline{\tau}+\eta}{\lambda_1}+2\eta)\frac{\lambda_2}{\lambda_2+\overline{\tau}+\eta+\eta\lambda_1}\qquad\quad (\text{using }  (\ref{equation-AQ}))\\
		\leq &2\eta+\frac{4\mathfrak{n}+3}{n}+(1+\eta)(1+\frac{\lambda_1}{\lambda_2})(\frac{\overline{\tau}+\eta}{\lambda_1}+2\eta)\frac{\lambda_2}{\lambda_2+\overline{\tau}+\eta+\eta\lambda_1}\\
		\leq &3\eta+(1+\eta)(1+\frac{\lambda_1}{\lambda_2})(\frac{\overline{\tau}+\eta}{\lambda_1}+2\eta)\frac{\lambda_2}{\lambda_2+\overline{\tau}+\eta+\eta\lambda_1}. \quad\qquad\qquad\qquad\qquad\qquad\quad(\text{using }  (\ref{equation-DR}))
	\end{split}
\end{equation*}
It follows that
\begin{equation}\label{equation-H2}
	(N_1+\dots+N_{i'-1}+j'')\mathfrak{n}/n\geq 1-3\eta -(1+\eta)(1+\frac{\lambda_1}{\lambda_2})(\frac{\overline{\tau}+\eta}{\lambda_1}+2\eta)\frac{\lambda_2}{\lambda_2+\overline{\tau}+\eta+\eta\lambda_1}.
\end{equation}
By (\ref{equation-AG}), (\ref{equation-H1}) and (\ref{equation-H2}),  we have
\begin{equation}\label{equation-BA}
	\mu_{k_{s}}(B_{n}(z,\eps^*,f))
		\leq e^{-n(1-3\eta -(1+\eta)(1+\frac{\lambda_1}{\lambda_2})(\frac{\overline{\tau}+\eta}{\lambda_1}+2\eta)\frac{\lambda_2}{\lambda_2+\overline{\tau}+\eta+\eta\lambda_1})(h_{top}(f,X)-\eta)}.
\end{equation}
Then
\begin{equation*}
	\begin{split}
	\mu(B_{n}(z,\eps^*,f))
		\leq \liminf_{s\rightarrow \infty}\mu_{k_s}(B_{n}(z,\eps^*,f))
		\leq e^{-n(1-3\eta -(1+\eta)(1+\frac{\lambda_1}{\lambda_2})(\frac{\overline{\tau}+\eta}{\lambda_1}+2\eta)\frac{\lambda_2}{\lambda_2+\overline{\tau}+\eta+\eta\lambda_1})(h_{top}(f,X)-\eta)}.
	\end{split}
\end{equation*}
This implies $$h_{top}(f,G)\geq (1-3\eta -(1+\eta)(1+\frac{\lambda_1}{\lambda_2})(\frac{\overline{\tau}+\eta}{\lambda_1}+2\eta)\frac{\lambda_2}{\lambda_2+\overline{\tau}+\eta+\eta\lambda_1})(h_{top}(f,X)-\eta)$$ by Corollary \ref{Cor-aa}.
Let $\eta\to 0$, we eventually have (\ref{equa-thm-1}).
\newline

\noindent {\bf Step 6: estimate of Hausdorff dimension of $\cap_{j=1}^{\infty}\mathfrak{S}(f,\varphi_j,\mathcal{Z}_j,S_j)$ for $L$-Lipschitz map}

Next, we estimate  the $\mu$ measure of ball $B(z,\varepsilon)$ and obtain  Hausdorff dimension of $\cap_{j=1}^{\infty}\mathfrak{S}(f,\varphi_j,\mathcal{Z}_j,S_j)$ by  the mass distribution principle.
Assume that $f:X\to X$ is  a $L$-Lipschitz map. For any $\eta>0,$ denote
$$\kappa_\eta:=(1-3\eta -(1+\eta)(1+\frac{\lambda_1}{\lambda_2})(\frac{\overline{\tau}+\eta}{\lambda_1}+2\eta)\frac{\lambda_2}{\lambda_2+\overline{\tau}+\eta+\eta\lambda_1})(h_{top}(f,X)-\eta).$$ Let $\eta>0$ small enough such that $\kappa_\eta>0.$
We denote $$r_0:=\frac{\varepsilon^{*}}{L^{\overline{N}+1}}, c:=e^{(1-\frac{\ln\varepsilon^*}{\ln L})\kappa_\eta}.$$
For any $0<r\leq r_0$, denote $\psi(r):=\max\{m\in\mathbb{Z}: L^mr<\varepsilon^*\}.$
Then
\begin{equation}\label{equation-BB}
	\psi(r)\geq \frac{\ln \varepsilon^*-\ln r}{\ln L}-1\geq \frac{\ln \varepsilon^*-\ln r_0}{\ln L}-1=\overline{N},
\end{equation}
and by the definition of $L$, for any $x\in X$ we have
$$B(x,r)\subset B_{\psi(r)}(x,\varepsilon^*,f).$$
Combining with  (\ref{equation-BA}), for any $z\in G$ we obtain
\begin{equation}\label{equation-BC}
		\mu(B(z,r))\leq \mu(B_{\psi(r)}(z,\varepsilon^*,f))
		\leq e^{-\psi(r)\kappa_\eta}.
\end{equation}
By (\ref{equation-BB}), we have
$$e^{-\psi(r)}\leq e^{\frac{\ln r-\ln \varepsilon^*}{\ln L}+1}=r^{\frac{1}{\ln L}}e^{-\frac{\ln \varepsilon^*}{\ln L}+1}.$$
So by (\ref{equation-BC}) we have $\mu(B(z,r))\leq e^{(1-\frac{\ln \varepsilon^*}{\ln L})\kappa_\eta}r^{\frac{\kappa_\eta}{\ln L}}=cr^{\frac{\kappa_\eta}{\ln L}}.$
Thus by Lemma \ref{Lem-mass}
$$\operatorname{dim}_H \cap_{j=1}^{\infty}\mathfrak{S}(f,\varphi_j,\mathcal{Z}_j,S_j)\geq \operatorname{dim}_HG\geq \frac{\kappa_\eta}{\ln L} .$$
Let $\eta\to 0$, we eventually have (\ref{equa-thm-2}).
\newline

\noindent {\bf Step 7: estimate of  Hausdorff dimension of $\cap_{j=1}^{\infty}\mathfrak{S}(f,\varphi_j,\mathcal{Z}_j,S_j)$ for $(L_1,L_2)$-bi-Lipschitz homeomorphism}

Assume that $f:X\to X$ is  a $(L_1,L_2)$-bi-Lipschitz homeomorphism. We will construct  a new subset $\tilde{G}$ of  $\cap_{j=1}^{\infty}\mathfrak{S}(f,\varphi_j,\mathcal{Z}_j,S_j)$ to give a bigger lower bound for Hausdorff dimension. For every point in  $\tilde{G},$ its forward orbit has the same behavior as points in  $G$.  However, we require that its backward orbit also traces some orbit segments, so its backward orbit will contribute to the  Hausdorff dimension.

For any integer $j\leq -1$, we define
$
n_{j}':=
	\mathfrak{n}+p(\mathfrak{n},\varepsilon_0).
$
Let $M'_{j}:=\sum_{i=-1}^{j}n_i'$ for any $j\leq -1.$
For any $k\geq 1$, denote $$\Gamma_{-k}:=\Gamma_{k}=\Gamma_{\mathfrak{n}}^{N_k}=\Gamma_{\mathfrak{n}}\times \cdots \times\Gamma_{\mathfrak{n}}=\{(y_1^k,\dots,y_{N_k}^k):y_j^k\in\Gamma_{\mathfrak{n}}\text{ for } 1\leq j\leq N_k\}.$$ Then $|\Gamma_{-k}|=|\Gamma_{\mathfrak{n}}|^{N_k}.$ Denote $N_{-k}:=N_k.$

Now, we construct  a new subset of $\cap_{j=1}^{\infty}\mathfrak{S}(f,\varphi_j,\mathcal{Z}_j,S_j)$.
Let $k\geq 1.$ For any $(\mathbf{y}_1,\dots,\mathbf{y}_{k})\in  \Gamma_1\times\dots\times\Gamma_k$  and any $(\mathbf{y}_{-1},\dots,\mathbf{y}_{-k})\in  \Gamma_{-1}\times\dots\times\Gamma_{-k}$ with $\mathbf{y}_i=(y_1^i,\dots,y_{N_i}^i)\in \Gamma_i,$
denote $$G(\mathbf{y}_i):=\left(\cap_{j=1}^{N_{i}}f^{-M'_{N_0+\dots+N_{i-1}+i+j-2}}\overline{B}_{\mathfrak{n}}\left(y_j^i,  \varepsilon_0,\lambda_1,\lambda_2,f\right)\right)\bigcap f^{-M'_{N_1+\dots+N_i+i-1}}\overline{B}_{m_i+m'_i}\left(z_i,  \varepsilon_0,\lambda_1,\lambda_2,f\right),$$
$$G(\mathbf{y}_{-i}):=\cap_{j=1}^{N_{i}}f^{M'_{-N_{0}-\dots-N_{-i+1}-j}}\overline{B}_{\mathfrak{n}}\left(y_j^{-i},  \varepsilon_0,\lambda_1,\lambda_2,f\right),$$
where $1\leq i\leq k.$
 By exponential non-uniform specification property  of $(X,f),$
\begin{equation*}
	\begin{split}
		\tilde{G}(\mathbf{y}_{-k},\dots,\mathbf{y}_{-1},\mathbf{y}_1,\dots,\mathbf{y}_k):=\bigcap_{i=1}^{k}(G(\mathbf{y}_i)\cap G(\mathbf{y}_{-i}))
	\end{split}
\end{equation*}
is a non-empty closed set and $\tilde{G}(\mathbf{y}_{-k},\dots,\mathbf{y}_{-1},\mathbf{y}_1,\dots,\mathbf{y}_k)\subset G(\mathbf{y}_1,\dots,\mathbf{y}_k)$.
Let $$\tilde{G}_k:=\bigcup_{(\mathbf{y}_{-k},\dots,\mathbf{y}_{-1},\mathbf{y}_1,\dots,\mathbf{y}_k)\in \Gamma_{-k}\times\dots\times\Gamma_k}\tilde{G}(\mathbf{y}_{-k},\dots,\mathbf{y}_{-1},\mathbf{y}_1,\dots,\mathbf{y}_k),$$ then we have $\tilde{G}_{k+1}\subset \tilde{G}_{k}\subset G_k$.
Let
\begin{equation*}
	\tilde{G}:=\bigcap_{k\geq 1}\tilde{G}_k.
\end{equation*}
Then $\tilde{G}$ is a non-empty closed set and $\tilde{G}\subset G$. So by Step 4 we have $\tilde{G}\subset G\subset \cap_{j=1}^{\infty}\mathfrak{S}(f,\varphi_j,\mathcal{Z}_j,S_j)$.

Next, we construct a probability measure $\tilde{\mu}$ supported on  $\tilde{G}$,  estimate  the $\tilde{\mu}$ measure of ball $B(z,\varepsilon)$ and obtain  Hausdorff dimension of $\cap_{j=1}^{\infty}\mathfrak{S}(f,\varphi_j,\mathcal{Z}_j,S_j)$ by  the mass distribution principle.

Let $\mathbf{y}_i^1=(y_1^{i,1},\dots,y_{N_i}^{i,1}),\ \mathbf{y}_i^2=(y_1^{i,2},\dots,y_{N_i}^{i,2})\in \Gamma_i$ for each $i\in[-k,-1]\cup [1,k]$ with $$(\mathbf{y}_{-k}^1,\dots,\mathbf{y}_{-1}^1,\mathbf{y}_1^1,\dots,\mathbf{y}_k^1)\neq (\mathbf{y}_{-k}^2,\dots,\mathbf{y}_{-1}^2,\mathbf{y}_1^2,\dots,\mathbf{y}_k^2).$$
Take $$z^1\in \tilde{G}_k(\mathbf{y}_{-k}^1,\dots,\mathbf{y}_{-1}^1,\mathbf{y}_1^1,\dots,\mathbf{y}_k^1), z^2\in \tilde{G}_k(\mathbf{y}_{-k}^2,\dots,\mathbf{y}_{-1}^2,\mathbf{y}_1^2,\dots,\mathbf{y}_k^2).$$
Similar as Step 5,  the following holds: if  $1\leq \tilde{i},\tilde{i}_{-1}\leq k$ and $1\leq \tilde{j}\leq N_{\tilde{i}},$ $1\leq \tilde{j}_{-1}\leq N_{-\tilde{i}_{-1}}$ satisfy
\begin{equation*}
	\begin{split}
		&(y^{-\tilde{i}_{-1},1}_{\tilde{j}_{-1}},\dots,y^{-\tilde{i}_{-1},1}_{1},\dots,y^{-1,1}_{N_{-1}},\dots,y^{-1,1}_{1},y_1^{1,1},\dots,y_{N_{1}}^{1,1},\dots,y_1^{\tilde{i},1},\dots,y_{\tilde{j}}^{\tilde{i},1})\\
		\neq& (y^{-\tilde{i}_{-1},2}_{\tilde{j}_{-1}},\dots,y^{-\tilde{i}_{-1},2}_{1},\dots,y^{-1,2}_{N_{-1}},\dots,y^{-1,2}_{1},y_1^{1,2},\dots,y_{N_{1}}^{1,2},\dots,y_1^{\tilde{i},2},\dots,y_{\tilde{j}}^{\tilde{i},2}),
	\end{split}
\end{equation*}
then  there is $-M'_{-N_0-\dots-N_{-\tilde{i}_{-1}+1}-\tilde{j}_{-1}}\leq t\leq M'_{N_0+\dots+N_{\tilde{i}-1}+\tilde{i}+\tilde{j}-1}$ such that
\begin{equation}\label{equation-H}
	d(f^{t}(z^1),f^{t}(z^2))>\varepsilon^*,
\end{equation}
in particular, $z^1\neq z^2.$
For any $k\geq 1$ and  $(\mathbf{y}_{-k},\dots,\mathbf{y}_{-1},\mathbf{y}_1,\dots,\mathbf{y}_k)\in   \Gamma_{-k}\times\dots\times\Gamma_k,$ we choose $$\tilde{z}(\mathbf{y}_{-k},\dots,\mathbf{y}_{-1},\mathbf{y}_1,\dots,\mathbf{y}_k)\in \tilde{G}(\mathbf{y}_{-k},\dots,\mathbf{y}_{-1},\mathbf{y}_1,\dots,\mathbf{y}_k).$$

Denote $\tilde{F}_{k}:=\{\tilde{z}(\mathbf{y}_{-k},\dots,\mathbf{y}_{-1},\mathbf{y}_1,\dots,\mathbf{y}_k):(\mathbf{y}_{-k},\dots,\mathbf{y}_{-1},\mathbf{y}_1,\dots,\mathbf{y}_k)\in  \Gamma_{-k}\times\dots\times\Gamma_k\}$.  Then, similar to Step 5, one has $\sharp \tilde{F}_k=\sharp \Gamma_{-k}\times\dots\times\sharp\Gamma_k=(\sharp\Gamma_{\mathfrak{n}})^{N_{-k}+\dots +N_{-1}+N_1+\dots+N_k}.$
Define $\tilde{\mu}_{k}:=\frac{1}{\sharp \tilde{F}_k}\sum_{z\in \tilde{F}_{k}}\delta_{z}.$ Suppose $\tilde{\mu}=\lim\limits_{n\to\infty}\tilde{\mu}_{k_s}$ for some $k_s\to \infty$. For any fixed $p\in\mathbb{N^{+}}$ and all $p'\geq 0$,
since $\tilde{F}_{p+p'}\subseteq \tilde{G}_{p+p'},$ one has $\tilde{\mu}_{p+p'}(\tilde{G}_{p+p'})=1.$ Combining with $\tilde{G}_{p+p'}\subseteq \tilde{G}_{p},$ we have $\tilde{\mu}_{p+p'}(\tilde{G}_{p})=1.$
Then $\tilde{\mu}(\tilde{G}_{p})\geq \limsup\limits_{n\to\infty}\tilde{\mu}_{k_s}(\tilde{G}_{p})=1$. It follows that
$\tilde{\mu}(\tilde{G})=\lim\limits_{p\to \infty}\tilde{\mu}(\tilde{G}_p)=1.$

For any $\eta>0,$ denote
$$\kappa_\eta^1:=(1-2\eta)(h_{top}(f,X)-\eta),$$
$$\kappa_\eta^2:=(1-3\eta -(1+\eta)(1+\frac{\lambda_1}{\lambda_2})(\frac{\overline{\tau}+\eta}{\lambda_1}+2\eta)\frac{\lambda_2}{\lambda_2+\overline{\tau}+\eta+\eta\lambda_1})(h_{top}(f,X)-\eta).$$ Let $\eta>0$ small enough such that $\kappa_\eta>0.$
We denote $$r_0:=\min\{\frac{\varepsilon^{*}}{L_1^{\tilde{N}+1}},\frac{\varepsilon^{*}}{L_2^{\tilde{N}+1}}\}, c:=e^{(1-\frac{\ln\varepsilon^*}{\ln L_1})\kappa_\eta^1+(1-\frac{\ln\varepsilon^*}{\ln L_2})\kappa_\eta^2}.$$
Take an integer $\tilde{i}_{-1}<0$ such that
\begin{equation}\label{equation-CA}
	\frac{1}{N_{-1}+\dots+N_{\tilde{i}_{-1}}}<\eta.
\end{equation}
Set $\tilde{N}:=\max\{M'_{N_1+N_2+3},M'_{-N_{-1}-\dots-N_{\tilde{i}_{-1}-1}},\frac{6\mathfrak{n} +2}{\eta},\frac{4\mathfrak{n} +3}{\eta}\}.$For any $0<r\leq r_0$, denote $\psi_i(r):=\max\{m\in\mathbb{Z}: L_i^mr<\varepsilon^*\}$ for each $i\in\{1,2\}$
Then
\begin{equation}\label{equation-DB}
	\psi_i(r)\geq \frac{\ln \varepsilon^*-\ln r}{\ln L_i}-1\geq \frac{\ln \varepsilon^*-\ln r_0}{\ln L_i}-1\geq \tilde{N},
\end{equation}
and by the definition of $L$, for any $x\in X$ we have
$$B(x,r)\subset B_{\psi_{1}(r)}(x,\varepsilon^*,f^{-1})\cap B_{\psi_{2}(r)}(x,\varepsilon^*,f).$$

Fix $z\in \tilde{G}$ and $0<r\leq r_0.$ By (\ref{equation-DB})  there exists $i'\geq 3,$ $-i'_{-1}\leq \tilde{i}_{-1}-1$ and $1\leq j'\leq N_{i'}+1,$ $1\leq j'_{-1}\leq N_{-i'_{-1}}$ such that
\begin{equation}\label{equation-BM}
	M'_{-N_{-1}-\dots-N_{-i'_{-1}+1}-j'_{-1}}\leq \psi_1(r) \leq M'_{-N_{-1}-\dots-N_{-i'_{-1}+1}-j'_{-1}-1}.
\end{equation}
\begin{equation}\label{equation-BL}
	M'_{N_1+\dots+N_{i'-1}+i'-1+j'}\leq \psi_2(r) \leq M'_{N_1+\dots+N_{i'-1}+i'+j'}.
\end{equation}
Since $z\in \tilde{G}\subseteq \tilde{G}_{i_{\max}}$ where $i_{\max}=\max\{i',i'_{-1}\},$  then there exists $\mathbf{y}_{i}^0=(y_1^{i,0},\dots,y_{N_{i}}^{i,0})\in \Gamma_{i}$ for each $i\in [-i_{\max},1]\cup [1,i_{\max}]$ such that
$$z\in \tilde{G}(\mathbf{y}_{-i_{\max}}^0,,\dots,\mathbf{y}_{-1}^0,\mathbf{y}_1^0,\dots,\mathbf{y}_{i_{\max}}^0).$$
Note that for any $k_s\geq i_{\max}$ and any $(\mathbf{y}_{-k_s},\dots,\mathbf{y}_{k_s})\in  \Gamma_{-k_s}\times\dots\times\Gamma_{k_s},$ if
\begin{equation*}
	\begin{split}
		&(y^{-i'_{-1}}_{j'_{-1}},\dots,y^{-i'_{-1}}_{1},\dots,y^{-1}_{N_{-1}},\dots,y^{-1}_{1},y_1^{1},\dots,y_{N_{1}}^{1},\dots,y_1^{i'},\dots,y_{j''}^{i'})\\
		\neq& (y^{-i'_{-1},0}_{j'_{-1}},\dots,y^{-i'_{-1},0}_{1},\dots,y^{-1,0}_{N_{-1}},\dots,y^{-1,0}_{1},y_1^{1,0},\dots,y_{N_{1}}^{1,0},\dots,y_1^{i',0},\dots,y_{j''}^{i',0}),
	\end{split}
\end{equation*}
where $j''=\min\{j',N_{i'}\},$ then  by (\ref{equation-H})  $$\tilde{z}(\mathbf{y}_{-k_s},\dots\mathbf{y}_{k_s})\not\in B_{M'_{-N_{0}-\dots-N_{-i'_{-1}+1}-j'_{-1}}}(z,\eps^*,f^{-1})\cap B_{M'_{N_1+\dots+N_{i'-1}+i'-1+j''}}(z,\eps^*,f).$$ Then we have
\begin{equation}\label{equation-BG}
	\begin{split}
		\mu_{k_{s}}(B(z,r))\leq&\mu_{k_{s}}(B_{\psi_{1}(r)}(z,\varepsilon^*,f^{-1})\cap B_{\psi_{2}(r)}(z,\varepsilon^*,f))\\
		\leq &\mu_{k_{s}}(B_{M'_{-N_1-\dots-N_{-i'_{-1}+1}-j'_{-1}}}(z,\varepsilon^*,f^{-1})\cap B_{M'_{N_1+\dots+N_{i'-1}+i'-1+j''}}(z,\varepsilon^*,f))\\
		\leq & \frac{(\sharp\Gamma_{\mathfrak{n}})^{N_{-i'_{-1}}-j'_{-1}}\times \sharp\Gamma_{-i'_{-1}-1}\times\dots\times\sharp\Gamma_{-k_s}\times (\sharp\Gamma_{\mathfrak{n}})^{N_{i'}-j''}\times \sharp\Gamma_{i'+1}\times\dots\times\sharp\Gamma_{k_s}}{\sharp\Gamma_{-k_s}\times\dots\times\sharp \Gamma_{k_s}}\\
		=&\frac{1}{\sharp\Gamma_{-1}\times\dots\times\sharp \Gamma_{-i'_{-1}+1}\times (\sharp\Gamma_{\mathfrak{n}})^{j'_{-1}}}\times\frac{1}{\sharp\Gamma_{1}\times\dots\times\sharp \Gamma_{i'-1}\times (\sharp\Gamma_{\mathfrak{n}})^{j''}}\\
		=&\frac{1}{(\sharp\Gamma_{\mathfrak{n}})^{N_{-1}+\dots+N_{-i'_{-1}+1}+j'_{-1}}}\times\frac{1}{(\sharp\Gamma_{\mathfrak{n}})^{N_1+\dots+N_{i'-1}+j''}}\\
		\leq& e^{-(N_{-1}+\dots+N_{-i'_{-1}+1}+j'_{-1})\mathfrak{n} (h_{top}(f,X)-\eta)}\times e^{-(N_1+\dots+N_{i'-1}+j'')\mathfrak{n} (h_{top}(f,X)-\eta)}
	\end{split}
\end{equation}
We have the following
\begin{equation*}
	\begin{split}
		&\frac{\psi_1(r)-(N_{-1}+\dots+N_{-i'_{-1}+1}+j'_{-1})\mathfrak{n}}{\psi_1(r)} \\
		\leq& \frac{M'_{-N_{-1}-\dots-N_{-i'_{-1}+1}-j'_{-1}-1}-(N_{-1}+\dots+N_{-i'_{-1}+1}+j'_{-1})\mathfrak{n}}{M'_{-N_{-1}-\dots-N_{-i'_{-1}+1}-j'_{-1}-1}}\qquad\qquad\qquad (\text{using }  (\ref{equation-BM}))\\
		=& \frac{(N_{-1}+\dots+N_{-i'_{-1}+1}+j'_{-1}+1)(\mathfrak{n}+p(\mathfrak{n},\varepsilon_0))-(N_{-1}+\dots+N_{-i'_{-1}+1}+j'_{-1})\mathfrak{n}}{(N_{-1}+\dots+N_{-i'_{-1}+1}+j'_{-1}+1)(\mathfrak{n}+p(\mathfrak{n},\varepsilon_0))} \\
		=& \frac{(N_{-1}+\dots+N_{-i'_{-1}+1}+j'_{-1})p(\mathfrak{n},\varepsilon_0)+\mathfrak{n}+p(\mathfrak{n},\varepsilon_0)}{(N_{-1}+\dots+N_{-i'_{-1}+1}+j'_{-1}+1)(\mathfrak{n}+p(\mathfrak{n},\varepsilon_0))} \\
		\leq &\frac{p(\mathfrak{n},\varepsilon_0)}{\mathfrak{n}+p(\mathfrak{n},\varepsilon_0)} +\frac{1}{N_{-1}+\dots+N_{-i'_{-1}+1}+j'_{-1}+1} \\
		\leq &\frac{\eta}{1+\eta} +\frac{1}{N_{-1}+\dots+N_{-i'_{-1}+1}+j'_{-1}+1} \qquad\qquad\qquad\qquad (\text{using } (\ref{equation-A1}) \text{ and the definition } \tilde{n})\\
		\leq &\eta +\frac{1}{N_{-1}+\dots+N_{-i'_{-1}+1}+j'_{-1}+1} \\
		\leq& 2\eta.\qquad\qquad\qquad\qquad\qquad\qquad\qquad\qquad\qquad\qquad\qquad\qquad\qquad\qquad\qquad\quad (\text{using }  (\ref{equation-CA}))\\
	\end{split}
\end{equation*}
It follows that
\begin{equation}\label{equation-CB}
	(N_{-1}+\dots+N_{-i'_{-1}+1}+j'_{-1})\mathfrak{n}\geq (1-2\eta)\psi_1(r)
\end{equation}
By   (\ref{equation-H1}) and (\ref{equation-H2}), we have
\begin{equation}\label{equation-CC}
	(N_1+\dots+N_{i'-1}+j'')\mathfrak{n}\geq (1+\eta)(1+\frac{\lambda_1}{\lambda_2})(\frac{\overline{\tau}+\eta}{\lambda_1}+2\eta)\frac{\lambda_2}{\lambda_2+\overline{\tau}+\eta+\eta\lambda_1}\psi_2(r)
\end{equation}
Then by  (\ref{equation-BG}), (\ref{equation-CB}) and (\ref{equation-CC}),  we have
\begin{equation*}
	\begin{split}
		\mu_{k_{s}}(B(z,r))\leq e^{-\psi_1(r)\kappa_\eta^1}\times e^{-\psi_2(r)\kappa_\eta^2}.
	\end{split}
\end{equation*}
By (\ref{equation-DB}), we have
$$e^{-\psi_i(r)}\leq e^{\frac{\ln r-\ln \varepsilon^*}{\ln L_i}+1}=r^{\frac{1}{\ln L_i}}e^{-\frac{\ln \varepsilon^*}{\ln L_i}+1}.$$
So we have
$$\mu_{k_{s}}(B(z,r))\leq r^{\frac{\kappa_\eta^1}{\ln L_1}}e^{(1-\frac{\ln \varepsilon^*}{\ln L_1})\kappa_\eta^1}\times r^{\frac{\kappa_\eta^2}{\ln L_2}}e^{(1-\frac{\ln \varepsilon^*}{\ln L_2})\kappa_\eta^2}=cr^{\frac{\kappa_\eta^1}{\ln L_1}+\frac{\kappa_\eta^2}{\ln L_2}}.$$
Thus by Lemma \ref{Lem-mass}
$$\operatorname{dim}_H \cap_{j=1}^{\infty}\mathfrak{S}(f,\varphi_j,\mathcal{Z}_j,S_j)\geq \operatorname{dim}_H\tilde{G}\geq \frac{\kappa_\eta^1}{\ln L_1}+\frac{\kappa_\eta^2}{\ln L_2} .$$
Let $\eta\to 0$, we eventually have (\ref{equa-thm-3}).
\qed

\subsection{Proof of Corollary \ref{theorem-shrinking-2}}
Assume that $\cap_{j=1}^{\infty}\mathrm{Ind}'(\mathcal{Z}_j,S_j,\Lambda)\neq\emptyset.$
Take $I'\in\cap_{j=1}^{\infty}\mathrm{Ind}'(\mathcal{Z}_j,S_j,\Lambda),$  there is $(I_1^j,I_2^j)\in \mathrm{Ind}(\mathcal{Z}_j,S_j,\Lambda)$ such that $I'=I_1^j-I_2^j.$ Then we have $(I_2^j,I_2^j)\in \mathrm{Ind}(\mathcal{Z}_j,S_j,f^{(I_1^j-I_2^j)\textrm{ mod } N}(\Lambda))$ for any $j\geq 1.$  Since $f^{N}(f^{(I_1^j-I_2^j)\textrm{ mod } N}(\Lambda))=f^{(I_1^j-I_2^j)\textrm{ mod } N}(\Lambda)$, $X=\bigcup_{l=0}^{N-1}f^{l}(f^{(I_1^j-I_2^j)\textrm{ mod } N}(\Lambda)),$  then there is no loss of generality in assuming $0\in\cap_{j=1}^{\infty}\mathrm{Ind}'(\mathcal{Z}_j,S_j,\Lambda),$ and $(I_2^j,I_2^j)\in \mathrm{Ind}(\mathcal{Z}_j,S_j,\Lambda)$ for any $j\geq 1.$

Since $f$ is Lipschitz, then   $0<L:=\sup\{\frac{d(f(x),f(y))}{d(x,y)}:x\neq y\in X\}<\infty.$  For each $j\geq 1,$ define
\begin{equation*}
	\begin{split}
		&\tilde{\mathcal{Z}}_j:=\{\tilde{z}_j^i\}_{i=0}^{\infty}\text{ with }  f^{I_2^j}(\tilde{z}_j^i)=z_j^i,\\
		&\tilde{S}_j:=\{\tilde{s}_j^i\}_{i=0}^{\infty}  \text{ with } \tilde{s}_j^i=s_j^i-I_2^j,\\
		&\tilde{\varphi_j}:=\frac{\varphi_j}{L^{I_2}}.
	\end{split}
\end{equation*}
Then we have $\mathfrak{S}(f,\tilde{\varphi_j},\tilde{\mathcal{Z}}_j,\tilde{S}_j)\subseteq \mathfrak{S}(f,\varphi_j,\mathcal{Z}_j,S_j)$ and $(0,0)\in \mathrm{Ind}(\tilde{\mathcal{Z}}_j,\tilde{S}_j,\Lambda).$
Thus there is no loss of generality in assuming $(0,0)\in \mathrm{Ind}(\mathcal{Z}_j,S_j,\Lambda)$ for any $j\geq 1.$
So there is $\overline{S}_j=\{\overline{s}_j^i\}_{i=0}^{\infty}\subset S_j$ with $\lim\limits_{i\to\infty}\overline{s}_j^i=\infty$ such that $z_j^{\overline{s}_j^i}\in \Lambda$ and $N|\overline{s}_j^i.$
Then we define $\frac{\overline{S}_j}{N}:=\{\frac{\overline{s}_j^i}{N}\}_{i=0}^{\infty}$ and $\overline{\mathcal{Z}}_j:=\{\overline{z}_j^i\}_{i=0}^{\infty}\subset X$ as
$$
\overline{z}_j^i:=\left\{\begin{array}{ll}
	z_j^{iN}, & \text { if } i\in \frac{\overline{S}_j}{N},\\
	z_j^1, & \text { if } i\in \mathbb{N}\setminus \frac{\overline{S}_j}{N}.
\end{array}\right.
$$
\
Define  $\overline{\varphi}_j$ on $\mathbb{N}$  as
$$
\overline{\varphi}_j(n):=\left\{\begin{array}{ll}
	\varphi_j(nN), & \text { if } n\in \frac{\overline{S}_j}{N},\\
	1, & \text { if } n\in \mathbb{N}\setminus \frac{\overline{S}_j}{N}.
\end{array}\right.
$$
Then we have
\begin{equation*}
	\begin{split}
		\overline{\tau}_j(\overline{\varphi}_j):&=\limsup\limits_{n\to\infty}-\frac{\ln\overline{\varphi}_j(n)}{n}=\limsup\limits_{n\to\infty, nN\in \overline{S}_j}-\frac{\ln\overline{\varphi}_j(n)}{n}  \\
		&=\limsup\limits_{n\to\infty, nN\in \overline{S}_j}-\frac{\ln\varphi_j(nN)}{n}  \\
		&= N\limsup\limits_{i\to \infty}-\frac{\ln\varphi(\tilde{s}_j^i)}{\tilde{s}_j^i}\leq N\overline{\tau}_j.
	\end{split}
\end{equation*}
It implies $\overline{\overline{\tau}}=\sup\limits_{j\geq 1}\overline{\tau}_j(\overline{\varphi}_j)\leq N\overline{\tau}<N\lambda_1.$
Then by Theorem \ref{theorem-shrinking-1} we have
\begin{equation*}
	\begin{split}
		h_{top}(f^N,\cap_{j=1}^{\infty}\mathfrak{S}(f^N,\overline{\varphi}_j,\overline{\mathcal{Z}}_j,\frac{\overline{S}_j}{N}))&\geq \frac{N\lambda_1N\lambda_2-N\lambda_2N\overline{\tau} }{N\lambda_1N\lambda_2+N\lambda_1N\overline{\tau}}h_{top}(f^N,\Lambda)\\
		&= \frac{\lambda_1\lambda_2-\lambda_2\overline{\tau} }{\lambda_1\lambda_2+\lambda_1\overline{\tau}}h_{top}(f^N,\Lambda).
	\end{split}
\end{equation*}
By Proposition \ref{basic properties}, we have
$$h_{top}(f,\cap_{j=1}^{\infty}\mathfrak{S}(f^N,\overline{\varphi}_j,\overline{\mathcal{Z}}_j,\frac{\overline{S}_j}{N}))\geq \frac{\lambda_1\lambda_2-\lambda_2\overline{\tau} }{\lambda_1\lambda_2+\lambda_1\overline{\tau}}h_{top}(f,X).$$
Note that $\mathfrak{S}(f^N,\overline{\varphi}_j,\overline{\mathcal{Z}}_j,\frac{\overline{S}_j}{N})\subset\mathfrak{S}(f,\varphi_j,\mathcal{Z}_j,S_j).$  So we obtain
(\ref{equa-thm-4}).
Similarly, one can get  (\ref{equa-thm-5}) and (\ref{equa-thm-6}).

\subsection{Proof of Corollary \ref{Cor-covering}}
Since $X$ is compact,  there is a sequence of points $\{x_i\}_{i=1}^{\infty}\subset X$ such that  $\overline{\{x_i\}_{i=1}^{\infty}}= X.$
For each $i\geq 1,$ define
\begin{equation*}
	\begin{split}
		&\varphi_j:=\varphi,\\
		&\mathcal{Z}_j:=\{z_j^i\}_{i=1}^{\infty}  \text{ with } z_j^i=x_j \text{ for any }i\geq 1,\\
		&S_j:=\mathbb{N}.
	\end{split}
\end{equation*}
Then $\cap_{j=1}^{\infty}\mathfrak{S}(f,\varphi_j,\mathcal{Z}_j,S_j)\subset \mathfrak{D}(f,\varphi).$
Note that if $x_j\in f^{I}(\Lambda),$ then there  are infinitely many $i\in \mathbb{N}$ such that $N|(i-I).$ Thus
we have  $0\in \cap_{j=1}^{\infty}\mathrm{Ind}'(\mathcal{Z}_j,S_j,\Lambda).$ So by Corollary \ref{theorem-shrinking-1} and Remark \ref{Rem-under} we obtain Corollary \ref{Cor-covering}.

\section{Upper bounds}\label{sec-upper-bounds}
In this section we give the proof of Theorem \ref{Thm-upper-bound}. It is clear that the result holds for $\underline{\tau}=0$ from Lemma \ref{lem-upper-bound}.  Now, we assume that $\underline{\tau}>0.$

{\bf Proof of item(1)}
Fix $\delta>0.$ Then there is $\underline{\tau}>\varepsilon>0$ such that
\begin{equation}\label{eq-upper-6}
	h_{top}(f,\mathfrak{S}(f,\varphi,\mathcal{Z},\mathbb{N}),\varepsilon)>h_{top}(f,\mathfrak{S}(f,\varphi,\mathcal{Z},\mathbb{N}))-\delta,
\end{equation}
\begin{equation}\label{eq-upper-7}
	\limsup _{n \rightarrow \infty} \frac{\ln s(X, n, \frac{\varepsilon}{2})}{n}<h_{top}(f,X)+\delta,
\end{equation}
\begin{equation}\label{eq-upper-1}
	\frac{2\varepsilon}{\underline{\tau}+\ln L}(h_{top}(f,X)+3\delta)<\frac{\delta}{2}.
\end{equation}
And there is $N^*\in\mathbb{N}$ such that for any $n\geq N^*$
\begin{equation}\label{eq-upper-2}
	\varphi(n)<e^{-(\underline{\tau}-\varepsilon)n}<\frac{\varepsilon}{2},
\end{equation}
\begin{equation}\label{eq-upper-3}
	s(X, n, \frac{\varepsilon}{2})<e^{n(h_{top}(f,X)+2\delta)},
\end{equation}
\begin{equation}\label{eq-upper-4}
	-\varepsilon n<\ln \frac{\varepsilon}{2}-\ln L.
\end{equation}

By definition of $\mathfrak{S}(f,\varphi,\mathcal{Z},\mathbb{N})$ we have
$$\mathfrak{S}(f,\varphi,\mathcal{Z},\mathbb{N})\subset \bigcap_{N=1}^{\infty}\bigcup_{n=N}^{\infty}\{x\in X: d(f^n(x),z_n)<\varphi(n)\}.$$
For any $n\in\mathbb{N},$ let $\Gamma_{n,\frac{\varepsilon}{2}}$ be an $(n, \frac{\varepsilon}{2})$-separated set in $X$ with $\sharp \Gamma_{n,\frac{\varepsilon}{2}}=s(X, n, \frac{\varepsilon}{2}).$ Then we have $X=\bigcup_{w\in \Gamma_{n,\frac{\varepsilon}{2}}}B_{n}(w,\frac{\varepsilon}{2},f)$ and thus $$\{x\in X: d(f^n(x),z_n)<\varphi(n)\}\subset \bigcup_{w\in \Gamma_{n,\frac{\varepsilon}{2}}}J(w)$$
where
$J(w):=\{x\in B_{n}(w,\frac{\varepsilon}{2},f): d(f^n(x),z_n)<\varphi(n)\}.$
Denote $\psi(n):=\max\{m\in\mathbb{Z}: L^m\varphi(n)<\frac{\varepsilon}{2}\}.$ Then  $\psi(n)\geq 0$ for any $n\geq N^*,$ and by (\ref{eq-upper-2}) and (\ref{eq-upper-4}) we have
\begin{equation}\label{eq-upper-5}
	\psi(n)\geq \frac{\ln \frac{\varepsilon}{2}-\ln \varphi(n)}{\ln L}-1\geq \frac{\ln \frac{\varepsilon}{2}+(\underline{\tau}-\varepsilon)n}{\ln L}-1= \frac{\ln \frac{\varepsilon}{2}-\ln L}{\ln L}+\frac{(\underline{\tau}-\varepsilon)n}{\ln L}\geq \frac{(\underline{\tau}-2\varepsilon)n}{\ln L}.
\end{equation}
By definition of $\psi(n)$ and $L,$  for any $n\geq N^*$ we have
$$J(w)\subset \{x\in B_{n}(w,\frac{\varepsilon}{2},f): d(f^{n+i}(x),f^{i}(z_n))<\frac{\varepsilon}{2} \text{ for any } 0\leq i\leq \psi(n)\}.$$
When $J(w)\neq\emptyset,$ choose $w'\in J(w).$ Then it is easy to see that $J(w)\subset B_{n+\psi(n)}(w',\varepsilon,f)$.
Hence, for any $N\geq N^*$, we get an  cover of $\mathfrak{S}(f,\varphi,\mathcal{Z},\mathbb{N})$ as:
$$\mathfrak{S}(f,\varphi,\mathcal{Z},\mathbb{N})\subset\bigcup_{n=N}^{\infty}\bigcup_{w\in \Gamma_{n,\frac{\varepsilon}{2}}, J(w)\neq\emptyset} B_{n+\psi(n)}(w',\varepsilon,f).$$
Let $s:=\frac{\ln L}{\ln L+\underline{\tau}}(h_{top}(f,X)+3\delta).$ For any $N\geq N^*$ we have
\begin{equation}\label{equation-EA}
	\begin{split}
		& \sum_{n=N}^{\infty}\sum_{w\in \Gamma_{n,\frac{\varepsilon}{2}}} e^{-s(n+\psi(n))}\\
		\leq & \sum_{n=N}^{\infty} e^{-s(n+\psi(n))}  \sharp \Gamma_{n,\frac{\varepsilon}{2}} \\
		\leq  &\sum_{n=N}^{\infty} e^{-\frac{\ln L}{\ln L+\underline{\tau}}(h_{top}(f,X)+3\delta)(n+\frac{(\underline{\tau}-2\varepsilon)n}{\ln L})+n(h_{top}(f,X)+2\delta)}\qquad (\text{using } (\ref{eq-upper-5}) \text{ and } (\ref{eq-upper-3}))\\
		=  &\sum_{n=N}^{\infty} e^{\frac{2\varepsilon n}{\ln L+\underline{\tau}}(h_{top}(f,X)+3\delta)-n\delta}\\
		\leq & \sum_{n=N}^{\infty} e^{-\frac{n\delta}{2}}.    \qquad\qquad\qquad\qquad\qquad\qquad\qquad\qquad\qquad\qquad\qquad(\text{using } (\ref{eq-upper-1}) )
	\end{split}
\end{equation}
Then
$
		\mathcal{M}_{N,\varepsilon}^{s}(\mathfrak{S}(f,\varphi,\mathcal{Z},\mathbb{N}))\leq \sum_{n=N}^{\infty}\sum_{w\in \Gamma_{n,\frac{\varepsilon}{2}}} e^{-s(n+\psi(n))}
		\leq  \sum_{n=N}^{\infty} e^{-\frac{n\delta}{2}}.
$
It implies $\mathcal{M}_{\varepsilon}^{s}(\mathfrak{S}(f,\varphi,\mathcal{Z},\mathbb{N}))=\lim\limits_{N\to \infty}\mathcal{M}_{N,\varepsilon}^{s}(\mathfrak{S}(f,\varphi,\mathcal{Z},\mathbb{N}))=0.$ So $$s=\frac{\ln L}{\ln L+\underline{\tau}}(h_{top}(f,X)+3\delta)\geq h_{top}(f,\mathfrak{S}(f,\varphi,\mathcal{Z},\mathbb{N}),\varepsilon)>h_{top}(f,\mathfrak{S}(f,\varphi,\mathcal{Z},\mathbb{N}))-\delta.$$
Let $\delta\to 0,$ we obtain (\ref{equa-thm-10}).

If further $f:X\to X$ is  $\lambda$-hyperbolic, we also  require that $0<\varepsilon<\varepsilon_h$ for some $\varepsilon_h>0$ such that for any $x\in X$ and $n\in\mathbb{N},$ one has $B_{n}(x,\varepsilon,f)\subset B_n(x,\varepsilon_h,\lambda,f).$ Then
we have $B_{n+\psi(n)}(w',\varepsilon,f)\subset B_{n+\psi(n)}(w',\varepsilon_h,\lambda,f)\subset B(w',\varepsilon_h e^{-(n+\psi(n)-1)\lambda})$.
Hence, for any $N\geq N^*$, we get an  cover of $\mathfrak{S}(f,\varphi,\mathcal{Z},\mathbb{N})$ as:
$$\mathfrak{S}(f,\varphi,\mathcal{Z},\mathbb{N})\subset\bigcup_{n=N}^{\infty}\bigcup_{w\in \Gamma_{n,\frac{\varepsilon}{2}}, J(w)\neq\emptyset} B(w',\varepsilon_h e^{-(n+\psi(n)-1)\lambda}).$$
Denote $\delta_N:=\varepsilon_h e^{-(N-1)\lambda}.$ Then $\varepsilon_h e^{-(n+\psi(n)-1)\lambda}<\delta_N$ for any $n\geq N$ and $\lim\limits_{N\to\infty}\delta_N=0.$
Let $s_h:=\frac{1}{\lambda}\frac{\ln L}{\ln L+\underline{\tau}}(h_{top}(f,X)+3\delta)=\frac{1}{\lambda}s.$ Then for any $N\geq N^*,$ by (\ref{equation-EA}) we have
\begin{equation*}
	\begin{split}
		\mathcal{H}_{\delta_N}^{s_h}(\mathfrak{S}(f,\varphi,\mathcal{Z},\mathbb{N}))&\leq \sum_{n=N}^{\infty}\sum_{w\in \Gamma_{n,\frac{\varepsilon}{2}}} \varepsilon_h^{s_h} e^{-(n+\psi(n)-1)\lambda s_h}\\
		&\leq \varepsilon_h^{s_h} e^s\sum_{n=N}^{\infty}\sum_{w\in \Gamma_{n,\frac{\varepsilon}{2}}}  e^{-(n+\psi(n))s}\\
		&\leq \varepsilon_h      ^{s_h} e^s \sum_{n=N}^{\infty} e^{-\frac{n\delta}{2}}.
	\end{split}
\end{equation*}
It implies $$\mathcal{H}^{s_h}(\mathfrak{S}(f,\varphi,\mathcal{Z},\mathbb{N})):=\lim\limits _{N \rightarrow \infty}\mathcal{H}_{\delta_N}^{s_h}(\mathfrak{S}(f,\varphi,\mathcal{Z},\mathbb{N}))=0.$$ So $$s_h=\frac{1}{\lambda}\frac{\ln L}{\ln L+\underline{\tau}}(h_{top}(f,X)+3\delta)\geq \operatorname{dim}_H \mathfrak{S}(f,\varphi,\mathcal{Z},\mathbb{N}).$$
Let $\delta\to 0,$ we obtain  (\ref{equa-thm-11}).

{\bf Proof of item(2.1)}
Fix $\delta>0.$  Since $\ln L_1>\underline{\tau},$ then there is $\underline{\tau}>\varepsilon>0$ such that (\ref{eq-upper-6}) and (\ref{eq-upper-7}) hold and
\begin{equation}\label{eq-upper-8}
	\frac{\ln L_1\ln L_2-\underline{\tau}\ln L_2}{\ln L_1\ln L_2+\underline{\tau}\ln L_1}\frac{2\varepsilon }{\ln L_2}(h_{top}(f,X)+3\delta)+\frac{2\varepsilon }{\ln L_1}(h_{top}(f,X)+2\delta)<\frac{\ln L_1-\underline{\tau}}{2\ln L_1}\delta
\end{equation}
\begin{equation}\label{eq-upper-11}
	\frac{\ln L_1\ln L_2-\underline{\tau}\ln L_2}{\ln L_1\ln L_2+\underline{\tau}\ln L_1}\frac{2\varepsilon }{\ln L_2}(h_{top}(f,X)+3\delta)<\frac{\ln L_1-\underline{\tau}}{\ln L_1}(h_{top}(f,X)+\frac{5}{2}\delta)
\end{equation}
And there is $N^*\in\mathbb{N}$ such that  (\ref{eq-upper-2}), (\ref{eq-upper-3})  and
\begin{equation}\label{eq-upper-41}
	-\varepsilon n<\ln \frac{\varepsilon}{2}-\ln L_i.
\end{equation}
 hold for any $n\geq N^*$ and $i\in\{1,2\}.$

By definition of $\mathfrak{S}(f,\varphi,\mathcal{Z},\mathbb{N})$ we have
$$\mathfrak{S}(f,\varphi,\mathcal{Z},\mathbb{N})\subset \bigcap_{N=1}^{\infty}\bigcup_{n=N}^{\infty}\{x\in X: d(f^n(x),z_n)<\varphi(n)\}.$$
For any $n\in\mathbb{N},$ let $\Gamma_{n,\frac{\varepsilon}{2}}$ be an $(n, \frac{\varepsilon}{2})$-separated set in $X$ with $\sharp \Gamma_{n,\frac{\varepsilon}{2}}=s(X, n, \frac{\varepsilon}{2}).$ Then we have $X=\bigcup_{w\in \Gamma_{n,\frac{\varepsilon}{2}}}B_{n}(w,\frac{\varepsilon}{2},f)$.
Denote $\psi_i(n):=\max\{m\in\mathbb{Z}: L_i^m\varphi(n)<\frac{\varepsilon}{2}\}$ for each $i\in \{1,2\}.$ Then  $\psi_i(n)\geq 0$ for any $n\geq N^*,$ and by (\ref{eq-upper-2}) and (\ref{eq-upper-41}) we have
\begin{equation}\label{eq-upper-51}
	\psi_i(n)\geq \frac{\ln \frac{\varepsilon}{2}-\ln \varphi(n)}{\ln L_i}-1\geq \frac{\ln \frac{\varepsilon}{2}+(\underline{\tau}-\varepsilon)n}{\ln L_i}-1= \frac{\ln \frac{\varepsilon}{2}-\ln L_i}{\ln L_i}+\frac{(\underline{\tau}-\varepsilon)n}{\ln L_i}\geq \frac{(\underline{\tau}-2\varepsilon)n}{\ln L_i}.
\end{equation}
For any $n\geq N^*,$ define
\begin{equation*}
	\Psi(n):=
	\begin{cases}
		n-\psi_1(n),&n-\psi_1(n)\geq N^*,\\
		N^*,&n-\psi_1(n)<N^*.
	\end{cases}
\end{equation*}
Then we have  $X=\bigcup_{w\in \Gamma_{\Psi(n),\frac{\varepsilon}{2}}}B_{\Psi(n)}(w,\frac{\varepsilon}{2})$ and thus
$$\{x\in X: d(f^n(x),z_n)<\varphi(n)\}\subset\bigcup_{w\in \Gamma_{\Psi(n),\frac{\varepsilon}{2}}}J^*(w)$$
where
$J^*(w):=\{x\in B_{\Psi(n)}(w,\frac{\varepsilon}{2},f): d(f^n(x),z_n)<\varphi(n)\}.$

By definition of $\Psi(n)$ and $L_i,$ if $n-\psi_1(n)\geq N^*$ we have
\begin{equation}\label{equation-EB}
	J^*(w)\subset \{x\in B_{n-\psi_1(n)}(w,\frac{\varepsilon}{2},f): d(f^{n+i}(x),f^{i}(z_n))<\frac{\varepsilon}{2} \text{ for any } -\psi_1(n)\leq i\leq \psi_2(n)\}.
\end{equation}
if $n-\psi(n)< N^*$ we have
\begin{equation}\label{equation-EC}
	J^*(w)\subset \{x\in B_{N^*}(w,\frac{\varepsilon}{2},f): d(f^{n+i}(x),f^{i}(z_n))<\frac{\varepsilon}{2} \text{ for any } -\psi_1(n)\leq i\leq \psi_2(n)\}.
\end{equation}
When $J^*(w)\neq\emptyset,$ choose $w'\in J^*(w).$ Then it is easy to see that $J^*(w)\subset B_{n+\psi_2(n)}(w',\varepsilon,f)$.
Hence, for any $N\geq N^*$, we get an  cover of $\mathfrak{S}(f,\varphi,\mathcal{Z},\mathbb{N})$ as:
$$\mathfrak{S}(f,\varphi,\mathcal{Z},\mathbb{N})\subset\bigcup_{n=N}^{\infty}\bigcup_{w\in \Gamma_{\Psi(n),\frac{\varepsilon}{2}}, J^*(w)\neq\emptyset} B_{n+\psi_2(n)}(w',\varepsilon,f).$$
Let $s:=\frac{\ln L_1\ln L_2-\underline{\tau}\ln L_2}{\ln L_1\ln L_2+\underline{\tau}\ln L_1}(h_{top}(f,X)+3\delta).$
If $n-\psi_1(n)\geq N^*$ we have
\begin{equation*}
	\begin{split}
		&\sum_{w\in \Gamma_{\Psi(n),\frac{\varepsilon}{2}}, J^*(w)\neq\emptyset} e^{-s(n+\psi_2(n))}\\
		&\leq e^{-s(n+\psi_2(n))}  \sharp \Gamma_{n-\psi_1(n),\frac{\varepsilon}{2}} \\
		&\leq  e^{-\frac{\ln L_1\ln L_2-\underline{\tau}\ln L_2}{\ln L_1\ln L_2+\underline{\tau}\ln L_1}(h_{top}(f,X)+3\delta)(n+\psi_2(n))+(n-\psi_1(n))(h_{top}(f,X)+2\delta)}\,\,\,\quad\quad\qquad\qquad(\text{using }  (\ref{eq-upper-3}))\\
		&\leq  e^{-\frac{\ln L_1\ln L_2-\underline{\tau}\ln L_2}{\ln L_1\ln L_2+\underline{\tau}\ln L_1}(h_{top}(f,X)+3\delta)(n+\frac{(\underline{\tau}-2\varepsilon)n}{\ln L_2})+(n-\frac{(\underline{\tau}-2\varepsilon)n}{\ln L_1})(h_{top}(f,X)+2\delta)}\quad\quad\quad\qquad\,(\text{using }  (\ref{eq-upper-5}))\\
		&=   e^{\frac{\ln L_1\ln L_2-\underline{\tau}\ln L_2}{\ln L_1\ln L_2+\underline{\tau}\ln L_1}\frac{2\varepsilon n}{\ln L_2}(h_{top}(f,X)+3\delta)+\frac{2\varepsilon n}{\ln L_1}(h_{top}(f,X)+2\delta)-\frac{\ln L_1-\underline{\tau}}{\ln L_1}n\delta}\\
		&\leq e^{-\frac{\ln L_1-\underline{\tau}}{2\ln L_1}n\delta}.      \,\,\,\qquad\qquad\qquad\qquad\qquad\qquad\qquad\qquad\qquad\qquad\quad\quad\qquad\qquad\qquad(\text{using }  (\ref{eq-upper-8}))
	\end{split}
\end{equation*}
If $n-\psi_1(n)<N^*$ we have
\begin{equation*}
	\begin{split}
		&\sum_{w\in \Gamma_{\Psi(n),\frac{\varepsilon}{2}}, J^*(w)\neq\emptyset} e^{-s(n+\psi_2(n))}\\
		&\leq e^{-s(n+\psi_2(n))}  \sharp \Gamma_{N^*,\frac{\varepsilon}{2}} \\
		&\leq  e^{-\frac{\ln L_1\ln L_2-\underline{\tau}\ln L_2}{\ln L_1\ln L_2+\underline{\tau}\ln L_1}(h_{top}(f,X)+3\delta)(n+\psi_2(n))}\sharp \Gamma_{N^*,\frac{\varepsilon}{2}}~~~~~\qquad\qquad\qquad\qquad\qquad\qquad\qquad\qquad~~~~~~~~~~~~~~~~~~~~~~~~(\text{using }  (\ref{eq-upper-3}))\\
		&\leq  e^{-\frac{\ln L_1\ln L_2-\underline{\tau}\ln L_2}{\ln L_1\ln L_2+\underline{\tau}\ln L_1}(h_{top}(f,X)+3\delta)(n+\frac{(\underline{\tau}-2\varepsilon)n}{\ln L_2})}\sharp \Gamma_{N^*,\frac{\varepsilon}{2}}~~~~~~~~~\,\qquad\qquad\qquad\quad\qquad\qquad\qquad\qquad~~~~~~~~~~~~~~~~~~~(\text{using }  (\ref{eq-upper-5}))\\
		&=   C^* e^{\frac{\ln L_1\ln L_2-\underline{\tau}\ln L_2}{\ln L_1\ln L_2+\underline{\tau}\ln L_1}\frac{2\varepsilon n}{\ln L_2}(h_{top}(f,X)+3\delta)-\frac{\ln L_1-\underline{\tau}}{\ln L_1}n(h_{top}(f,X)+3\delta)}~~~~~~~~~\qquad\qquad\qquad\qquad~~~~~~~~~~~(\text{denote }C^*:=\sharp \Gamma_{N^*,\frac{\varepsilon}{2}})\\
		&\leq C^* e^{-\frac{\ln L_1-\underline{\tau}}{2\ln L_1}n\delta}.      ~~~~~~~~~~~~~~~~~~~~~~~~~~~~~~~~~~~~~~~~~~~~\qquad\qquad\qquad\qquad\qquad\qquad\qquad\qquad\qquad\qquad\qquad\qquad\qquad\qquad\qquad~~~~~~~~~~~~~~~~~~~~~~~(\text{using }  (\ref{eq-upper-11}))
	\end{split}
\end{equation*}
Then for any $N\geq N^{*}$ we have
\begin{equation}\label{equation-EF}
	\begin{split}
		\mathcal{M}_{N,\varepsilon}^{s}(\mathfrak{S}(f,\varphi,\mathcal{Z},\mathbb{N}))\leq \sum_{n=N}^{\infty}\sum_{w\in \Gamma_{\Psi(n),\frac{\varepsilon}{2}}, J^*(w)\neq\emptyset} e^{-s(n+\psi_2(n))}
		\leq  \max\{1,C^*\}\sum_{n=N}^{\infty} e^{-\frac{\ln L_1-\underline{\tau}}{2\ln L_1}n\delta}.
	\end{split}
\end{equation}
It implies $\mathcal{M}_{\varepsilon}^{s}(\mathfrak{S}(f,\varphi,\mathcal{Z},\mathbb{N}))=\lim\limits_{N\to \infty}\mathcal{M}_{N,\varepsilon}^{s}(\mathfrak{S}(f,\varphi,\mathcal{Z},\mathbb{N}))=0.$ So $$s=\frac{\ln L_1\ln L_2-\underline{\tau}\ln L_2}{\ln L_1\ln L_2+\underline{\tau}\ln L_1}(h_{top}(f,X)+3\delta)\geq h_{top}(f,\mathfrak{S}(f,\varphi,\mathcal{Z},\mathbb{N}),\varepsilon)>h_{top}(f,\mathfrak{S}(f,\varphi,\mathcal{Z},\mathbb{N}))-\delta.$$
Let $\delta\to 0,$ we obtain  (\ref{equa-thm-12}).

If further $f:X\to X$ is  $(\lambda_1,\lambda_2)$-hyperbolic, we also  require that $0<\varepsilon<\varepsilon_h$ for some $\varepsilon_h>0$ such that for any $x\in X$ and $n\in\mathbb{N},$ one has $B_{n}(x,\varepsilon,f)\subset B_n(x,\varepsilon_h,\lambda_1,\lambda_2,f).$ Take $\tilde{N}^*\geq N^*$ such that $\frac{\lambda_2(n+\psi_{2}(n))}{\lambda_1}\geq N^*+1$ for any $n\geq \tilde{N}^*.$ Denote $\psi_{3}(n)=\lceil\frac{\lambda_2(n+\psi_{2}(n))}{\lambda_1}\rceil.$
Then we have  $\psi_{3}(n)\geq N^*$ and
\begin{equation}\label{equation-EE}
	\frac{\lambda_2(n+\psi_{2}(n))}{\lambda_1}\leq \psi_{3}(n)\leq \frac{\lambda_2(n+\psi_{2}(n))}{\lambda_1}+1.
\end{equation}
Since  $X=\bigcup_{w\in \Gamma_{\psi_3(n),\frac{\varepsilon}{2}}} B_{\psi_3(n)}(w,\frac{\varepsilon}{2},f)$, then $X=\bigcup_{w\in \Gamma_{\psi_3(n),\frac{\varepsilon}{2}}}f^{\psi_{3}(n)}B_{\psi_3(n)}(w,\frac{\varepsilon}{2},f)$ and thus
$$\{x\in X: d(f^n(x),z_n)<\varphi(n)\}\subset \bigcup_{w\in \Gamma_{\Psi(n),\frac{\varepsilon}{2}}} \bigcup_{\tilde{w}\in \Gamma_{\psi_3(n),\frac{\varepsilon}{2}}}J^*(w)\cap f^{\psi_{3}(n)}B_{\psi_3(n)}(\tilde{w},\frac{\varepsilon}{2},f).$$
Combining with (\ref{equation-EB}) and (\ref{equation-EC},)
when $J^*(w)\cap f^{\psi_{3}(n)}B_{\psi_3(n)}(\tilde{w},\frac{\varepsilon}{2},f)\neq\emptyset,$ choose $w''\in J^*(w)\cap f^{\psi_{3}(n)}B_{\psi_3(n)}(\tilde{w},\frac{\varepsilon}{2},f).$ Then it is easy to see that $J^*(w)\cap f^{\psi_{3}(n)}B_{\psi_3(n)}(\tilde{w},\frac{\varepsilon}{2},f)\subset B_{n+\psi_2(n)}(w'',\varepsilon,f)\cap B_{\psi_3(n)}(w'',\varepsilon,f^{-1})=f^{\psi_3(n)}B_{n+\psi_2(n)+\psi_{3}(n)}(f^{-\psi_3(n)}(w''),\varepsilon,f)$.
Since $f:X\to X$ is  $(\lambda_1,\lambda_2)$-hyperbolic,
then
we have
\begin{equation*}
	\begin{split}
		&B_{n+\psi_2(n)+\psi_{3}(n)}(f^{-\psi_3(n)}(w''),\varepsilon,f)\\
		\subset& B_{n+\psi_2(n)+\psi_{3}(n)}(f^{-\psi_3(n)}(w''),\varepsilon_h,\lambda_1,\lambda_2,f)\\
		\subset &\{y\in X:d(f^{\psi_3(n)}(f^{-\psi_3(n)}(w'')),f^{\psi_3(n)}(y))< \varepsilon_h e^{-\min\{\psi_3(n)\lambda_1,(n+\psi_2(n)+\psi_{3}(n)-1-\psi_3(n))\lambda_2\} }\}\\
		=&\{y\in X:d(w'',f^{\psi_3(n)}(y))< \varepsilon_h e^{-\min\{\psi_3(n)\lambda_1,(n+\psi_2(n)-1)\lambda_2\} }\}\\
		=&\{y\in X:d(w'',f^{\psi_3(n)}(y))< \varepsilon_h e^{-(n+\psi_2(n)-1)\lambda_2 }\}\qquad\qquad\qquad\qquad\qquad\qquad~~~(\text{using }  (\ref{equation-EE}))\\
		=&f^{-\psi_3(n)}B(w'',\varepsilon_h e^{-(n+\psi_2(n)-1)\lambda_2}).
	\end{split}
\end{equation*}
Hence, $J^*(w)\cap f^{\psi_{3}(n)}B_{\psi_3(n)}(\tilde{w},\frac{\varepsilon}{2},f)\subseteq B(w'',\varepsilon_h e^{-(n+\psi_2(n)-1)\lambda_2}).$
For any $N\geq \tilde{N}^*$, we get an  cover of $\mathfrak{S}(f,\varphi,\mathcal{Z},\mathbb{N})$ as:
$$\mathfrak{S}(f,\varphi,\mathcal{Z},\mathbb{N})\subset\bigcup_{n=N}^{\infty}\{x\in X: d(f^n(x),z_n)<\varphi(n)\}\subset \bigcup_{n=N}^{\infty}\bigcup_{w\in \Gamma_{\Psi(n),\frac{\varepsilon}{2}}} \bigcup_{\tilde{w}\in \Gamma_{\psi_3(n),\frac{\varepsilon}{2}}}B(w'',\varepsilon_h e^{-(n+\psi_2(n)-1)\lambda_2}).$$
Denote $\delta_N:=\varepsilon_h e^{-(N-1)\lambda_2}.$ Then $\varepsilon_h e^{-(n+\psi_2(n)-1)\lambda_2}<\delta_N$ for any $n\geq N$ and $\lim\limits_{N\to\infty}\delta_N=0.$
Let $$s_h:=\frac{1}{\lambda_2}\frac{\ln L_1\ln L_2-\underline{\tau}\ln L_2}{\ln L_1\ln L_2+\underline{\tau}\ln L_1}(h_{top}(f,X)+3\delta)+\frac{1}{\lambda_1}(h_{top}(f,X)+2\delta)=\frac{1}{\lambda_2}s+\frac{1}{\lambda_1}(h_{top}(f,X)+2\delta).$$ Then for any $N\geq \tilde{N}^*,$  we have
\begin{equation*}
	\begin{split}
		&\mathcal{H}_{\delta_N}^{s_h}(\mathfrak{S}(f,\varphi,\mathcal{Z},\mathbb{N}))\\
		\leq &\sum_{n=N}^{\infty}\sum_{w\in \Gamma_{\Psi(n),\frac{\varepsilon}{2}}} \sum_{\tilde{w}\in \Gamma_{\psi_3(n),\frac{\varepsilon}{2}}}\varepsilon_h^{s_h} e^{-(n+\psi_2(n)-1)\lambda_2 s_h}\\
		\leq &\varepsilon_h^{s_h} e^{\lambda_2s_h}\sum_{n=N}^{\infty}\sum_{w\in \Gamma_{\Psi(n),\frac{\varepsilon}{2}}}  e^{-(n+\psi_2(n))\lambda_2 s_h}e^{\psi_3(n)(h_{top}(f,X)+2\delta)}\\
		=& \varepsilon_h^{s_h} e^{\lambda_2s_h}\sum_{n=N}^{\infty}\sum_{w\in \Gamma_{\Psi(n),\frac{\varepsilon}{2}}}  e^{-(n+\psi_2(n))s}e^{(-\frac{\lambda_2(n+\psi_2(n))}{\lambda_1}+\psi_3(n))(h_{top}(f,X)+2\delta)}\\
		\leq & \varepsilon_h^{s_h} e^{\lambda_2s_h}e^{h_{top}(f,X)+2\delta}\sum_{n=N}^{\infty}\sum_{w\in \Gamma_{\Psi(n),\frac{\varepsilon}{2}}}  e^{-(n+\psi_2(n))s}~~~~~\qquad\quad\qquad\qquad\qquad~~~~~~~~~~~~~~~(\text{using }  (\ref{equation-EE}))\\
		\leq & \varepsilon_h^{s_h} e^{\lambda_2s_h}e^{h_{top}(f,X)+2\delta}\max\{1,C^*\}\sum_{n=N}^{\infty} e^{-\frac{\ln L_1-\underline{\tau}}{2\ln L_1}n\delta}.~~~~~~~\qquad\qquad\qquad\qquad~~~~~~~~~~~~(\text{using }  (\ref{equation-EF}))
	\end{split}
\end{equation*}
It implies $$\mathcal{H}^{s_h}(\mathfrak{S}(f,\varphi,\mathcal{Z},\mathbb{N})):=\lim\limits _{N \rightarrow \infty}\mathcal{H}_{\delta_N}^{s_h}(\mathfrak{S}(f,\varphi,\mathcal{Z},\mathbb{N}))=0.$$ So $$s_h=\frac{1}{\lambda_2}\frac{\ln L_1\ln L_2-\underline{\tau}\ln L_2}{\ln L_1\ln L_2+\underline{\tau}\ln L_1}(h_{top}(f,X)+3\delta)+\frac{1}{\lambda_1}(h_{top}(f,X)+2\delta)\geq \operatorname{dim}_H \mathfrak{S}(f,\varphi,\mathcal{Z},\mathbb{N}).$$
Let $\delta\to 0,$ we obtain  (\ref{equa-thm-13}).

{\bf Proof of item(2.2)}
Assume that $(X,f)$ is $(L_1,L_2)$-bi-Lipschitz with $\ln L_1=\underline{\tau}.$ Fix $\delta>0.$ Then $(X,f)$ is also $(L^\delta_1,L_2)$-bi-Lipschitz where $L_1^\delta=L_1+\delta$. Since $\ln L_1^\delta>\underline{\tau},$ by Theorem \ref{Thm-upper-bound}(2.1) we have
$$h_{top}(f,\mathfrak{S}(f,\varphi,\mathcal{Z},\mathbb{N}))\leq \frac{\ln L_1^\delta\ln L_2-\underline{\tau}\ln L_2}{\ln L_1^\delta\ln L_2+\underline{\tau}\ln L_1^\delta}h_{top}(f,X).$$
Let $\delta\to 0,$ we have
$$h_{top}(f,\mathfrak{S}(f,\varphi,\mathcal{Z},\mathbb{N}))\leq  \frac{\ln L_1\ln L_2-\underline{\tau}\ln L_2}{\ln L_1\ln L_2+\underline{\tau}\ln L_1}h_{top}(f,X)=\frac{\ln L_1\ln L_2-\ln L_1\ln L_2}{\ln L_1\ln L_2+\ln L_1\ln L_1}h_{top}(f,X)=0$$
and thus get  (\ref{equa-thm-14}).

If further $f:X\to X$ is  $(\lambda_1,\lambda_2)$-hyperbolic, then
$$\operatorname{dim}_H \mathfrak{S}(f,\varphi,\mathcal{Z},\mathbb{N})\leq (\frac{1}{\lambda_1}+\frac{1}{\lambda_2}\frac{\ln L_1^\delta\ln L_2-\underline{\tau}\ln L_2}{\ln L^\delta_1\ln L_2+\underline{\tau}\ln L^\delta_1})h_{top}(f,X).$$
Let $\delta\to 0,$ we have
$$\operatorname{dim}_H \mathfrak{S}(f,\varphi,\mathcal{Z},\mathbb{N})\leq (\frac{1}{\lambda_1}+\frac{1}{\lambda_2}\frac{\ln L_1\ln L_2-\underline{\tau}\ln L_2}{\ln L_1\ln L_2+\underline{\tau}\ln L_1})h_{top}(f,X)=\frac{1}{\lambda_1}h_{top}(f,X)$$
and thus get  (\ref{equa-thm-15}).

{\bf Proof of item(2.3)}
Now we assume that $(X,f)$ is $(L_1,L_2)$-bi-Lipschitz with $\ln L_1<\underline{\tau}.$ Fix $\delta>0.$  Then there is $\frac{1}{2}(\underline{\tau}-\ln L_1)>\varepsilon>0$ such that (\ref{eq-upper-6}) holds. Take $N^*\in\mathbb{N}$ such that  (\ref{eq-upper-2})  and
\begin{equation}\label{eq-upper-42}
	-\varepsilon n<\ln \frac{\varepsilon}{2}-\ln L_1.
\end{equation}
 hold for any $n\geq N^*.$  Denote $\psi_1(n):=\max\{m\in\mathbb{Z}: L_1^m\varphi(n)<\frac{\varepsilon}{2}\}.$ Then  $\psi_1(n)\geq 0$ for any $n\geq N^*,$ and by (\ref{eq-upper-2}) and (\ref{eq-upper-42}) we have
 \begin{equation}\label{eq-upper-52}
 	\psi_1(n)\geq \frac{\ln \frac{\varepsilon}{2}-\ln \varphi(n)}{\ln L_1}-1\geq \frac{\ln \frac{\varepsilon}{2}+(\underline{\tau}-\varepsilon)n}{\ln L_1}-1= \frac{\ln \frac{\varepsilon}{2}-\ln L_1}{\ln L_1}+\frac{(\underline{\tau}-\varepsilon)n}{\ln L_1}\geq \frac{(\underline{\tau}-2\varepsilon)n}{\ln L_1}.
 \end{equation}
 Since $\frac{1}{2}(\underline{\tau}-\ln L_1)>\varepsilon,$ by (\ref{eq-upper-52})  we have
\begin{equation*}
	\psi_1(n)\geq \frac{(\underline{\tau}-2\varepsilon)n}{\ln L_1}>n.
\end{equation*}
By definition, for any $N\in\mathbb{N}$ we have
$$\mathfrak{S}(f,\varphi,\mathcal{Z},\mathbb{N})\subset \bigcup_{n=N}^{\infty} J^{**}(n)$$
where $J^{**}(n):=\{x\in X: d(f^n(x),z_n)<\varphi(n)\}.$
By definition of $\psi_1(n)$ and $L,$ for any $n\geq N^*$ we have
$$J^{**}(n)\subset \{x\in X: d(f^{n+i}(x),f^{i}(z_n))<\frac{\varepsilon}{2} \text{ for any } -\psi_1(n)\leq i\leq 0\}\subset B_{n}(f^{-n}(z_n),\frac{\varepsilon}{2},f).$$
Let $s:=\delta$. Then $N\geq N^*$ we have
\begin{equation*}
	\begin{split}
		\mathcal{M}_{N,\varepsilon}^{s}(\mathfrak{S}(f,\varphi,\mathcal{Z},\mathbb{N}))&\leq \sum_{n=N}^{\infty} e^{-n\delta}\\
	\end{split}
\end{equation*}
It implies $\mathcal{M}_{\varepsilon}^{s}(\mathfrak{S}(f,\varphi,\mathcal{Z},\mathbb{N}))=\lim\limits_{N\to \infty}\mathcal{M}_{N,\varepsilon}^{s}(\mathfrak{S}(f,\varphi,\mathcal{Z},\mathbb{N}))=0.$ So $$s=\delta>h_{top}(f,\mathfrak{S}(f,\varphi,\mathcal{Z},\mathbb{N}))-\delta.$$
Let $\delta\to 0,$ we obtain (\ref{equa-thm-16}).

If further $f:X\to X$ is  $(\lambda_1,\lambda_2)$-hyperbolic, we also  require that $0<\varepsilon<\varepsilon_h$ for some $\varepsilon_h>0$ such that for any $x\in X$ and $n\in\mathbb{N},$ one has $B_{n}(x,\varepsilon,f)\subset B_n(x,\varepsilon_h,\lambda_1,\lambda_2,f).$  Then
\begin{equation*}
	\begin{split}
		&J^{**}(n)\\
		\subset &\{x\in X: d(f^{n+i}(x),f^{i}(z_n))<\frac{\varepsilon}{2} \text{ for any } -\psi_1(n)\leq i\leq 0\}\\
		\subset &f^{\psi_{1}(n)-n}B_{\psi_1(n)}(f^{-\psi_1(n)}(z_n),\frac{\varepsilon}{2},f)\\
		\subset & f^{\psi_{1}(n)-n}B_{\psi_1(n)}(f^{-\psi_1(n)}(z_n),\varepsilon_h,\lambda_1,\lambda_2,f)\\
		\subset &f^{\psi_{1}(n)-n}\{y\in X:d(f^{\psi_1(n)-n}(f^{-\psi_1(n)}(z_n)),f^{\psi_1(n)-n}(y))< \varepsilon_h e^{-\min\{(\psi_1(n)-n)\lambda_1,(\psi_1(n)-1-\psi_1(n)+n)\lambda_2\} }\}\\
		=&f^{\psi_{1}(n)-n}\{y\in X:d(f^{-n}(z_n),f^{\psi_1(n)-n}(y))< \varepsilon_h e^{-\min\{(\psi_1(n)-n)\lambda_1,(n-1)\lambda_2\} }\}\\
		=&\{y\in X:d(f^{-n}(z_n),y)< \varepsilon_h e^{-\min\{(\psi_1(n)-n)\lambda_1,(n-1)\lambda_2\} }\}\\
		\subset &\{y\in X:d(f^{-n}(z_n),y)<\max\{ \varepsilon_h e^{-(\psi_1(n)-n)\lambda_1 },\varepsilon_h e^{-(n-1)\lambda_2})\}\}.
	\end{split}
\end{equation*}
By (\ref{eq-upper-52}), we have
\begin{equation}
	e^{-(\psi_1(n)-n)\lambda_1 }\leq    e^{-(\frac{(\underline{\tau}-2\varepsilon)n}{\ln L_1}-n)\lambda_1 } =e^{-\frac{\underline{\tau}-\ln L_1-2\varepsilon}{\ln L_1}\lambda_1 n}.
\end{equation}
Denote $\delta_N:=\varepsilon_h (e^{-\frac{\underline{\tau}-\ln L_1-2\varepsilon}{\ln L_1}\lambda_1 N}+e^{-(N-1)\lambda_2}).$ Then $\max\{ \varepsilon_h e^{-(\psi_1(n)-n)\lambda_1 },\varepsilon_h e^{-(n-1)\lambda_2})\}<\delta_N$ for any $n\geq N$ and $\lim\limits_{N\to\infty}\delta_N=0.$ Fix $s_h>0$. Then
\begin{equation*}
	\begin{split}
		\mathcal{H}_{\delta_N}^{s_h}(\mathfrak{S}(f,\varphi,\mathcal{Z},\mathbb{N}))&\leq \varepsilon_h^{s_h}\sum_{n=N}^{\infty}  \max\{  e^{-(\psi_1(n)-n)\lambda_1s_h }, e^{-(n-1)\lambda_2s_h}\}\\
		&\leq \varepsilon_h^{s_h}\sum_{n=N}^{\infty}    e^{-(\psi_1(n)-n)\lambda_1s_h }+e^{-(n-1)\lambda_2s_h}\\
		&= \varepsilon_h^{s_h}\sum_{n=N}^{\infty}    e^{-(\psi_1(n)-n)\lambda_1s_h }+\varepsilon_h^{s_h}\sum_{n=N}^{\infty} e^{-(n-1)\lambda_2s_h}\\
		&\leq \varepsilon_h^{s_h}\sum_{n=N}^{\infty}  e^{-\frac{\underline{\tau}-\ln L_1-2\varepsilon}{\ln L_1}\lambda_1 ns_h}+\varepsilon_h^{s_h}\sum_{n=N}^{\infty} e^{-(n-1)\lambda_2s_h}.
	\end{split}
\end{equation*}
It implies $$\mathcal{H}^{s_h}(\mathfrak{S}(f,\varphi,\mathcal{Z},\mathbb{N})):=\lim\limits _{N \rightarrow \infty}\mathcal{H}_{\delta_N}^{s_h}(\mathfrak{S}(f,\varphi,\mathcal{Z},\mathbb{N}))=0.$$ So $$s_h\geq \operatorname{dim}_H \mathfrak{S}(f,\varphi,\mathcal{Z},\mathbb{N}).$$
Let $s_h\to 0,$ we obtain (\ref{equa-thm-17}).  \qed

\section{Applications}\label{sec-app}
\subsection{Proof of Theorem \ref{Thm-hyper}}
Given $\varepsilon>0$, for each $x \in \Lambda$, set
$$
V_{\varepsilon}^s(x):=\left\{y : \lim\limits_{n\to+\infty}d(f^n(x), f^n(y))=0 \text{ and  }d(f^n(x), f^n(y))\leq\varepsilon \text { for every } n \geq 0\right\},
$$
and
$$
V_{\varepsilon}^u(x):=\left\{y :  \lim\limits_{n\to-\infty}d(f^n(x), f^n(y))=0 \text{ and  }d(f^n(x), f^n(y))\leq\varepsilon \text { for every } n \leq 0\right\}.
$$
By \cite[Theorem 6.2 and IV.1]{Shub-1987}, there is $\varepsilon^*>0$ such that for any $0<\varepsilon<\varepsilon^*,$ $V_{\varepsilon}^s(x)$ and $V_{\varepsilon}^u(x)$ are embedded disks. Moreover, for any $\lambda_1<\lambda'_1<1$ and $\lambda_2<\lambda'_2<1$ one has
$$
d(f^n(x), f^n(y)) \leq (\lambda'_1)^n d(x, y) \text{ for any } y \in V^s_\varepsilon(x) \text{ and } n\geq 0
$$
and
$$
d(f^{-n}(x), f^{-n}(y)) \leq (\lambda'_2)^n d(x, y) \text{ for any } y \in V^u_\varepsilon(x) \text{ and } n\geq 0.
$$
By \cite[Proposition 7.2]{Shub-1987}, $\Lambda$ has a product structure, that is, for any $0<\varepsilon^{**}<\varepsilon^*$  there exists $\delta^{**}=\delta^{**}(\varepsilon^{**})>0$ such that
$$
\sharp \left(V_{\varepsilon^{**}}^s(x) \cap V_{\varepsilon^{**}}^u(y)\right)=1
$$
whenever $x, y \in \Lambda$ and $d(x, y) \leq \delta^{**}$.
\begin{Prop}\label{Prop-hype}
	If $\Lambda\subset M$ is a uniformly $(\lambda_1,\lambda_2)$-hyperbolic set for  a diffeomorphism $f,$ then  $f:\Lambda\to \Lambda$ is  $(\ln(\lambda'_1)^{-1},\ln(\lambda'_2)^{-1})$-hyperbolic for any $\lambda_1<\lambda'_1<1$ and $\lambda_2<\lambda'_2<1$.
\end{Prop}
\begin{proof}
	Fix $\varepsilon_h>0.$ Denote $\varepsilon^{**}=\min\{\frac{\varepsilon_h}{4},\varepsilon^*\}$ and $\varepsilon=\min\{\delta^{**}(\varepsilon^{**}),\varepsilon^{**}\}$.  Fix $x,y\in \Lambda$ and $n\in\mathbb{N}$  such that $d_n(x,y)<\varepsilon.$ Since $d(x,y)<\varepsilon\leq\delta^{**},$ we have $
	\sharp \left(V_{\varepsilon^{**}}^s(x) \cap V_{\varepsilon^{**}}^u(y)\right)=1.
	$ Denote $V_{\varepsilon^{**}}^s(x) \cap V_{\varepsilon^{**}}^u(y)=\{z\}.$ Then
	$$
	d(f^i(x), f^i(z)) \leq (\lambda'_1)^i d(x, z)\leq  (\lambda'_1)^i \varepsilon^{**} \text{ for any } 0\leq i\leq n-1.
	$$
	Combing this inequality with $d_n(x,y)<\varepsilon^{**}$ we have
	$$
	d(f^i(y), f^i(z)) \leq d(f^i(y), f^i(x)) +d(f^i(x), f^i(z))  \leq \varepsilon^{**}+\varepsilon^{**}<\frac{\varepsilon_h}{2}  \text{ for any } 0\leq i\leq n-1.
	$$
	It implies $f^{n-1}(z)\in V^u_{\frac{\varepsilon_h}{2}}(f^{n-1}(y)),$ and thus
	$$
	d(f^{i}(z), f^{i}(y)) \leq  (\lambda'_2)^{n-1-i} d(f^{n-1}(z), f^{n-1}(y))\leq (\lambda'_2)^{n-1-i} \frac{\varepsilon_h}{2}  \text{ for any } 0\leq i\leq n-1.
	$$
	So
	\begin{equation*}
		\begin{split}
			d(f^{i}(x), f^{i}(y))&\leq d(f^{i}(x), f^{i}(z)+d(f^{i}(z), f^{i}(y)))\\
			&\leq  (\lambda'_1)^i \varepsilon^{**}+(\lambda'_2)^{n-1-i} \frac{\varepsilon_h}{2} \\
			&\leq \varepsilon_h\max \{(\lambda'_1)^i ,(\lambda'_2)^{n-1-i}\}\\
			&=\varepsilon_h e^{-\min\{-i\ln \lambda'_1, -(n-1-i)\ln\lambda'_2\}}\\
			&=\varepsilon_h e^{-\min\{i\ln (\lambda'_1)^{-1}, (n-1-i)\ln(\lambda'_2)^{-1}\}}
		\end{split}
	\end{equation*}
	for any $0\leq i\leq n-1.$
\end{proof}

	By the definition of $L_1$ and $L_2$, $f:\Lambda\to\Lambda$ is a $(L_1,L_2)$-bi-Lipschitz homeomorphism and $\lambda_1^{-1}\leq L_1$. So by Proposition \ref{Prop-hyper-spec}, Corollary \ref{theorem-shrinking-2}, Theorem \ref{Thm-upper-bound} and Remark \ref{Rem-under}, we obtain Theorem \ref{Thm-hyper}. \qed

\subsection{Proof of Theorem \ref{Thm-hyper-torus}}
Given a matrix $B \in S L(d, \mathbb{Z})$, denote $$||B||:=\sup_{x\in \mathbb{R}^d}\frac{|Bx|}{|x|}, m(B):=\inf_{x\in \mathbb{R}^d}\frac{|Bx|}{|x|}.$$ Denote by $\lambda_1, \ldots, \lambda_d$ the eigenvalues of $B$ counted with multiplicity and ordered so that
$$
0<\left|\lambda_1\right| \leq\left|\lambda_2\right| \leq \ldots \leq\left|\lambda_d\right| .
$$
By the spectral radius formula, one has
$$\lim\limits_{m\to\infty}||B^n||^{\frac{1}{n}}=|\lambda_d|,$$
$$\lim\limits_{m\to\infty}(m(B^n))^{\frac{1}{n}}=\lim\limits_{m\to\infty}||B^{-n}||^{-\frac{1}{n}}=|\lambda_1|.$$

	Now we give the proof of Theorem \ref{Thm-hyper-torus}. Assume that the multiplicities of $\lambda_u$ and $\lambda_s$ are $d_u$ and $d_s,$ respectively. Denote
	$$V_u:=\{x\in \mathbb{R}^d:(A-\lambda_uI)^{d_u}x=0\},$$
	$$V_s:=\{x\in \mathbb{R}^d:(A-\lambda_sI)^{d_s}x=0\},$$
	where $I$ is the identity matrix.
	Then we have $ \mathbb{R}^d=V^{u}\oplus V^{s}$ and the restriction of $A$ on $V_u$ and $V_s,$ $A|_{V_u}$ and $A|_{V_s},$ has only one eigenvalue $\lambda_u$ and $\lambda_s,$ respectively. By the spectral radius formula, one has
	$$\lim\limits_{m\to\infty}(m((A|_{V_u})^n))^{\frac{1}{n}}=\lim\limits_{m\to\infty}||(A|_{V_u})^n||^{\frac{1}{n}}=|\lambda_u|,$$
	$$\lim\limits_{m\to\infty}(m((A|_{V_s})^n))^{\frac{1}{n}}=\lim\limits_{m\to\infty}||(A|_{V_s})^n||^{\frac{1}{n}}=|\lambda_s|.$$
	Fix $0<\varepsilon<|\lambda_s|.$ There is $N\in\mathbb{N^{+}}$ such that
	$$(|\lambda_u|-\varepsilon)^N<\frac{|A^Nx|}{|x|}<(|\lambda_u|+\varepsilon)^N, \forall x\in V_u,$$
	$$(|\lambda_s|-\varepsilon)^N<\frac{|A^Nx|}{|x|}<(|\lambda_s|+\varepsilon)^N, \forall x\in V_s.$$
	It follows that $\mathbb{T}^d$ is a
	uniformly $((|\lambda_s|+\varepsilon)^N,(|\lambda_u|-\varepsilon)^{-N})$-hyperbolic set for $f_{A^N}$ and  $f_{A^N}$ is $((|\lambda_s|-\varepsilon)^{-N},(|\lambda_u|+\varepsilon)^N)$-bi-Lipschitz.
	Every  hyperbolic automorphism of $\mathbb{T}^d$ is mixing (see, for example, \cite{BS}). So $f_{A^N}$ is mixing.
	Since  the multiplicities of  $\lambda_u$ and $\lambda_s$ are $d_u$ and $d_s$ respectivelty, then we have $d_s+d_u=d$ and $|\lambda_s|^{d_s}|\lambda_u|^{d_u}=1$. From \cite[Theorem 8.14]{Walters}, one has $h_{top}(f_{A},\mathbb{T}^d)=d_u\ln|\lambda_u|=d_s\ln|\lambda_s|^{-1}$ and thus $h_{top}(f_{A^N},\mathbb{T}^d)=Nh_{top}(f_{A},\mathbb{T}^d)=Nd_s\ln|\lambda_s|^{-1}.$
	
	For each $j\geq 1,$ we can choose $S_j=\{s_j^i\}_{i=1}^{\infty}\subset \mathbb{N}$ and $0\leq I_j\leq N-1$ such that $\lim\limits_{i\to \infty}-\frac{\ln\varphi(s_j^i)}{s_j^i}=\underline{\tau}_j$ and $s_j^i=I_j \mod N$ for any $i\in\mathbb{N}.$ Since $f_A$ is Lipschitz, then   $0<L=\sup\{\frac{d(f_A(x),f_A(y))}{d(x,y)}:x\neq y\in X\}<\infty.$  For each $j\geq 1,$ define
	\begin{equation*}
		\begin{split}
			&\tilde{\mathcal{Z}}_j:=\{\tilde{z}_j^i\}_{i=0}^{\infty}\text{ with }  f^{I_j}_A(\tilde{z}_j^{i})=z_j^{iN+I_j},\\
			&\tilde{S}_j:=\{\tilde{s}_j^i\}_{i=0}^{\infty}  \text{ with } \tilde{s}_j^i=\frac{s_j^i-I_j}{N},\\
		\end{split}
	\end{equation*}
	and
	$$
	\tilde{\varphi}_j(i):=\left\{\begin{array}{ll}
		\frac{\varphi_j(iN+I_j)}{L^{I_j}}, & \text { if } i\in \tilde{S}_j,\\
		1, & \text { if } i\in \mathbb{N}\setminus \tilde{S}_j.
	\end{array}\right.
	$$
	Then we have $\mathfrak{S}(f_{A^N},\tilde{\varphi_j},\tilde{\mathcal{Z}}_j,\tilde{S}_j)\subseteq \mathfrak{S}(f_{A},\varphi_j,\mathcal{Z}_j,\mathbb{N})$ and
	\begin{equation*}
		\begin{split}
			\overline{\tau}_j(\tilde{\varphi}_j):&=\limsup\limits_{i\to\infty}-\frac{\ln\tilde{\varphi}_j(i)}{i}=\limsup\limits_{i\to\infty}-\frac{\ln\tilde{\varphi}_j(\tilde{s}_j^i)}{\tilde{s}_j^i}\\
			&=N\lim\limits_{i\to\infty}-\frac{\ln\varphi_j(s_j^i)}{s_j^i}  = N\underline{\tau}_j.
		\end{split}
	\end{equation*}
	Denote $\tilde{\overline{\tau}}:=\sup\limits_{j\geq 1}\overline{\tau}_j(\tilde{\varphi}_j)=\sup\limits_{j\geq 1}\underline{\tau}_j= N\underline{\tau}.$

	On the other hand, for each $j\geq 1$ and $0\leq n\leq N-1,$ define
	\begin{equation*}
		\begin{split}
			&\mathcal{Z}_{j,n}:=\{z_{j,n}^i\}_{i=0}^{\infty}\text{ with } z_{j,n}^{i+1}=f^{N-n}_A(z_{j}^{iN+n}),\\
			&\varphi_{j,n}(i+1):=\varphi_j(iN+n)L^{N-n}.
		\end{split}
	\end{equation*}
	Then we have $\mathfrak{S}(f_{A},\varphi_j,\mathcal{Z}_j,\mathbb{N})\subseteq \cup_{n=0}^{N-1}\mathfrak{S}(f_{A},\varphi_j,\mathcal{Z}_j,N\mathbb{N}+n)\subseteq \cup_{n=0}^{N-1}\mathfrak{S}(f_{A^N},\varphi_{j,n},\mathcal{Z}_{j,n},\mathbb{N})$ and
	\begin{equation}\label{equa-RT}
		\begin{split}
			\underline{\tau}_{j,n}(\varphi_{j,n}):&=\liminf\limits_{i\to\infty}-\frac{\ln\varphi_{j,n}(i)}{i}=\liminf\limits_{i\to\infty}-\frac{\ln\varphi_j(iN+n)}{i}\\
			&=N\liminf\limits_{i\to\infty}-\frac{\ln\varphi_j(iN+n)}{iN+n}  \geq N\underline{\tau}_j,
		\end{split}
	\end{equation}
	where $N\mathbb{N}+n:=\{mN+n:m\in\mathbb{N}\}.$

	\bigskip
	\noindent{\bf Upper bounds}
	
	If  $\underline{\tau}_{j,n}(\varphi_{j,n})\leq \ln(|\lambda_s|-\varepsilon)^{-N},$ by Theorem \ref{Thm-upper-bound}(2.1)(2.2), Proposition \ref{Prop-hype} and (\ref{equa-RT})  we have
	\begin{equation*}
		\begin{split}
			h_{top}(f_{A^N},\mathfrak{S}(f_{A^N},\varphi_{j,n},\mathcal{Z}_{j,n},\mathbb{N}))&\leq \frac{\ln(|\lambda_s|-\varepsilon)^{-N}\ln(|\lambda_u|+\varepsilon)^{N}-\underline{\tau}_{j,n}\ln(|\lambda_u|+\varepsilon)^{N} }{\ln(|\lambda_s|-\varepsilon)^{-N}\ln(|\lambda_u|+\varepsilon)^{N}+\underline{\tau}_{j,n}\ln(|\lambda_s|-\varepsilon)^{-N}}h_{top}(f_{A^N},\mathbb{T}^d)\\
			& \leq  \frac{\ln(|\lambda_s|-\varepsilon)^{-1}\ln(|\lambda_u|+\varepsilon)-\underline{\tau}_j\ln(|\lambda_u|+\varepsilon) }{\ln(|\lambda_s|-\varepsilon)^{-1}\ln(|\lambda_u|+\varepsilon)+\underline{\tau}_j\ln(|\lambda_s|-\varepsilon)^{-1}}Nd_s\ln|\lambda_s|^{-1}
		\end{split}
	\end{equation*}
	and
	\begin{equation*}
		\begin{split}
			& \operatorname{dim}_H \mathfrak{S}(f_{A^N},\varphi_{j,n},\mathcal{Z}_{j,n},\mathbb{N})\\
			\leq  &(\frac{1}{\ln (|\lambda_s|+\varepsilon)^{-N}}+\frac{1}{\ln (|\lambda_u|-\varepsilon)^N}\frac{\ln(|\lambda_s|-\varepsilon)^{-N}\ln(|\lambda_u|+\varepsilon)^{N}-\underline{\tau}_{j,n}\ln(|\lambda_u|+\varepsilon)^{N} }{\ln(|\lambda_s|-\varepsilon)^{-N}\ln(|\lambda_u|+\varepsilon)^{N}+\underline{\tau}_{j,n}\ln(|\lambda_s|-\varepsilon)^{-N}})h_{top}(f,X)\\
			\leq &(\frac{1}{\ln (|\lambda_s|+\varepsilon)^{-1}}+\frac{1}{\ln (|\lambda_u|-\varepsilon)}\frac{\ln(|\lambda_s|-\varepsilon)^{-1}\ln(|\lambda_u|+\varepsilon)-\underline{\tau}_j\ln(|\lambda_u|+\varepsilon) }{\ln(|\lambda_s|-\varepsilon)^{-1}\ln(|\lambda_u|+\varepsilon)+\underline{\tau}_j\ln(|\lambda_s|-\varepsilon)^{-1}})d_s\ln|\lambda_s|^{-1}.
		\end{split}
	\end{equation*}
	
	If  $\underline{\tau}_{j,n}(\varphi_{j,n})> \ln(|\lambda_s|-\varepsilon)^{-N},$ by Theorem \ref{Thm-upper-bound}(2.3)  we have
	\begin{equation}\label{equa-RO}
		\begin{split}
			h_{top}(f_{A^N},\mathfrak{S}(f_{A^N},\varphi_{j,n},\mathcal{Z}_{j,n},\mathbb{N}))=0,
		\end{split}
	\end{equation}
	\begin{equation}\label{equa-RP}
		\operatorname{dim}_H \mathfrak{S}(f_{A^N},\varphi_{j,n},\mathcal{Z}_{j,n},\mathbb{N})=0.
	\end{equation}
	
	It follows that
	\begin{equation*}
		\begin{split}
			h_{top}(f_{A},\mathfrak{S}(f_{A},\varphi_j,\mathcal{Z}_j,\mathbb{N}))&\leq \max_{0\leq n\leq N-1 } h_{top}(f_{A},\mathfrak{S}(f_{A^N},\varphi_{j,n},\mathcal{Z}_{j,n},\mathbb{N}))\\
			&= \max_{0\leq n\leq N-1 } \frac{1}{N} h_{top}(f_{A^N},\mathfrak{S}(f_{A^N},\varphi_{j,n},\mathcal{Z}_{j,n},\mathbb{N})) \\
			& \leq \frac{\ln(|\lambda_s|-\varepsilon)^{-1}\ln(|\lambda_u|+\varepsilon)-\underline{\tau}_j\ln(|\lambda_u|+\varepsilon) }{\ln(|\lambda_s|-\varepsilon)^{-1}\ln(|\lambda_u|+\varepsilon)+\underline{\tau}_j\ln(|\lambda_s|-\varepsilon)^{-1}}d_s\ln|\lambda_s|^{-1}
		\end{split}
	\end{equation*}
	and
	\begin{equation*}
		\begin{split}
			&\operatorname{dim}_H\mathfrak{S}(f_{A},\varphi_j,\mathcal{Z}_j,\mathbb{N})\\
			\leq& \max_{0\leq n\leq N-1 } \operatorname{dim}_H\mathfrak{S}(f_{A^N},\varphi_{j,n},\mathcal{Z}_{j,n},\mathbb{N})\\
			\leq &(\frac{1}{\ln (|\lambda_s|+\varepsilon)^{-1}}+\frac{1}{\ln (|\lambda_u|-\varepsilon)}\frac{\ln(|\lambda_s|-\varepsilon)^{-1}\ln(|\lambda_u|+\varepsilon)-\underline{\tau}_j\ln(|\lambda_u|+\varepsilon) }{\ln(|\lambda_s|-\varepsilon)^{-1}\ln(|\lambda_u|+\varepsilon)+\underline{\tau}_j\ln(|\lambda_s|-\varepsilon)^{-1}})d_s\ln|\lambda_s|^{-1}.
		\end{split}
	\end{equation*}
	Thus we have
	$$h_{top}(f_A,\cap_{j=1}^{\infty}\mathfrak{S}(f,\varphi_j,\mathcal{Z}_j,\mathbb{N})) \leq \frac{\ln(|\lambda_s|-\varepsilon)^{-1}\ln(|\lambda_u|+\varepsilon)-\underline{\tau}\ln(|\lambda_u|+\varepsilon) }{\ln(|\lambda_s|-\varepsilon)^{-1}\ln(|\lambda_u|+\varepsilon)+\underline{\tau}\ln(|\lambda_s|-\varepsilon)^{-1}}d_s\ln|\lambda_s|^{-1}$$
	and
	\begin{equation*}
		\begin{split}
			&\operatorname{dim}_H \cap_{j=1}^{\infty}\mathfrak{S}(f_{A},\varphi_j,\mathcal{Z}_j,\mathbb{N})\\
			\leq &(\frac{1}{\ln (|\lambda_s|+\varepsilon)^{-1}}+\frac{1}{\ln (|\lambda_u|-\varepsilon)}\frac{\ln(|\lambda_s|-\varepsilon)^{-1}\ln(|\lambda_u|+\varepsilon)-\underline{\tau}\ln(|\lambda_u|+\varepsilon) }{\ln(|\lambda_s|-\varepsilon)^{-1}\ln(|\lambda_u|+\varepsilon)+\underline{\tau}\ln(|\lambda_s|-\varepsilon)^{-1}})d_s\ln|\lambda_s|^{-1}.
		\end{split}
	\end{equation*}
	Let $\varepsilon\to 0,$ we obtain
	\begin{equation}\label{equa-RQ}
		h_{top}(f_A,\cap_{j=1}^{\infty}\mathfrak{S}(f_A,\varphi_j,\mathcal{Z}_j,\mathbb{N}))\leq d_s\frac{\ln |\lambda_s|^{-1}\ln |\lambda_u|-\underline{\tau}\ln |\lambda_u|}{\ln |\lambda_u|+\underline{\tau}}.
	\end{equation}
	\begin{equation}\label{equa-RU}
		\begin{split}
			\operatorname{dim}_H\cap_{j=1}^{\infty}\mathfrak{S}(f_A,\varphi_j,\mathcal{Z}_j,\mathbb{N})\leq d_s\frac{\ln |\lambda_s|^{-1}+\ln |\lambda_u|}{\ln |\lambda_u|+\underline{\tau}}.
		\end{split}
	\end{equation}
	In particular, when $\underline{\tau}=-\ln|\lambda_s|$, we have
	\begin{equation}\label{equa-RV}
		h_{top}(f_A,\cap_{j=1}^{\infty}\mathfrak{S}(f_A,\varphi_j,\mathcal{Z}_j,\mathbb{N}))=0,
	\end{equation}
	\begin{equation}\label{equa-RW}
		\operatorname{dim}_H\cap_{j=1}^{\infty}\mathfrak{S}(f_A,\varphi_j,\mathcal{Z}_j,\mathbb{N})\leq d_s.
	\end{equation}
	
	When $\underline{\tau}>-\ln|\lambda_s|$, there is $j_0\in\mathbb{N}$ such that $\underline{\tau}_{j_0}>-\ln|\lambda_s|.$
	We require that $\varepsilon>0$ is small enough such that $$\underline{\tau}_{j_0}>-\ln(|\lambda_s|-\varepsilon).$$ Then for any $0\leq n\leq N-1$  we have  $\underline{\tau}_{j_0,n}(\varphi_{j_0,n})\geq N\underline{\tau}_{j_0}>-\ln(|\lambda_s|-\varepsilon)^{N}.$ So by (\ref{equa-RO}) and (\ref{equa-RP})  for any $0\leq n\leq N-1$ we have
	 \begin{equation*}
		\begin{split}
			h_{top}(f_{A^N},\mathfrak{S}(f_{A^N},\varphi_{j_0,n},\mathcal{Z}_{j_0,n},\mathbb{N}))=0,
		\end{split}
	\end{equation*}
	\begin{equation*}
		\operatorname{dim}_H \mathfrak{S}(f_{A^N},\varphi_{j_0,n},\mathcal{Z}_{j_0,n},\mathbb{N})=0.
	\end{equation*}
	 And thus
 	\begin{equation}\label{equa-RX}
 		h_{top}(f_A,\cap_{j=1}^{\infty}\mathfrak{S}(f_A,\varphi_j,\mathcal{Z}_j,\mathbb{N}))=0,
 	\end{equation}
 	\begin{equation}\label{equa-RY}
 		\operatorname{dim}_H\cap_{j=1}^{\infty}\mathfrak{S}(f_A,\varphi_j,\mathcal{Z}_j,\mathbb{N})=0.
 	\end{equation}

	\bigskip
	\noindent{\bf Lower bounds}
	
	When  $\underline{\tau}<-\ln|\lambda_s|$, we require that $\varepsilon>0$ is small enough such that
	\begin{equation*}
		\underline{\tau}<-\ln(|\lambda_s|+\varepsilon).
	\end{equation*}
	Then we have  $\tilde{\overline{\tau}}= N\underline{\tau}<-\ln(|\lambda_s|+\varepsilon)^{N}.$
	Apply Theorem \ref{Thm-hyper}(1) to $(\mathbb{T}^d,f_{A^N}),$ the set
	$\cap_{j=1}^{\infty}\mathfrak{S}(f_{A^N},\tilde{\varphi_j},\tilde{\mathcal{Z}}_j,\tilde{S}_j)$ is non-empty and
	\begin{equation*}
		\begin{split}
			h_{top}(f_{A^N},\cap_{j=1}^{\infty}\mathfrak{S}(f_{A^N},\tilde{\varphi_j},\tilde{\mathcal{Z}}_j,\tilde{S}_j)&\geq \frac{\ln(|\lambda_s|+\varepsilon)^{-N}\ln(|\lambda_u|-\varepsilon)^{N}-\tilde{\overline{\tau}}\ln(|\lambda_u|-\varepsilon)^{N} }{\ln(|\lambda_s|+\varepsilon)^{-N}\ln(|\lambda_u|-\varepsilon)^{N}+\tilde{\overline{\tau}}\ln(|\lambda_s|+\varepsilon)^{-N}}h_{top}(f_{A^N},\mathbb{T}^d)\\
			& =\frac{\ln(|\lambda_s|+\varepsilon)^{-1}\ln(|\lambda_u|-\varepsilon)-\underline{\tau}\ln(|\lambda_u|-\varepsilon) }{\ln(|\lambda_s|+\varepsilon)^{-1}\ln(|\lambda_u|-\varepsilon)+\underline{\tau}\ln(|\lambda_s|+\varepsilon)^{-1}}Nd_s\ln|\lambda_s|^{-1}
		\end{split}
	\end{equation*}
	and
		\begin{equation*}
		\begin{split}
			& \operatorname{dim}_H \cap_{j=1}^{\infty}\mathfrak{S}(f_{A^N},\tilde{\varphi_j},\tilde{\mathcal{Z}}_j,\tilde{S}_j)\\
			\geq &(\frac{1}{\ln (|\lambda_s|-\varepsilon)^{-N}}+\frac{1}{\ln (|\lambda_u|+\varepsilon)^N}\frac{\ln(|\lambda_s|+\varepsilon)^{-N}\ln(|\lambda_u|-\varepsilon)^{N}-\tilde{\overline{\tau}}\ln(|\lambda_u|-\varepsilon)^{N} }{\ln(|\lambda_s|+\varepsilon)^{-N}\ln(|\lambda_u|-\varepsilon)^{N}+\tilde{\overline{\tau}}\ln(|\lambda_s|+\varepsilon)^{-N}})h_{top}(f,X)\\
			=&(\frac{1}{\ln (|\lambda_s|-\varepsilon)^{-1}}+\frac{1}{\ln (|\lambda_u|+\varepsilon)}\frac{\ln(|\lambda_s|+\varepsilon)^{-1}\ln(|\lambda_u|-\varepsilon)-\underline{\tau}\ln(|\lambda_u|-\varepsilon) }{\ln(|\lambda_s|+\varepsilon)^{-1}\ln(|\lambda_u|-\varepsilon)+\underline{\tau}\ln(|\lambda_s|+\varepsilon)^{-1}})d_s\ln|\lambda_s|^{-1}.
		\end{split}
	\end{equation*}
	It follows that
	\begin{equation*}
		\begin{split}
			h_{top}(f_A,\cap_{j=1}^{\infty}\mathfrak{S}(f_A,\varphi_j,\mathcal{Z}_j,\mathbb{N}))&\geq h_{top}(f_{A},\cap_{j=1}^{\infty}\mathfrak{S}(f_{A^N},\tilde{\varphi_j},\tilde{\mathcal{Z}}_j,\tilde{S}_j))\\
			&= \frac{1}{N}h_{top}(f_{A^N},\cap_{j=1}^{\infty}\mathfrak{S}(f_{A^N},\tilde{\varphi_j},\tilde{\mathcal{Z}}_j,\tilde{S}_j)\\
			&\geq\frac{\ln(|\lambda_s|+\varepsilon)^{-1}\ln(|\lambda_u|-\varepsilon)-\underline{\tau}\ln(|\lambda_u|-\varepsilon) }{\ln(|\lambda_s|+\varepsilon)^{-1}\ln(|\lambda_u|-\varepsilon)+\underline{\tau}\ln(|\lambda_s|+\varepsilon)^{-1}}d_s\ln|\lambda_s|^{-1}
		\end{split}
	\end{equation*}
	and
	\begin{equation*}
		\begin{split}
			&\operatorname{dim}_H\cap_{j=1}^{\infty}\mathfrak{S}(f_A,\varphi_j,\mathcal{Z}_j,\mathbb{N})\\
			\geq &\operatorname{dim}_H \cap_{j=1}^{\infty}\mathfrak{S}(f_{A^N},\tilde{\varphi_j},\tilde{\mathcal{Z}}_j,\tilde{S}_j)\\
			\geq&(\frac{1}{\ln (|\lambda_s|-\varepsilon)^{-1}}+\frac{1}{\ln (|\lambda_u|+\varepsilon)}\frac{\ln(|\lambda_s|+\varepsilon)^{-1}\ln(|\lambda_u|-\varepsilon)-\underline{\tau}\ln(|\lambda_u|-\varepsilon) }{\ln(|\lambda_s|+\varepsilon)^{-1}\ln(|\lambda_u|-\varepsilon)+\underline{\tau}\ln(|\lambda_s|+\varepsilon)^{-1}})d_s\ln|\lambda_s|^{-1}.
		\end{split}
	\end{equation*}
	Let $\varepsilon\to 0,$ we obtain
	\begin{equation}\label{equa-RR}
		h_{top}(f_A,\cap_{j=1}^{\infty}\mathfrak{S}(f_A,\varphi_j,\mathcal{Z}_j,\mathbb{N}))\geq d_s\frac{\ln |\lambda_s|^{-1}\ln |\lambda_u|-\underline{\tau}\ln |\lambda_u|}{\ln |\lambda_u|+\underline{\tau}},
	\end{equation}
	\begin{equation}\label{equa-RS}
		\operatorname{dim}_H\cap_{j=1}^{\infty}\mathfrak{S}(f_A,\varphi_j,\mathcal{Z}_j,\mathbb{N})\geq d_s\frac{\ln |\lambda_s|^{-1}+\ln |\lambda_u|}{\ln |\lambda_u|+\underline{\tau}}.
	\end{equation}

	When  $\underline{\tau}<-\ln|\lambda_s|,$ by (\ref{equa-RQ}), (\ref{equa-RU}), (\ref{equa-RR}) and (\ref{equa-RS}), we have $$h_{top}(f_A,\cap_{j=1}^{\infty}\mathfrak{S}(f_A,\varphi_j,\mathcal{Z}_j,\mathbb{N}))= d_s\frac{\ln |\lambda_s|^{-1}\ln |\lambda_u|-\underline{\tau}\ln |\lambda_u|}{\ln |\lambda_u|+\underline{\tau}}.$$
	\begin{equation*}
		\begin{split}
			\operatorname{dim}_H\cap_{j=1}^{\infty}\mathfrak{S}(f_A,\varphi_j,\mathcal{Z}_j,\mathbb{N})= d_s\frac{\ln |\lambda_s|^{-1}+\ln |\lambda_u|}{\ln |\lambda_u|+\underline{\tau}}.
		\end{split}
			\end{equation*}
		So we obtain item(1). By (\ref{equa-RV}) and (\ref{equa-RW}), we obtain item(2). By (\ref{equa-RX}) and (\ref{equa-RY}), we obtain item(3).  \qed
	
\subsection{Proof of Theorem \ref{Thm-exp-torus}}
	By the spectral radius formula, one has
	$$\lim\limits_{m\to\infty}||A^n||^{\frac{1}{n}}=|\lambda_d|,$$
	$$\lim\limits_{m\to\infty}(m(A^n))^{\frac{1}{n}}=\lim\limits_{m\to\infty}||A^{-n}||^{-\frac{1}{n}}=|\lambda_1|.$$
	Then there is $N\in\mathbb{N^{+}}$ such that
	$$(|\lambda_1|-\varepsilon)^N<\frac{|A^Nx|}{|x|}<(|\lambda_d|+\varepsilon)^N, \forall x\in \mathbb{R}^d.$$
	It follows that $f_{A^N}$ is a
	$(|\lambda_1|-\varepsilon)^N$-expanding map and  $f_{A^N}$ is $(|\lambda_d|+\varepsilon)^N$-Lipschitz.
	Since $\mathbb{T}^d$ is connected, then  $f_{A^N}:\mathbb{T}^d \rightarrow \mathbb{T}^d$   is mixing (see \cite[Corollary 11.2.16]{VO-2016}). From \cite[Theorem 8.14]{Walters}, we have $h_{top}(f_{A},\mathbb{T}^d)=\sum_{i=1}^{d}\ln|\lambda_i|$ and thus $h_{top}(f_{A^N},\mathbb{T}^d)=Nh_{top}(f_{A},\mathbb{T}^d)=N\sum_{i=1}^{d}\ln|\lambda_i|.$

	For each $j\geq 1,$ we can choose $S_j=\{s_j^i\}_{i=1}^{\infty}\subset \mathbb{N}$ and $0\leq I_j\leq N-1$ such that $\lim\limits_{i\to \infty}-\frac{\ln\varphi(s_j^i)}{s_j^i}=\underline{\tau}_j$ and $s_j^i=I_j \mod N$ for any $i\in\mathbb{N}.$ Since $f_A$ is Lipschitz, then   $0<L=\sup\{\frac{d(f_A(x),f_A(y))}{d(x,y)}:x\neq y\in X\}<\infty.$  For each $j\geq 1,$ define
	\begin{equation*}
		\begin{split}
			&\tilde{\mathcal{Z}}_j:=\{\tilde{z}_j^i\}_{i=0}^{\infty}\text{ with }  f^{I_j}_A(\tilde{z}_j^{i})=z_j^{iN+I_j},\\
			&\tilde{S}_j:=\{\tilde{s}_j^i\}_{i=0}^{\infty}  \text{ with } \tilde{s}_j^i=\frac{s_j^i-I_j}{N},\\
		\end{split}
	\end{equation*}
	and
	$$
	\tilde{\varphi}_j(i):=\left\{\begin{array}{ll}
		\frac{\varphi_j(iN+I_j)}{L^{I_j}}, & \text { if } i\in \tilde{S}_j,\\
		1, & \text { if } i\in \mathbb{N}\setminus \tilde{S}_j.
	\end{array}\right.
	$$
	Then we have $\mathfrak{S}(f_{A^N},\tilde{\varphi_j},\tilde{\mathcal{Z}}_j,\tilde{S}_j)\subseteq \mathfrak{S}(f_{A},\varphi_j,\mathcal{Z}_j,\mathbb{N})$ and
	\begin{equation*}
		\begin{split}
			\overline{\tau}_j(\tilde{\varphi}_j):&=\limsup\limits_{i\to\infty}-\frac{\ln\tilde{\varphi}_j(i)}{i}=\limsup\limits_{i\to\infty}-\frac{\ln\tilde{\varphi}_j(\tilde{s}_j^i)}{\tilde{s}_j^i}\\
			&=N\lim\limits_{i\to\infty}-\frac{\ln\varphi_j(s_j^i)}{s_j^i}  = N\underline{\tau}_j.
		\end{split}
	\end{equation*}
	Denote $\tilde{\overline{\tau}}:=\sup\limits_{j\geq 1}\overline{\tau}_j(\tilde{\varphi}_j)=\sup\limits_{j\geq 1}\underline{\tau}_j= N\underline{\tau}.$
	
	Apply Theorem \ref{Thm-exp} to $(\mathbb{T}^d,f_{A^N}),$ the set
	$\cap_{j=1}^{\infty}\mathfrak{S}(f_{A^N},\tilde{\varphi_j},\tilde{\mathcal{Z}}_j,\tilde{S}_j)$ is non-empty and
	\begin{equation*}
		\begin{split}
			h_{top}(f_{A^N},\cap_{j=1}^{\infty}\mathfrak{S}(f_{A^N},\tilde{\varphi_j},\tilde{\mathcal{Z}}_j,\tilde{S}_j)\geq \frac{\ln(|\lambda_1|-\varepsilon)^N }{\ln(|\lambda_1|-\varepsilon)^N+\tilde{\overline{\tau}}}h_{top}(f_{A^N},\mathbb{T}^d)=\frac{\ln(|\lambda_1|-\varepsilon) }{\ln(|\lambda_1|-\varepsilon)+\underline{\tau}}N\sum_{i=1}^{d}\ln|\lambda_i|.
		\end{split}
	\end{equation*}
	and
	\begin{equation*}
		\begin{split}
			 \operatorname{dim}_H \cap_{j=1}^{\infty}\mathfrak{S}(f_{A^N},\tilde{\varphi_j},\tilde{\mathcal{Z}}_j,\tilde{S}_j)
			&\geq \frac{1}{\ln (|\lambda_d|+\varepsilon)^N}\frac{\ln(|\lambda_1|-\varepsilon)^N }{\ln(|\lambda_1|-\varepsilon)^N+\tilde{\overline{\tau}}}h_{top}(f_{A^N},\mathbb{T}^d)\\
			&= \frac{1}{\ln (|\lambda_d|+\varepsilon)}\frac{\ln(|\lambda_1|-\varepsilon) }{\ln(|\lambda_1|-\varepsilon)+\underline{\tau}}\sum_{i=1}^{d}\ln|\lambda_i|.
		\end{split}
	\end{equation*}
		It follows that
	\begin{equation*}
		\begin{split}
			h_{top}(f_A,\cap_{j=1}^{\infty}\mathfrak{S}(f_A,\varphi_j,\mathcal{Z}_j,\mathbb{N}))&\geq h_{top}(f_{A},\cap_{j=1}^{\infty}\mathfrak{S}(f_{A^N},\tilde{\varphi_j},\tilde{\mathcal{Z}}_j,\tilde{S}_j))\\
			&= \frac{1}{N}h_{top}(f_{A^N},\cap_{j=1}^{\infty}\mathfrak{S}(f_{A^N},\tilde{\varphi_j},\tilde{\mathcal{Z}}_j,\tilde{S}_j)\\
			&\geq\frac{\ln(|\lambda_1|-\varepsilon) }{\ln(|\lambda_1|-\varepsilon)+\underline{\tau}}\sum_{i=1}^{d}\ln|\lambda_i|
		\end{split}
	\end{equation*}
	and
	\begin{equation*}
		\begin{split}
			\operatorname{dim}_H\cap_{j=1}^{\infty}\mathfrak{S}(f_A,\varphi_j,\mathcal{Z}_j,\mathbb{N})
			&\geq \operatorname{dim}_H \cap_{j=1}^{\infty}\mathfrak{S}(f_{A^N},\tilde{\varphi_j},\tilde{\mathcal{Z}}_j,\tilde{S}_j)\\
			&\geq \frac{1}{\ln (|\lambda_d|+\varepsilon)}\frac{\ln(|\lambda_1|-\varepsilon) }{\ln(|\lambda_1|-\varepsilon)+\underline{\tau}}\sum_{i=1}^{d}\ln|\lambda_i|.
		\end{split}
	\end{equation*}
	Let $\varepsilon\to 0,$ we obtain
	\begin{equation}\label{equa-SR}
		h_{top}(f_A,\cap_{j=1}^{\infty}\mathfrak{S}(f_A,\varphi_j,\mathcal{Z}_j,\mathbb{N}))\geq \frac{\ln|\lambda_1| }{\ln|\lambda_1|+\underline{\tau}}\sum_{i=1}^{d}\ln|\lambda_i|,
	\end{equation}
	\begin{equation}\label{equa-SS}
		\operatorname{dim}_H\cap_{j=1}^{\infty}\mathfrak{S}(f_A,\varphi_j,\mathcal{Z}_j,\mathbb{N})\geq \frac{1}{\ln |\lambda_d|}\frac{\ln|\lambda_1| }{\ln|\lambda_1|+\underline{\tau}}\sum_{i=1}^{d}\ln|\lambda_i|.
	\end{equation}

	On the other hand, for each $j\geq 1$ and $0\leq n\leq N-1,$ define
	\begin{equation*}
		\begin{split}
			&\mathcal{Z}_{j,n}:=\{z_{j,n}^i\}_{i=0}^{\infty}\text{ with } z_{j,n}^{i+1}=f^{N-n}_A(z_{j}^{iN+n}),\\
			&\varphi_{j,n}(i+1):=\varphi_j(iN+n)L^{N-n}.
		\end{split}
	\end{equation*}
	Then we have $\mathfrak{S}(f_{A},\varphi_j,\mathcal{Z}_j,\mathbb{N})\subseteq \cup_{n=0}^{N-1}\mathfrak{S}(f_{A},\varphi_j,\mathcal{Z}_j,N\mathbb{N}+n)\subseteq \cup_{n=0}^{N-1}\mathfrak{S}(f_{A^N},\varphi_{j,n},\mathcal{Z}_{j,n},\mathbb{N})$ and
	\begin{equation*}
		\begin{split}
			\underline{\tau}_{j,n}(\varphi_{j,n}):&=\liminf\limits_{i\to\infty}-\frac{\ln\varphi_{j,n}(i)}{i}=\liminf\limits_{i\to\infty}-\frac{\ln\varphi_j(iN+n)}{i}\\
			&=N\liminf\limits_{i\to\infty}-\frac{\ln\varphi_j(iN+n)}{iN+n}  \geq N\underline{\tau}_j.
		\end{split}
	\end{equation*}
	By Theorem \ref{Thm-upper-bound}(1), we have
	\begin{equation*}
		\begin{split}
			h_{top}(f_{A^N},\mathfrak{S}(f_{A^N},\varphi_{j,n},\mathcal{Z}_{j,n},\mathbb{N}))\leq \frac{\ln(|\lambda_d|+\varepsilon)^N }{\ln(|\lambda_d|+\varepsilon)^N+\underline{\tau}_{j,n}}h_{top}(f_{A^N},\mathbb{T}^d)\leq \frac{\ln(|\lambda_d|+\varepsilon) }{\ln(|\lambda_d|+\varepsilon)+\underline{\tau}_{j}}N\sum_{i=1}^{d}\ln|\lambda_i|.
		\end{split}
	\end{equation*}
	and
	\begin{equation*}
		\begin{split}
			\operatorname{dim}_H \mathfrak{S}(f_{A^N},\varphi_{j,n},\mathcal{Z}_{j,n},\mathbb{N})
			&\leq \frac{1}{\ln (|\lambda_1|-\varepsilon)^N}\frac{\ln(|\lambda_d|+\varepsilon)^N }{\ln(|\lambda_d|+\varepsilon)^N+\underline{\tau}_{j,n}}h_{top}(f_{A^N},\mathbb{T}^d)\\
			&\leq  \frac{1}{\ln (|\lambda_1|-\varepsilon)}\frac{\ln(|\lambda_d|+\varepsilon) }{\ln(|\lambda_d|+\varepsilon)+\underline{\tau}_j}\sum_{i=1}^{d}\ln|\lambda_i|.
		\end{split}
	\end{equation*}
	It follows that
	\begin{equation*}
		\begin{split}
			h_{top}(f_{A},\mathfrak{S}(f_{A},\varphi_j,\mathcal{Z}_j,\mathbb{N}))&\leq \max_{0\leq n\leq N-1 } h_{top}(f_{A},\mathfrak{S}(f_{A^N},\varphi_{j,n},\mathcal{Z}_{j,n},\mathbb{N}))\\
			&= \max_{0\leq n\leq N-1 } \frac{1}{N} h_{top}(f_{A^N},\mathfrak{S}(f_{A^N},\varphi_{j,n},\mathcal{Z}_{j,n},\mathbb{N})) \\
			& \leq\frac{\ln(|\lambda_d|+\varepsilon) }{\ln(|\lambda_d|+\varepsilon)+\underline{\tau}_{j}}\sum_{i=1}^{d}\ln|\lambda_i|,
		\end{split}
	\end{equation*}
	and
	\begin{equation*}
		\begin{split}
			\operatorname{dim}_H\mathfrak{S}(f_{A},\varphi_j,\mathcal{Z}_j,\mathbb{N})
			&\leq \max_{0\leq n\leq N-1 } \operatorname{dim}_H\mathfrak{S}(f_{A^N},\varphi_{j,n},\mathcal{Z}_{j,n},\mathbb{N})\\
			&\leq \frac{1}{\ln (|\lambda_1|-\varepsilon)}\frac{\ln(|\lambda_d|+\varepsilon) }{\ln(|\lambda_d|+\varepsilon)+\underline{\tau}_j}\sum_{i=1}^{d}\ln|\lambda_i|.
		\end{split}
	\end{equation*}
	Thus we have
	$$h_{top}(f,\cap_{j=1}^{\infty}\mathfrak{S}(f,\varphi_j,\mathcal{Z}_j,\mathbb{N})) \leq \frac{\ln(|\lambda_d|+\varepsilon) }{\ln(|\lambda_d|+\varepsilon)+\underline{\tau}}\sum_{i=1}^{d}\ln|\lambda_i|,$$
	and
	\begin{equation*}
		\begin{split}
			\operatorname{dim}_H \cap_{j=1}^{\infty}\mathfrak{S}(f_{A},\varphi_j,\mathcal{Z}_j,\mathbb{N})\leq  \frac{1}{\ln (|\lambda_1|-\varepsilon)}\frac{\ln(|\lambda_d|+\varepsilon) }{\ln(|\lambda_d|+\varepsilon)+\underline{\tau}}\sum_{i=1}^{d}\ln|\lambda_i|.
		\end{split}
	\end{equation*}
	Let $\varepsilon\to 0,$ we obtain
	\begin{equation}\label{equa-SQ}
		h_{top}(f,\cap_{j=1}^{\infty}\mathfrak{S}(f,\varphi_j,\mathcal{Z}_j,\mathbb{N}))\leq \frac{\ln|\lambda_d| }{\ln|\lambda_d|+\underline{\tau}}\sum_{i=1}^{d}\ln|\lambda_i|,
	\end{equation}
	\begin{equation}\label{equa-SU}
		\begin{split}
			\operatorname{dim}_H\cap_{j=1}^{\infty}\mathfrak{S}(f,\varphi_j,\mathcal{Z}_j,\mathbb{N})\leq \frac{1}{\ln |\lambda_1|}\frac{\ln|\lambda_d| }{\ln|\lambda_d|+\underline{\tau}}\sum_{i=1}^{d}\ln|\lambda_i|.
		\end{split}
	\end{equation}
	By (\ref{equa-SR}), (\ref{equa-SS}),  (\ref{equa-SQ}) and  (\ref{equa-SU}), we complete the proof of Theorem \ref{Thm-exp-torus}.
	\qed

%
\bigskip

\textbf{Acknowledgements.}
We are grateful to  M. Rams for his comments. X. Hou and X. Tian are supported by the National Natural Science Foundation of China (grant No. 12071082) and in part by Shanghai Science and Technology Research Program (grant No. 21JC1400700). X. Hou is also funded by China Postdoctoral Science Foundation No. 2023M740713 and supported by the Postdoctoral Fellowship Program of CPSF under Grant Number GZB20240167.
Y. Zhang is partially supported by National Natural Science Foundation of China (grant Nos. 12161141002, 12271432), and Guangdong Basic and applied Basic Research Foundation (grant No. 2024A1515010974).

\end{document}